\documentclass[11pt,a4paper]{article}
\usepackage{amssymb,amsmath,amsfonts,mathrsfs,bm,enumerate}
\allowdisplaybreaks[1]
\numberwithin{equation}{section}

\usepackage[colorlinks=true, pdfstartview=FitV, linkcolor=blue, citecolor=red, urlcolor=blue,pagebackref=false]{hyperref}

\usepackage{tikz}
\usepackage{float}
\usepackage[font=footnotesize]{caption}

\usepackage{tocloft}

\usepackage{color} 
\usepackage{times,theorem,latexsym,color,comment}

\newcommand{\BOX}{\ensuremath\Box}

\newtheorem{theorema}{Theorem}

\newtheorem{theorem}{Theorem }[section]

\newtheorem{corollary}[theorem]{Corollary}
\newtheorem{definition}[theorem]{Definition}
{\theorembodyfont{\rmfamily}}
{\theorembodyfont{\rmfamily}}
{\theorembodyfont{\rmfamily}}
\newtheorem{lemma}[theorem]{Lemma}
\newtheorem{proposition}[theorem]{Proposition}
{\theorembodyfont{\rmfamily}\newtheorem{remark}[theorem]{Remark}}
{\theorembodyfont{\rmfamily}}

\newcommand{\N}{\mathbb{N}}
\newcommand{\Z}{\mathbb{Z}}
\newcommand{\R}{\mathbb{R}}

\newcommand{\PP}{\mathbb{P}}
\newcommand{\E}{\mathbb{E}}
\newcommand{\Ent}{{\rm Ent}}
\newcommand{\dd}{\,{\rm d}}
\newcommand{\e}{\varepsilon}

\DeclareMathOperator{\var}{var}
\DeclareMathOperator{\osc}{osc}

\newcommand{\bxi}{\boldsymbol{\xi}}

\def\Xint#1{\mathchoice
	{\XXint\displaystyle\textstyle{#1}}
	{\XXint\textstyle\scriptstyle{#1}}
	{\XXint\scriptstyle\scriptscriptstyle{#1}}
	{\XXint\scriptscriptstyle\scriptscriptstyle{#1}}
	\!\int}
\def\XXint#1#2#3{{\setbox0=\hbox{$#1{#2#3}{\int}$}
		\vcenter{\hbox{$#2#3$}}\kern-.5\wd0}}

\def\dashint{\Xint-}

\newenvironment{proof}{{\vskip\baselineskip\noindent\textbf{Proof:}}}
{\hspace*{.1pt}\hspace*{\fill}\BOX\vskip\baselineskip}

\newenvironment{proofx}[1]
{\vskip\baselineskip\noindent\textbf{Proof of {#1}:}}
{\hspace*{.1pt}\hspace*{\fill}\BOX\vskip\baselineskip}

\begin{document}

\title{
Wall laws for viscous flows in 3D randomly rough pipes: optimal convergence rates and stochastic integrability 
}

\author{
Mitsuo Higaki\thanks{Department of Mathematics, Graduate School of Science, Kobe University, 1-1 Rokkodai, Nada-ku, Kobe 657-8501, Japan. 
\textit{Email address:}\texttt{higaki@math.kobe-u.ac.jp}
}
\and
Yulong Lu \thanks{206 Church St. SE, School of Mathematics, University of Minnesota, Minneapolis, MN 55455, USA.
\textit{Email address:}\texttt{yulonglu@umn.edu}}
\and
Jinping Zhuge \thanks{Morningside Center of Mathematics, Academy of Mathematics and Systems Science, Chinese Academy of Sciences, Beijing 100190, China.
\textit{Email address:}\texttt{jpzhuge@amss.ac.cn}} 
}	
\date{}

\maketitle

\noindent {\bf Abstract.}
This paper is concerned with effective approximations and wall laws of viscous laminar flows in 3D pipes with randomly rough boundaries. The random roughness is characterized by the boundary oscillation scale $\varepsilon \ll 1 $ and a probability space with ergodicity quantified by functional inequalities. The results in this paper generalize the previous work for 2D channel flows with random Lipschitz boundaries to 3D pipe flows with random boundaries of John type. Moreover, we establish the optimal convergence rates and substantially improve the stochastic integrability obtained in the previous studies. Additionally, we prove a refined version of the Poiseuille's law in 3D pipes with random boundaries, which seems unaddressed in the literature. Our systematic approach combines several classical and recent ideas (particularly from homogenization theory), including the Saint-Venant's principle for pipe flows, the large-scale regularity theory over rough boundaries, and applications of functional inequalities.


\tableofcontents

    \section{Introduction}
    \label{sec.intro}

    \subsection{Motivation}

\noindent
This paper is devoted to the effective approximation of viscous flows, governed by the stationary Navier-Stokes system, in 3D pipes with randomly rough boundaries. Understanding the effect on fluid induced by boundary roughness remains a core challenge in fluid dynamics even in the laminar regime when the Reynolds number of the flow is sufficiently small. Indeed, laminar flows with rough boundaries discussed in this paper naturally appear in biomechanics (e.g., blood flow in vessels) and engineering (e.g., corroded or corrugated pipes). If the roughness is relatively mild, one may address such problems by finding analytic solutions \cite{CS72,SJZ01,Wang06,MeiJin2016} or by numerical computations \cite{FN77,DS79,NOK84,MohFlo2013,LLS2019}, albeit not exhaustively.
        
These direct approaches are unavailable or require extensive computational resources if the boundaries are more complicated and wildly oscillating at small scales. In the presence of such irregularities, it becomes impractical to record the fluid dynamics meticulously; instead, one should focus on understanding the average fluid configuration. Wall laws or effective boundary conditions serve this purpose. Here we quote a definition from the pioneering work by J\"{a}ger-Mikeli\'{c} \cite{JM01}: ``In general, the approach of replacing the no-slip condition at rough boundaries with the non-penetration condition plus a relation between the tangential velocity and the shear stress, is called the \textit{wall laws}." Namely, the rough boundary is replaced by an artificially smoothed one, and an effective boundary condition is imposed to represent the averaged effect of the roughness. Typically, the effective boundary condition for laminar flows is known to be Navier-type slip, and thus called Navier's wall law.

As a representative work from the engineering side, Nikuradse \cite{Ni50} empirically studied the wall laws by detailed experiments for turbulent flows; see \cite{KSCT2005,FlaSch2014,CHSF2021} for recent progress in this direction. Achdou-Pironneau-Valentin \cite{APV98} verified by numerical simulation the Navier's wall laws for laminar flows, derived by two-scale asymptotic expansion, for a periodic boundary with roughness in scales $\e\ll1$; see also \cite{AP95,APV98b}. Then the mathematically rigorous justification of the Navier's wall laws is obtained by J\"{a}ger-Mikeli\'{c} for 2D periodic channels in \cite{JM01} and for 3D slabs between periodic plates in \cite{JM03}, with the optimal $O(\e^{3/2})$ convergence rates, i.e., the error bounds in $L^2$ space between the real solution and the effective approximation. In most physically relevant scenarios, however, the boundaries are self-similar with non-periodic or random structures. In a series of papers \cite{BasGer2008,Ger09,GerMas2010}, G\'{e}rard-Varet with collaborators studied the Navier's wall laws in 2D channels with randomly oscillating boundaries. The seminal work by Basson--G\'{e}rard-Varet \cite{BasGer2008} established the Navier's wall law for rough boundaries generated by stationary random processes and obtained suboptimal convergence rates $o(\e)$. The near-optimal convergence rate $O(\e^{3/2}|\log \e|^{1/2})$ was obtained by \cite{Ger09} under the assumption of finite range of dependence on the random boundaries. Finally, G\'{e}rard-Varet--Masmoudi \cite{GerMas2010} proved the optimal convergence rates $O(\e^{3/2})$ for quasi-periodic boundaries satisfying a Diophantine condition.

In summary, the previous mathematical study of wall laws mainly targets laminar flows in 2D channels or 3D slabs, which are less applicable compared to those in 3D pipes. Additionally, all of these studies assume that the boundaries are at least locally Lipschitz. The objective of this paper is to develop a systematic approach to study wall laws and quantitative convergence rates for laminar flows in 3D pipes with random roughness. We will also relax the boundary regularity from local Lipschitz boundaries to more generalized boundaries of John type, which are not necessarily local graphs of functions and allow for fractals and cusps. Moreover, we will substantially improve the stochastic integrability under certain ergodicity conditions.

    \subsection{Main results}

We first define the infinite pipes with rough boundaries.
Let $D_r = \{ x = (x_1,x_2,x_3) \in \R^3~|~ x_2^2 + x_3^2 < r^2  \}$ denote an infinite cylinder with radius $r>0$. Then we consider a rough infinite cylinder (connected open set) $D^\e$ satisfying the thickness condition
\begin{equation}\label{def.De.pipelike}
    D_1 \subset D^\e \subset D_{1+\e}.
\end{equation}
Here $0<\e \ll 1$ is the thickness of the rough boundary of $D^\e$, as well as the wavelength of the oscillation. Due to this set-up, we often write $D_1$ as $D^0$ without confusion. In view of the geometry, it is convenient to use the cylindrical coordinates $(x_1,r,\theta)$ on $\R^3$
\begin{align}\label{def.cyl.coord}
\begin{bmatrix}
x_1\\
x_2\\
x_3
\end{bmatrix}
=
\begin{bmatrix}
x_1\\
r\cos \theta\\
r\sin \theta
\end{bmatrix},
\quad
(x_1,r,\theta)
\in
\R
\times
[0,\infty)
\times
[0,2\pi).
\end{align}
For example, the cylinder $D^0$ can be simply expressed as $D^0 = \{ r<1\}$. For a specific point $x = (x_1,x_2,x_3)$, we use $(x_1 ,r_x, \theta_x)$ to denote its cylindrical coordinates. Moreover, given $a \ge \e$ and $x\in \partial D^0$, let $S_{a}(x)$ be the cylindrical cube defined by
\begin{align}\label{def.cyl.cube}
    S_{a}(x) 
    = \{ y\in \R^3 ~|~
    |y_1 -x_1|<a, 
    \mkern9mu 
    |\theta_y - \theta_x|<a, \mkern9mu 
    |r_y - r_x| < a \}.
\end{align}
Define $S_{a,+}(x) = D^\e \cap S_a(x)$.

The mathematical study of fluids occupying $D^\e$ requires some minimal regularity on the boundary known as John domains \cite{John1961} for which the divergence equation is solvable in $L^2$ space \cite{ADM06}. We first recall the definition of the usual bounded John domains.

\begin{definition}\label{def.John}
    Let $\omega \subset \R^d$ be an open bounded set and $\tilde{x} \in \omega$. We say that $\omega$ is a John domain 
    (or a domain with boundary of John type)
    with respect to the center $\tilde{x}$ and with constant $L>0$ if for any $y\in \omega$, there exists a Lipschitz mapping $\rho:[0,|y-\tilde{x}|] 
    \to \omega$ with Lipschitz constant $L$, such that $\rho(0) = y, \rho(|y-\tilde{x}|) = \tilde{x}$ and ${\rm dist}(\rho(t),\partial\omega) \ge t/L$ for all $t\in [0, |y-\tilde{x}|]$. 
\end{definition}

Examples of John domains are: Lipschitz domains, one-sided NTA domains, domains with inward cusps or certain fractals such as Koch's snowflake. Notice that domains with outward cusps are not John domains. 

The following definition is a variant of the unbounded John domains in \cite{HPZ-21,HZ23}.
	\begin{definition}\label{def.John2}
		Let $D^\e$ be an unbounded domain satisfying \eqref{def.De.pipelike}.
		We say that $D^\e$ is a \emph{John domain} above $\e$-scale with constants $(L,K)$ if for any $x \in \partial D^0$ and any $R \in [\e, 1]$, there exists a bounded John domain $S^*_R(x)$ with respect to the center $x_R \in D^0$ and with constant $L>0$
		according to Definition \ref{def.John}
		 such that
		\begin{equation}\label{e.exisinterdom}
			S_{R,+}(x) \subset S^*_R(x) \subset S_{KR,+}(x),
		\end{equation}
		where $S_{R,+}(x) = D^\e \cap S_R(x)$.
	\end{definition}

The constants $(L, K)$ do not play roles in our proof (except for the dependence of constant $C$) and we may assume $K = 2$ without loss of generality: $(L, K)=(L, 2)$.

Next we define the flux in John domains $D^\e$. Let $\Gamma^\e = \partial D^\e$. Let $-\infty\le a < b\le \infty$.  Let $D^\e_{a,b} = D^\e \cap \{ a<x_1 < b \}$ and $\Gamma^\e_{a,b} = \Gamma^\e \cap \{ a<x_1 < b \}$. Assume that $u = (u_1,u_2,u_3)\in H^1(D^\e_{a,b})^3$ satisfies $\nabla\cdot u = 0$ in $D^\e_{a,b}$ and $u = 0$ on $\Gamma^\e_{a,b}$ (in the sense of trace). We say that the 
flux of $u$ in $D^\e_{a,b}$ is $\phi$ if it holds that, for any $\varphi(x_1) \in C_0^\infty([a,b])$,
\begin{equation}\label{def.flux}
    \int_{D^\e_{a,b}} u_1(x) \varphi(x_1) \dd x = \phi \int_a^b \varphi(x_1) \dd x_1.
\end{equation}
If this is the case, we write $\Phi^\e [u] = \phi$.
Alternatively, one can extend $u$ by zero to a slightly larger cylinder (e.g. to $D_{1+\e}$) so that the flux can be define in the usual way in terms of the trace on $\Sigma^\e_t := D^\e \cap \{ x_1 = t\}$. In this extension, the divergence-free condition is preserved. Indeed, we can show that, for any $t\in [a,b]$, 
\begin{equation}
    \Phi^\e[u] = \int_{\Sigma^\e_t} e_1\cdot u \dd \sigma = \int_{\Sigma^\e_t} u_1 \dd \sigma,
\end{equation}
where $e_1 = (1,0,0)$ and $\dd\sigma$ is the surface measure. Similarly, for $u\in H^1(D^0_{a,b})^3$ with $\nabla\cdot u = 0$ in $D^0_{a,b}$ and $u\cdot \nu = 0$ on $\Gamma^0_{a,b}$ (where $\nu$ is the inward normal vector on $\Gamma^0_{a,b}$), we say that the flux of $u$ in $D^0_{a,b}$ is $\phi$ if it holds that, for any $t\in [a,b]$,
\begin{equation}
    \Phi^0[u] = \int_{\Sigma_t^0} e_1\cdot u \dd \sigma = \int_{\Sigma_t^0} u_1 \dd \sigma.
\end{equation}

Now we introduce the stationary Navier-Stokes system in a rough cylinder $D^\e$ with the no-slip boundary condition on $\Gamma^\e = \partial D^\e$.
Consider
\begin{equation}\tag{NS}\label{eq.NS}
\left\{
\begin{array}{ll}
-\Delta u^\e + u^\e\cdot \nabla u^\e+ \nabla p^\e = 0 &\mbox{in}\ D^\e,\\
\nabla\cdot u^\e = 0 &\mbox{in}\ D^\e,\\
\Phi^\e[u^\e] = \phi,\\
u^\e = 0 &\mbox{on}\ \Gamma^\e.
\end{array}\right.
\end{equation}
Since we are interested in the laminar regime, we have assumed that the viscosity is equal to $1$ and assume that the flux $\phi$ is small. The existence and uniqueness of solutions for \eqref{eq.NS}, known as the Leray's problem, can be established for small flux and remains to be an open question without smallness; see Section \ref{SVP}. In view of the formal limit $\e\to0$, it is naturally expected that the solution $u^\e$ of \eqref{eq.NS} is approximated well by the solution of the Navier-Stokes system in the smooth pipe $D^0$ with the no-slip condition on $\Gamma^0 := \partial D^0$
\begin{equation}\label{eq.NS0}
\left\{
\begin{array}{ll}
-\Delta u^0 + u^0\cdot \nabla u^0 + \nabla p^0 = 0 &\mbox{in}\ D^0,\\
\nabla\cdot u^0 = 0 &\mbox{in}\ D^0,\\
\Phi^0[u^0] = \phi,\\
u^0 = 0 &\mbox{on}\ \Gamma^0. 
\end{array}\right.
\end{equation}
This approach to find approximations is called the Dirichlet's wall law. The unique bounded solution of \eqref{eq.NS0} is, under the smallness on $\phi$, the famous Hagen-Poiseuille flow in a pipe
\begin{equation}\label{eq.HPflow1}
u^0 = \frac{2\phi}{\pi}(1-r^2)e_1,
\qquad
p^0=-\frac{8\phi}{\pi} x_1.
\end{equation}
The Hagen-Poiseuille flow $u^0$ does not decay along the $x_1$-axis. Since we are considering the solutions of \eqref{eq.NS} that are near $u^0$ in appropriate topologies, we introduce
\begin{equation}
    L^2_{\rm uloc}(D^\e):= \Big\{ f\in L^2_{\rm loc}(D^\e)~\Big|~ \|f \|_{L^2_{\rm uloc}(D^\e)} :=\sup_{k\in \R} \|f \|_{L^2(D^\e_{k,k+1})} < \infty  \Big\}.
\end{equation}
It is not hard to define the Sobolev space $H^1_{\rm uloc}(D^\e)$ in a similar manner.

Without any additional self-similar structure on the rough boundary $\Gamma^\e$, we have the following approximation estimate that verifies the Dirichlet's wall law rigorously.
\begin{theorema}\label{thma}
Assume that $D^\e$ is a John domain satisfying \eqref{def.De.pipelike}.
    There exists sufficiently small $\phi_0 > 0$ such that for all $|\phi|<\phi_0$, the system \eqref{eq.NS} has a unique solution $u^\e \in H^1_{\rm uloc}(D^\e)^3$ such that
    \begin{equation}\label{0th.err}
        \|u^\e - u^0\|_{L^2_{{\rm uloc}}(D^\e)}
        \le
        C |\phi| \e, 
        \qquad
        \|\nabla u^\e - \nabla u^0\|_{L^2_{{\rm uloc}}(D^\e)}
        \le
        C |\phi| \e^{1/2}, 
    \end{equation}
    where $u^0$ is given by \eqref{eq.HPflow1} and the constant $C$ is independent of $\phi, \e$. 
\end{theorema}

Observe that theorem \ref{thma} is established under the deterministic setting independent of the particular element $D^\e$ with the exception of the John domain assumption. In order to get a higher-order approximation estimate or convergence rate, the self-similar structure of the boundaries of $D^\e$ at small scales must be taken into account. In this paper, we are focusing on randomly rough pipes under certain stationary and ergodic assumptions.

Let $\Omega^\e$ be the sample space of rough infinite cylinders $D^\e$ satisfying the thickness condition \eqref{def.De.pipelike} and the John domain condition in Definition \ref{def.John2} with uniform constants $(L,K)=(L,2)$. Two elements $D^\e$ and $\widetilde{D}^\e$ are regarded as the same element if and only if $D^\e = \widetilde{D}^\e$. 
Let $\mathcal{F}^\e$ be the maximal $\sigma$-algebra generated by the elements of $\Omega^\e$. We further assume that there exists a probability measure $\PP :\mathcal{F}^\e \to [0,1]$ satisfying the following property: $\PP $ is stationary with respect to a group of transformations consisting of $\theta$-rotations and $x_1$-translations, i.e., $\PP $ satisfies, for any $\theta \in [0, 2\pi)$ and $t \in \R$,
\begin{equation}\label{Stationarity3D}
\PP  \circ R_\theta = \PP  \circ T_t =  \PP ,
\end{equation}
where
\begin{equation}\label{Def.Trans+Rotation}
R_\theta 
=\begin{bmatrix}
1 & 0 & 0\\
0 & \cos \theta & -\sin \theta\\
0 & \sin \theta & \cos \theta 
\end{bmatrix},
\qquad
T_t D^\e = (t,0,0) + D^\e.
\end{equation}
Under the above assumptions, $(\Omega^\e, \mathcal{F}^\e, \PP)$ is a probability space. Roughly, the random elements $D^\e$, or their boundaries $\Gamma^\e$, can be viewed as homogeneous random fields on $\R \times \mathbb{S}^1 $ which are stationary with respect to $\theta$-rotations and $x_1$-translations. For the corresponding homogeneous random field for 2D channels, we refer to \cite{BasGer2008,Ger09}.

To establish a higher-order convergence rate that improves Theorem \ref{thma}, we need to assume certain quantitative ergodicity for random elements (fields) in $\Omega^\e$. Indeed, it is assumed in G\'{e}rard-Varet \cite{Ger09} that random rough boundaries have no correlation at large distances, interpreted as a version of finite range of dependence assumption, leading to almost sharp convergence rate under $L^2$ stochastic integrability. In this paper, inspired by stochastic homogenization theory recently advanced by Gloria-Otto (e.g. \cite{GloOtt17}), we introduce the functional inequalities on random rough boundaries, rather than on coefficients as is typically done in homogenization theory. Such quantitative ergodicity eventually yields sharp convergence rates with exponential or Gaussian stochastic integrability.

We consider two functional inequalities as displayed in Definition \ref{def.funct.ineq} below, in the similar spirit as \cite[Definition 2.2]{GloOtt17} for the first one and as \cite[Definition 2.1]{DueGlo20-2} for the second one. Let $X$ be a scalar or vector-valued random variable defined on the probability space $(\Omega^\e, \mathcal{F}^\e, \PP)$. Then we define the oscillation of a random variable $X$ with respect to $D^\e$ restricted to the cube $S_{a}(z)$ with radius $a \ge \e$ and center $z \in \Gamma := \Gamma^0$ by
\begin{equation}\label{def.osc}
\begin{aligned}
    & \Big(\underset{D^\e|_{S_{a}(z)}}{\osc} X\Big)(D^\e) \\
    & = \sup \bigg\{ |X(\widetilde{D}^\e) - X(\widehat{D}^\e)| ~\bigg|~ 
    \begin{array}{l}
    \widetilde{D}^\e, \widehat{D}^\e\in \Omega^\e, \\\partial\widetilde{D}^\e = \partial\widehat{D}^\e = \partial D^\e \text{ in } \R^3\setminus S_{a}(z)
    \end{array}
    \bigg\}.
\end{aligned}
\end{equation}

\begin{definition}\label{def.funct.ineq}
Let $(\Omega^\e,\mathcal{F}^\e, \PP)$ be the probability space defined above.
\begin{enumerate}[(1)]
\item
We say that $(\Omega^\e, \mathcal{F}^\e, \PP)$ satisfies the spectral gap inequality if there exists $C>0$ such that for any random variable $X$ on $(\Omega^\e,\mathcal{F}^\e, \PP)$, we have
\begin{equation}\tag{SG}\label{ineq.SG}
\var[X]
\le
C \E\bigg[ \e^{-2} \int_{\Gamma}
\big(
\underset{D^\e|_{S_{\e}(z)}}{\osc} X
\big)^2 
\dd \Gamma_z \bigg].
\end{equation}

\item
We say that $(\Omega^\e, \mathcal{F}^\e, \PP)$ satisfies the logarithmic Sobolev inequality if there exists $C>0$ such that for any random variable $X$ on $(\Omega^\e,\mathcal{F}^\e, \PP)$, we have
\begin{equation}\tag{LSI}\label{ineq.LSI}
\Ent[X]
\le
C \E\bigg[ \e^{-2} \int_{\Gamma}
\big(
\underset{D^\e|_{S_{\e}(z)}}{\osc} X
\big)^2 
\dd \Gamma_z \bigg],
\end{equation}
where
\begin{equation}
\Ent[X] = \E[X^2 \log X^2] - \E[X^2] \log \E[X^2].
\end{equation}
\end{enumerate}
\end{definition}

\begin{remark}\label{rem.def.funct.ineq}
    The functional inequalities in probability spaces are widely used in various applications, not limited to homogenization theory. Concrete examples satisfying \eqref{ineq.SG} or \eqref{ineq.LSI} will be constructed by Poisson point process or Bernoulli process in Appendix \ref{sec.ex}.
\end{remark}

Our second main theorem concerns the Navier's wall law, which provides a higher-order approximation for the solution $u^\e$ of \eqref{eq.NS}. In fact, by a formal calculation with the help of boundary layers (see Section \ref{outline} and Section \ref{sec.bl}) and under the stationarity assumption, the higher-order approximation of $(u^\e,p^\e)$ takes the form of
\begin{equation}\label{eq.uN.sol}
    u^{\rm N}
    = u^0 + \e \alpha u^1,
    \qquad
    p^{\rm N} 
    = p^0 + \e \alpha p^1,
\end{equation}
where $u^0$ is the Hagen-Poiseuille flow in \eqref{eq.HPflow1} and 
\begin{equation}\label{eq.u1p1}
    u^1 = -\frac{2\phi}{\pi}(1-2r^2)e_1
    \qquad
    p^1 = \frac{16\phi}{\pi} x_1.
\end{equation}
Here the parameter $\alpha$ is determined by the expectation of the boundary layers and therefore depends only on the probability space $(\Omega^\e, \mathcal{F}^\e, \PP)$ with the stationarity assumption; see Section \ref{NWL} for details. In particular, it reflects a physical property of the rough pipe and is independent of the flux $\phi$.

Then it is straightforward to verify that $(u^{\rm N},p^{\rm N})$ in \eqref{eq.uN.sol} satisfies the Navier-Stokes system in the smooth pipe $D^0$ with a Navier's slip condition on $\Gamma^0 = \partial D^0$
\begin{equation}\label{eq.uN}
    \left\{
    \begin{array}{ll}
    -\Delta u^{\rm N} + u^{\rm N}\cdot\nabla u^{\rm N} + \nabla p^{\rm N} = 0 &\mbox{in}\ D^0,\\
    \nabla\cdot u^{\rm N} = 0 &\mbox{in}\ D^0,\\
    \Phi^0[u^{\rm N}] = \phi, & \\
    (I-\nu\otimes \nu) u^{\rm N}
    = \lambda^\e (I-\nu\otimes \nu) \mathbb{D}(u^{\rm N}) \nu
    &\mbox{on}\ \Gamma^0,\\
    u^{\rm N}\cdot \nu = 0 &\mbox{on}\ \Gamma^0.
    \end{array}\right.
\end{equation}
Here $\nu$ is the inward normal vector on $\Gamma^0$ and $I-\nu\otimes \nu$ is the tangential projection. Moreover, 
\begin{equation*}
    \mathbb{D}(u) = \frac{1}{2}(\nabla u + (\nabla u)^T)
\end{equation*}
is the strain rate tensor. In \eqref{eq.uN}, the relationship between $\lambda^\e$ and $\alpha$ is given by
\begin{equation}\label{eq.sliplength-alpha}
    \lambda^\e = \frac{\alpha \e }{1 - 2 \alpha \e}.
\end{equation}
The system \eqref{eq.uN} will be called the \emph{effective} Navier-Stokes system. The last two equations in \eqref{eq.uN} on the boundary $\Gamma^0$ is called the Navier's slip condition. The fourth equation in \eqref{eq.uN} says that the tangential velocity is proportional to the tangential stress on the boundary, where the constant $\lambda^\e$ is called the slip length and $1/\lambda^\e$ is called the friction factor. The last equation in \eqref{eq.uN} is the non-penetration condition. The Navier's slip condition recovers the no-slip or full-slip condition for $\lambda^\e = 0$ or $\lambda^\e = \infty$, respectively. Observe that the parameters $\phi,\lambda^\e$ uniquely determine the effective Navier-Stokes system. Moreover, the solution \eqref{eq.uN.sol} is unique in $L^2_{\rm uloc}(D^0)^3$ if one imposes the smallness of $\phi,\e$.

The theorem below verifies the Navier's wall law rigorously under functional inequalities and establishes sharp convergence rates. Throughout this paper, we let, for any $x\in D^0$, 
\[
    \delta = \delta(x) := {\rm dist}(x, \Gamma^0) = 1-r_x.
\]
\begin{theorema}\label{thmb}
    Let the probability space $(\Omega^\e, \mathcal{F}^\e, \PP)$ satisfy the stationary assumption. There exists sufficiently small $\phi_0>0$ such that for all $|\phi|<\phi_0$, the solution $u^\e$ of \eqref{eq.NS} in $H^1_{\rm uloc}(D^\e)^3$ and the solution $u^{\rm N}$ of \eqref{eq.uN} given by \eqref{eq.uN.sol} satisfy the following:
    \begin{enumerate}[(1)]
    \item\label{item1.thmb}
    Under the spectral gap inequality \eqref{ineq.SG}, we have
    \begin{equation}\label{est.ue-uN.SG}
        \sup_{k\in \R} \E \Big[ \exp \Big( \frac{c \|u^\e - u^{{\rm N}} \|_{L^2(D^0_{k,k+1})}}{|\phi| \e^{3/2}}\Big) \Big] 
    \le C,
    \end{equation}
    and
    \begin{equation}\label{est.Due-DuN.SG}
        \sup_{k\in \R} \E \Big[ \exp \Big( \frac{c \|\delta (\nabla u^\e - \nabla u^{{\rm N}} )\|_{L^2(D^0_{k,k+1})}}{|\phi| \e^{3/2}}\Big) \Big] 
    \le C.
    \end{equation}
    \item\label{item2.thmb}
    Under the logarithmic Sobolev inequality \eqref{ineq.LSI}, we have
    \begin{equation}\label{est.ue-uN.LSI}
        \sup_{k\in \R} \E \Big[ \exp \Big( \frac{c \|u^\e - u^{{\rm N}} \|^2_{L^2(D^0_{k,k+1}) }}{|\phi|^2 \e^{3}}\Big) \Big] 
    \le C,
    \end{equation}
    and
    \begin{equation}\label{est.Due-DuN.LSI}
        \sup_{k\in \R} \E \Big[ \exp \Big( \frac{c \|\delta (\nabla u^\e - \nabla u^{{\rm N}}) \|^2_{L^2(D^0_{k,k+1})}}{|\phi|^2 \e^{3}}\Big) \Big] 
    \le C.
    \end{equation}
    \end{enumerate}
    The constants $c, C>0$ are independent of $\phi,\e$.
\end{theorema}

\begin{remark}\label{rem.thmb}
    The solution $u^{\rm N}$ of the effective Navier-Stokes system \eqref{eq.uN} is actually defined only by using the stationarity assumption. However, this ``predicted" approximation of $u^\e$ may not provide better convergence rates than the ones in Theorem \ref{thma} if no ergodicity assumptions are made on the probability space. The general validity of the Navier's wall law cannot be expected due to the counterexamples in \cite[Section 3]{GerMas2010}. 
\end{remark}

Let us briefly discuss the novelty of Theorem \ref{thma} and Theorem \ref{thmb}. First, compared with the previous work \cite{BasGer2008,Ger09}, these two theorems generalize the results from 2D channels to 3D pipes with the John domain assumption and establish the Navier's wall law on the cylinder. We point out that the approaches in \cite{BasGer2008,GerMas2010} are not applicable to our domains with extreme roughness, which is explained more in detail in Section \ref{outline}.
Second, in Theorem \ref{thmb}, we establish the optimal convergence rates $O(\e^{3/2})$ (see Remark \ref{rmk.optimality}) with exponential or Gaussian stochastic integrability under \eqref{ineq.SG} or \eqref{ineq.LSI} respectively, including the new convergence rate for gradient. While in \cite{Ger09}, the near-optimal convergence rate $O(\e^{3/2}|\log \e|^{1/2})$ was obtained with $L^2$ stochastic integrability under the assumption of finite range of dependence. It remains to be an interesting question whether one can improve the stochastic integrability with the assumption of finite range of dependence, just as in stochastic homogenization for elliptic equations \cite{AS16,AKM19,AK2022}.

Our last main theorem is an approximate Poiseuille's law, illustrating the relationship between the flux, the pressure drop, and the slip length (or fraction factor) appearing in the effective system \eqref{eq.uN}. Let us recall from e.g. \cite{SS93, MOHR12} that the well-known Poiseuille's law in a circular pipe for laminar flows under no-slip condition is 
    \begin{equation}
        \Delta_\ell p := p(0) - p(\ell e_1) = \frac{8\mu \ell \phi}{\pi R^4},
    \end{equation}
where $\Delta_\ell p$ is the so-called pressure drop,  $\ell$ is the distance at which the pressure drop is measured, $\mu$ is the dynamic viscosity, $\phi$ is the flux and $R$ is the radius of the pipe. Even in rough domains, the above Poiseuille's law remains accurate up to an error of $O(\e)$ thanks to Theorem \ref{thma}. Hence, in a manner akin to the derivation of the Navier's wall law, it is reasonable to ask if there is a more accurate Poiseuille's law. A further question for this problem involves practical calculation of the slip length $\lambda^\e$ in the effective system \eqref{eq.uN}. Indeed, the slip length $\lambda^\e$, determined by the expectation of the boundary layers (see Section \ref{NWL}), plays the same role as the homogenized coefficients in homogenization theory \cite{JKO1994,AK2022}, whose computation in practice from a realization of random coefficients is a classical challenge; see e.g. \cite{mou19, Fis19, CJOX24} for some recent developments. As far as the authors are aware, a robust algorithm for computing the slip length in the presence of random boundaries has not been developed in the previous mathematical studies.

    The following theorem answers the above two questions.

    \begin{theorema}[Poiseuille's law for rough pipes] \label{thmc}
        Under the same assumptions as Theorem \ref{thmb}, for any $\ell \in (0,1]$, we have:
        \begin{enumerate}[(1)]
            \item Under the spectral gap inequality \eqref{ineq.SG}, we have
            \begin{equation}\label{est.D1p.SG}
            \E \Big[\exp \Big(  \frac{c |\Delta_\ell p^\e - (8\phi \ell/\pi)(1 - 2\e \alpha)| }{|\phi| \e^{3/2}}  \Big) \Big] \le C.
        \end{equation}
        \item Under the logarithmic Sobolev inequality \eqref{ineq.LSI}, we have
        \begin{equation}\label{est.D1p.LSI}
            \E \Big[\exp \Big(  \frac{c |\Delta_\ell p^\e - (8\phi \ell/\pi)(1 - 2\e \alpha)|^2 }{|\phi|^2 \e^{3}}  \Big) \Big] \le C.
        \end{equation}
        \end{enumerate}  
        The constants $c, C>0$ are independent of $\phi,\e$.
    \end{theorema}
    
    The Poiseuille's law for a rough pipe in Theorem \ref{thmc} together with \eqref{eq.sliplength-alpha} relates the slip length to the pressure drop and flux for a laminar flow under consideration. To the authors' best knowledge, Theorem \ref{thmc} is new and provides a practical and robust method to experimentally measure the slip length $\lambda^\e$ from the pressure drop and flux.

    \subsection{Outline of proofs}\label{outline}

A novel and systematic approach is developed in this paper to study the wall laws, combining several concepts originating from the recent progress in stochastic homogenization theory. Our proofs fundamentally depend on the quantitative analysis of boundary layers which are understood as correctors in the context of homogenization theory. Precisely, the boundary layers are the solutions of linear Stokes system with an oscillating boundary data
\begin{equation}\label{intro.eq.1stBL}
\left\{
\begin{array}{ll}
-\Delta v^\e_{\rm bl}+\nabla \pi^\e_{\rm bl} =0&\mbox{in}\ D^\e, \\
\nabla\cdot v^\e_{\rm bl}=0&\mbox{in}\ D^\e, \\
v^\e_{\rm bl} = (1-r^2) e_1 &\mbox{on}\ \Gamma^\e.
\end{array}\right.
\end{equation}
Notice that the solution to this system is not unique as the flux is allowed to be arbitrary. For our application, we ``normalize" the boundary layers through
\begin{equation}\label{eq.fluxcorrector}
    \Phi^\e[(1-r^2)e_1 - v^\e_{\rm bl}] = \frac{\pi}{2}.
\end{equation}
The normalized boundary layer is carefully constructed so that the approximation $u^\e_{\rm app}$ in \eqref{Intro.uapp} below has favorable properties in the proofs. It is emphasized that our boundary layers are still defined in the whole domain $D^\e$ without any rescaling due to the geometry of $D^\e$. This is quite different from the boundary layers in 2D channels considered by \cite{BasGer2008,Ger09,GerMas2010}, which are usually defined in upper half-planes after rescaling. 

Boundary layers are commonly employed as correctors used to eliminate discrepancies caused by irregular boundaries. In the system \eqref{intro.eq.1stBL}, the boundary data is of $O(\e)$ on $\Gamma^\e$ due to the condition  \eqref{def.De.pipelike} and does not decay along the $x_1$-axis. Thus the solutions of \eqref{intro.eq.1stBL} are supposed to have size of $O(\e)$ locally and do not decay as $x_1$ goes to infinity.
Under such constraints, obtaining the solutions of \eqref{intro.eq.1stBL} is a challenging problem due to the unboundedness and the extreme roughness of the domain $D^\e$. 

Actually, the existence of boundary layers is one of the major questions that has been addressed in \cite{GerMas2010} where the well-posedness of boundary layer systems over bumpy half-space was established for arbitrary Lipschitz boundaries. This Lipschitz regularity assumption was recently relaxed to the John domains in \cite{HPZ-21} by using the large-scale Lipschitz estimate (also see \cite{KP15,KP18,HP20,Z21,HZ23} for closely related developments), which is a key idea originating from homogenization theory. This paper adopts the approach developed in \cite{HPZ-21} to construct and estimate the boundary layers in rough pipes of John type.

The essential tool to tackle the difficulties in solving \eqref{intro.eq.1stBL} is the Green's function for the Stokes system. The Green's function method is advantageous due to its broad applicability; however, it requires detailed and precise analysis when one deals with quantitative estimates. We thus combine the Saint-Venant's principle and the large-scale Lipschitz estimate near the rough boundary at a scale no less than $\e$. This combination enables us obtain both exponential decay of the Green's function in the direction of axis and the sharp large-scale estimates. These two results are crucial for constructing the solution of \eqref{intro.eq.1stBL} and for deriving its deterministic size estimates in Proposition \ref{prop.ubl}, as well as for showing decay of correlation for locally perturbed boundaries in Proposition \ref{prop.correlation}. Let us reiterate that our approach of construction and estimates is established for the boundaries of John type which could be fractals, and can be readily adapted to rough domains with different geometry.

Both in the proofs of Theorem \ref{thma} and Theorem \ref{thmb}, an approximation of the solution $u^\e$
\begin{equation}\label{Intro.uapp}
    u^\e_{\rm app}=u^0 - \frac{2\phi}{\pi} v^\e_{\rm bl}
\end{equation}
will be used, where $v^\e_{\rm bl}$ is the boundary layer determined by \eqref{intro.eq.1stBL}--\eqref{eq.fluxcorrector}. Notice that
\[
\Phi^\e[u^\e_{\rm app}] = \Phi^\e[u^\e] = \phi,\qquad
u^\e_{\rm app} = u^\e = 0 \quad \mbox{on} \mkern9mu \Gamma^\e.
\]
Hence, given the size estimates of the boundary layers, $u^\e_{\rm app}$ is naturally expected to be more accurate than $u^0$ in \eqref{eq.HPflow1} as an approximation of $u^\e$ in $D^\e$, sufficiently capturing the highly oscillating structure of $u^\e$ near the rough boundary $\Gamma^\e$ due to the no-slip condition. This insight is indeed correct, and moreover, the approximation $u^\e_{\rm app}$ is so well-constructed that Theorem \ref{thma} follows just by energy estimate and the size estimates of the boundary layers. 

Remark that the proof of Theorem \ref{thma} is purely deterministic and the random structure of boundaries does not play any role. However, to prove Theorem \ref{thmb}, 
it is necessary to take advantage of the quantitative ergodicity assumptions. The usefulness of ergodic assumptions is guaranteed by the expectation of the approximation $u^\e_{\rm app}$ under the stationarity assumption: 
\begin{equation}\label{Intro.asym}
    \E[u^\e_{\rm app}] = u^0 - \frac{2\phi}{\pi} \E[v^\e_{\rm bl}] = u^{{\rm N}} + O(\e^2)
    \quad \text{in} \mkern9mu L^2_{\rm uloc}(D^0)^3,
\end{equation}
where $u^{{\rm N}}$ is defined in \eqref{eq.uN.sol} and the order $O(\e^2)$ can be verified by asymptotic analysis of $\E[v^\e_{\rm bl}]$ for $\e\ll 1$; see Section \ref{NWL} for details. The relation \eqref{Intro.asym} suggests that one may convert the concentration inequality of the boundary layers to that of the solutions $u^\e$ of \eqref{eq.NS}. This is actually the case thanks to the following technical lemma of independent interest, which is proved under the minimal assumption on the probability space $(\Omega^\e, \mathcal{F}^\e, \PP)$. 
\begin{lemma}\label{lema}
    Let the probability space $(\Omega^\e, \mathcal{F}^\e, \PP)$ satisfy the stationary assumption. There exists sufficiently small $\phi_0>0$ such that for all $|\phi|<\phi_0$, the solution $u^\e$ of \eqref{eq.NS} in $H^1_{\rm uloc}(D^\e)^3$ and the solution $u^{\rm N}$ of \eqref{eq.uN} given by \eqref{eq.uN.sol} satisfy the following: for all $k\in\R$,
    \begin{align}\label{1st.err}
    \begin{split}
    \|u^\e - u^{\rm N}\|_{L^2(D^0_{k,k+1})} & \le
    C_0|\phi| ( \|v^\e_{\rm bl} - \E[v^\e_{\rm bl}] \|_{L^2(D^0_{k,k+1})}
    + \e^{3/2})\\
    &\quad
    + C_0|\phi|^2 
    \int_{\R}
    \|v^\e_{\rm bl} - \E[v^\e_{\rm bl}] \|_{L^2(D^0_{s,s+1})}  e^{-c |s-k|} \dd s,
    \end{split}
    \end{align}
    and
    \begin{align}\label{1st.err-2}
    \begin{split}
    \|\delta(\nabla u^\e - \nabla u^{\rm N})\|_{L^2(D^0_{k,k+1})} 
    &\le
    C_0|\phi| ( \|\delta(\nabla v^\e_{\rm bl} - \nabla \E[v^\e_{\rm bl}]) \|_{L^2(D^0_{k,k+1})}
    + \e^{3/2})\\
    &\quad
    + C_0|\phi|^2 
    \int_{\R}
    \|v^\e_{\rm bl} - \E[v^\e_{\rm bl}] \|_{L^2(D^0_{s,s+1})}  e^{-c |s-k|} \dd s.
    \end{split}
    \end{align}
    The constants $c, C_0>0$ are independent of $\phi,\e,k$.
\end{lemma}

Observe that the estimates themselves in Lemma \ref{lema} are of deterministic feature and the functional inequalities are not used at all. The advantage of this lemma is that it allow us to reduce the convergence rate of the solutions of the nonlinear system \eqref{eq.NS} to the localized convergence rate of boundary layers, independent of the ergodicity condition imposed on the probability space $(\Omega^\e, \mathcal{F}^\e, \PP)$. We also mention that the exponential localization for the boundary layers in \eqref{1st.err}--\eqref{1st.err-2} is due to the Saint-Venant's principle.

Theorem \ref{thmb} is a consequence of Lemma \ref{lema} together with the concentration inequality of the boundary layers $v^\e_{\rm bl}$ in Proposition \ref{prop.flu.vbl}, where the functional inequalities \eqref{ineq.SG} and \eqref{ineq.LSI} in Definition \ref{def.funct.ineq} are applied to $v^\e_{\rm bl}$ viewed as a random variable. Similar to the approach in \cite{Ger09}, the decay of correlation of the boundary layers is pivotal in the procedure.

Finally, Theorem \ref{thmc} follows from Theorem \ref{thmb}, particularly the interior gradient convergence rate, and a standard oscillation estimate of pressure.

This paper is organized as follows. In Section \ref{sec.est}, we collect useful estimates in this paper. In Section \ref{sec.bl}, we construct and estimate the boundary layers. Additionally, we discuss the Navier's wall law and establish the concentration inequalities of boundary layers based on the functional inequalities in Definition \ref{def.funct.ineq}. In Section \ref{sec.prf}, we prove the main results, Theorems \ref{thma}, \ref{thmb} and \ref{thmc}.

\smallskip

\noindent
{\bf Notation.}\
In this paper, all the constants $C$ will not depend $\e$ nor $\phi$ and may change from lines to lines.. Some constants may depend on the parameter $(L,K)=(L,2)$ of John domains without mentioning. We write $A\lesssim B$ if there is some constant $C$ such that $A\le CB$. We also write $A\approx B$ if $A \lesssim B\lesssim A$. We denote the average integral over a measurable set $E$ by $\dashint_E f = |E|^{-1} \int_E f$.

    \section{Useful estimates}\label{sec.est}

   This section gathers and establishes useful estimates throughout this paper. The estimates based on the Saint-Venant's principle are provided for the Stokes and Navier-Stokes systems in Section \ref{SVP}, and the large-scale boundary regularity estimates are recalled from \cite{HPZ-21} for the Stokes system in Section \ref{LSLipschitz}. Using these estimates, we establish the detailed quantitative estimates for the Green’s function for the Stokes system in Section \ref{Green}.

    \subsection{Saint-Venant's principle}\label{SVP}

The Saint-Venant's principle, as applied to our current situation, states that in the case of vanishing flux, the Stokes flow exhibiting mild growth at infinity must exponentially decay. This classical property has been well-understood in straight cylinders and more general smooth unbounded domains with outlets at infinity, as documented in the monograph \cite{Gal2011book}. In this section, we will extend the Saint-Venant's principle to the bumpy John domain $D^\e$ and present the proofs that highlights how to effectively handle the boundary roughness of John type. It is crucial for the uniqueness of solutions for the Stokes system in $D^\e$ in suitable function spaces, as well as the exponential decay of the Green's function in Section \ref{Green}. We also prove the Saint-Venant's principle to the Navier-Stokes system in $D^\e$.

Recall from the introduction that for $-\infty\le a<b\le \infty$, we define $D^\e_{a,b}= D^\e \cap \{ a<x_1<b\}$ and $\Gamma^\e_{a,b} = \Gamma^\e \cap  \{ a<x_1<b\}$. For $t \in \R$, we define $D^\e_t = D^\e \cap \{t< x_1 < \infty\}$ and $\Gamma^\e_t = \Gamma^\e \cap \{t< x_1 < \infty\}$.

Given $u_0 \in H^1_{\rm uloc}(D^\e_0)^3$ satisfying $\nabla\cdot u_0 = 0$, consider the linearized Navier-Stokes system around $u_0$
\begin{equation}\label{eq.Stokes.SemiDe}
\left\{
\begin{array}{ll}
-\Delta u + u_0\cdot \nabla u + u\cdot \nabla u_0 +\nabla p=0&\mbox{in}\ D^\e_0,\\
\nabla\cdot u=0&\mbox{in}\ D^\e_0,\\
\Phi^\e[u] = 0,\\
u = 0 &\mbox{on}\ \Gamma^\e_0.
\end{array}\right.
\end{equation}

\begin{theorem}\label{thm.De.SVP}
    Let $u_0 \in H^1_{\rm uloc}(D^\e_0)^3$ with $\nabla\cdot u_0 = 0$ and $u\in H^1_{\rm loc}(\overline{D_0^\e})^3$ be a weak solution of \eqref{eq.Stokes.SemiDe}. Then there exists $\tau_0,b>0$ such that if 
    \begin{equation}\label{cond.thm.De.SVP.u0<tau0}
        \| u_0 \|_{H^1_{\rm uloc}(D^\e_0)} \le \tau_0
    \end{equation}
    and
    \begin{equation}\label{cond.thm.De.SVP.exp}
        \liminf_{t\to \infty} e^{-bt}\int_{D^\e_{0,t}} |\nabla u|^2 =0,
    \end{equation}
    then there exist $C,c>0$ such that, for all $t\ge0$,
    \begin{equation}\label{est.SVP}
    \int_{D^\e_t} |\nabla u|^2 \le C e^{-c t} \int_{D^\e_{0,1}} |\nabla u|^2.
    \end{equation}
\end{theorem}

\begin{remark}\label{rem.thm.De.SVP}
    With obvious modifications, the similar result to Theorem \ref{thm.De.SVP} as well as to Lemma \ref{lem.exp2bdd} below holds if the half-pipe $D^\e_0$ is replaced by $D^\e_k$ for any $k\in \R$, or by the other half-pipe $D^\e_{-\infty,k} := D^\e \cap \{-\infty< x_1 < k\}$ approaching $-\infty$, with the same constants $\tau_0,b,C,c$. However, to maintain brevity, the details will not be explicitly written.
\end{remark}

We first prove the following lemma.

\begin{lemma}\label{lem.exp2bdd}
    Under the assumptions of Theorem \ref{thm.De.SVP}, we have
    \begin{equation}\label{est.Finite Energy}
    \int_{D^\e_0} |\nabla u|^2 
    \le 
    C \int_{D^\e_{0,1}} |\nabla u|^2.
    \end{equation}
\end{lemma}
\begin{proof}
    Let $\eta$ be a smooth function defined in $\R$ such that $\eta(t) = 1$ for $t<0$, $\eta(t) = 0$ for $t>1$ and $|\eta'(t)|\le 2$. Define the nondecreasing function $G(s)$ by
    \begin{equation}\label{def.Gs}
            G(s) = \int_{D^\e_0} |\nabla u|^2 \eta(x_1-s) \dd x,
            \quad s>0.
    \end{equation}

    To avoid the boundary term at the left end $\{x_1=0\}$ due to integration by parts, we introduce another cut-off function. Let $\tilde{\eta}$ be a smooth function in $\R$ supported in $[-2/3,2/3]$ such that $\tilde{\eta}(t) = 1$ for  $t\in[-1/2,1/2]$. Then $\eta(x_1-s)-\tilde{\eta}(x_1)$ vanishes in $0<x_1<1/2$.

    Using the system \eqref{eq.Stokes.SemiDe} and integration by parts, we obtain
    \begin{equation}\label{eq.Gs}
        \begin{aligned}
            G(s) 
            & = \int_{D^\e_0} |\nabla u|^2 (\eta(x_1-s)-\tilde{\eta}(x_1)) \dd x 
            + \int_{D^\e_0} |\nabla u|^2 \tilde{\eta}(x_1) \dd x \\
            & = -\int_{D^\e_0} (u_0\cdot \nabla u)\cdot u (\eta(x_1 -s) - \tilde{\eta}(x_1)) \dd x \\
            &\quad
            - \int_{D^\e_0} (u\cdot \nabla u_0)\cdot u (\eta(x_1 -s) - \tilde{\eta}(x_1)) \dd x \\
            &\quad 
            + \int_{D^\e_0} p u_1 (\eta'(x_1 - s) - \tilde{\eta}'(x_1)) \dd x \\
            &\quad 
            - \int_{D^\e_0} u\cdot \partial_1 u (\eta'(x_1 - s) - \tilde{\eta}'(x_1)) \dd x 
            + \int_{D^\e_0} |\nabla u|^2 \tilde{\eta}(x_1) \dd x \\
            & =: \sum_{j=1}^5 I_j.
        \end{aligned}
    \end{equation}

    \textbf{Estimate of $I_1$ and $I_2$.} 
    Using the identity $(u_0\cdot \nabla u)\cdot u = (1/2) u_0\cdot \nabla |u|^2$ and the integration by parts, under the assumption \eqref{cond.thm.De.SVP.u0<tau0}, we have
    \begin{equation*}
    \begin{aligned}
        |I_1| 
        & \le C \int_{D^\e_{0,s+2}} (|u_0|+|\nabla u_0|) |u|^2 \\
        &\le C\| u_0  \|_{H^1_{\rm uloc}(D^\e_{0,s+2})} \|  u\|_{L^4(D^\e_{0,s+2})}^2\\
        & \le 
        C\tau_0 \int_{D^\e_{0,s+2}} |\nabla u|^2,
    \end{aligned}
    \end{equation*}
    where the Cauchy-Schwarz inequality is applied in the second line and the Sobolev-Poincar\'{e} inequality (in a long pipe) in the third. Similarly, we have the same estimate for $|I_2|$.
    
    \textbf{Estimate of $I_3$.} The estimate involving the pressure is the main technical difficulty of the proof due to the rough domains. Notice that $\eta'(x_1-s)$ is supported in $[s,s+1]$. As a preliminary fact, we first claim that there exists a John domain $D_{s,s+1}^*$ such that $D_{s,s+1}^\e \subset D_{s,s+1}^* \subset D_{s-1,s+2}^\e$, and the associated constant depends only on $L$ in Definition \ref{def.John2}.
    
    To see the claim, we first note that if $\omega_1$ and $\omega_2$ are two John domains and $\omega_1 \cap \omega_2 \neq \emptyset$, then $\omega_1 \cup \omega_2$ is also a John domain. Assume that $x^1$ and $x^2$ are the centers of $\omega_1$ and $\omega_2$, respectively, and there exists a Lipschitz curve $\gamma = \gamma(\omega_1, \omega_2)$ contained in $\omega^1 \cup \omega^2$ and connecting $x^1$ and $x^2$. Let $d_\gamma$ be the distance from $\gamma$ to $\partial(\omega_1\cup \omega_2)$. Then it is not hard to see from the geometry and Definition \ref{def.John} that the constant of the John domain $\omega_1 \cup \omega_2$ (with a center on $\gamma$) depends on the individual constants of $\omega_1$ and $\omega_2$, as well as $d_\gamma$. From this fact, we can apply the John domain assumption of $D^\e$ and choose a finite number John domains $\omega_j, \ j=1,2,\cdots, j_0$, with diameters comparable to $1$ that cover $D^\e_{s,s+1}$ and are contained in $D^\e_{s-1,s+2}$. Moreover, these John domains $\omega_j$ are well-chosen so that their centers are all connected with Lipschitz curves with a uniform lower bound of $d_\gamma$. By this construction, and defining $D_{s,s+1}^* = \cup_{j=1}^{j_0} \omega_j$, we see that $D_{s,s+1}^*$ is a John domain with a constant $L^*$ depending only on $L$. Clearly,  $D_{s,s+1}^\e \subset D_{s,s+1}^* \subset D_{s-1,s+2}^\e$ is satisfied. 
    
    Now, to estimate the integral involving the pressure in $D^\e_{s,s+1}$, let
    \begin{equation*}
        \bar{p} = \dashint_{D_{s,s+1}^*} p.
    \end{equation*}
    Then, by \eqref{def.flux} with the flux assumption $\Phi^\e[u] = 0$, we have
    \begin{equation}\label{est.pbar.flux}
        \int_{D^\e_0} u_1(x) \bar{p} \eta'(x_1 - s) \dd x 
        = \int_{D^\e_0} u_1(x) \bar{p} \tilde{\eta}'(x_1) \dd x
        = 0,
    \end{equation}
    since $\eta'(x_1-s)$ is supported in $[s,s+1]$. Consequently, using \eqref{est.pbar.flux} and the Cauchy–Schwarz inequality, we have
    \begin{equation}\label{est.pu1-1}
    \begin{aligned}
        \bigg| \int_{D^\e_0} p u_1 \eta'(x_1 - s) \dd x \bigg| & = \bigg| \int_{D^\e_0} (p - \bar{p}) u_1 \eta'(x_1 - s) \dd x \bigg| \\
        & \le C \bigg( \int_{D^\e_{s,s+1}  } | p - \bar{p}|^2 \bigg)^{1/2} \bigg( \int_{D^\e_{s,s+1}  } | u_1|^2 \bigg)^{1/2} \\
        & \le C \bigg( \int_{D^*_{s,s+1}  } | p - \bar{p}|^2 \bigg)^{1/2} \bigg( \int_{D^\e_{s,s+1}  } | \nabla u|^2 \bigg)^{1/2},
    \end{aligned}
    \end{equation}
    where we have used the Poincar\'{e} inequality in the last inequality.

    Next, observe that by $\nabla\cdot u = \nabla\cdot u_0 = 0$,  $(u,p)$ satisfies the system $-\Delta u + \nabla p = \nabla\cdot F$ with $F = -(u_0\otimes u + u\otimes u_0)$ in the John domain $D_{s,s+1}^*$. By Lemma \ref{lem.p2uF}, we have
    \begin{equation}
    \begin{aligned}
        \| p - \bar{p}\|_{L^2(D_{s,s+1}^*)} & \le C(\| \nabla u \|_{L^2(D_{s,s+1}^*)} + \| u_0\otimes u + u\otimes u_0 \|_{L^2(D_{s,s+1}^*)} ) \\
        & \le C(\| \nabla u \|_{L^2(D_{s,s+1}^*)} + \| u_0 \|_{H^1(D^\e_{s-1,s+2})} \| \nabla u \|_{L^2(D_{s-1,s+2}^\e)} ) \\
        & \le C\| \nabla u \|_{L^2(D_{s-1,s+2}^\e)},
    \end{aligned}
    \end{equation}
    where we have used the Sobolev and Poincar\'{e} inequalities in John domains in the second inequality. Inserting the last estimate into \eqref{est.pu1-1}, we obtain
    \begin{equation}
        \bigg| \int_{D^\e_0} p u_1 \eta'(x_1 - s) \dd x \bigg| \le C\int_{D^\e_{s-1,s+2}  } |\nabla u|^2
    \end{equation}
    
    Similarly, since $\eta'(x_1)$ restricted in $\{ x_1>0 \}$ is supported in $[1/2,2/3]$, we have
    \begin{equation*}
        \bigg| \int_{D^\e_0} p u_1 \tilde{\eta}'(x_1) \dd x \bigg| \le C\int_{D^\e_{0,1}  } |\nabla u|^2,
    \end{equation*}
    and thus
    \begin{equation*}
        |I_3| \le C\int_{D^\e_{0,1}  } |\nabla u|^2 + C \int_{D^\e_{s-1,s+2}  } |\nabla u|^2.
    \end{equation*}

    \textbf{Estimate of $I_4$ and $I_5$.} By the Poincar\'{e} inequality, it is easy to see that
    \begin{equation*}
        |I_4| \le C\int_{D^\e_{0,1}  } |\nabla u|^2 + C\int_{D^\e_{s,s+1}  } |\nabla u|^2
    \end{equation*}
    and that
    \begin{equation*}
        |I_5| \le \int_{D^\e_{0,1}  } |\nabla u|^2.
    \end{equation*}

    Inserting all the estimates for $I_1$--$I_5$ into \eqref{eq.Gs} gives
    \begin{equation}\label{est.Gs+}
        G(s) \le C \tau_0 \int_{D^\e_{0,s+2}} |\nabla u|^2 +  C \int_{D^\e_{s-1,s+2}  } |\nabla u|^2 + C \int_{D^\e_{0,1}  } |\nabla u|^2.
    \end{equation}
    Observing that
    \begin{equation*}
        \int_{D^\e_{0,s+2}} |\nabla u|^2 \le G(s) + \int_{D^\e_{s,s+2}} |\nabla u|^2
    \end{equation*}
    and that, for $s>2$,
    \begin{equation*}
        \int_{D^\e_{s-1,s+2}  } |\nabla u|^2 \le G(s+2) - G(s-2),
    \end{equation*}
    we have
    \begin{equation*}
        G(s) \le C\tau_0 G(s) + C(G(s+2) - G(s-2)) + C \int_{D^\e_{0,1}  } |\nabla u|^2.
    \end{equation*}
    Assume that $\tau_0>0$ is small so that $C\tau_0\le 1/2$. Then we have
    \begin{equation*}
        G(s-2) \le G(s) \le C(G(s+2) - G(s-2)) + C \int_{D^\e_{0,1}  } |\nabla u|^2.
    \end{equation*}
    Since $s>2$ is arbitrary, we replace $s-2$ by $s$ to see that
    \begin{equation*}
        G(s) \le \frac{C}{C+1} G(s+4) +  \int_{D^\e_{0,1}  } |\nabla u|^2.
    \end{equation*}
    Set $\rho = C/(C+1)<1$ and iterate this inequality. Then we obtain, for all $k\ge 1$,
    \begin{equation*}
    \begin{aligned}
        G(s) 
        &\le 
        \rho^k G(s+4k) 
        + \Big(\sum_{j=0}^{k-1} \rho^j \Big) \int_{D^\e_{0,1}} |\nabla u|^2 \\
        &\le 
        e^{k \log \rho} G(s+4k) + \frac{1}{1-\rho} \int_{D^\e_{0,1}} |\nabla u|^2.
    \end{aligned}
    \end{equation*}
    Now if $0<b<-(1/4) \log \rho$, then the assumption \eqref{cond.thm.De.SVP.exp} implies that there exists a subsequence of integers $k$ such that $e^{k \log \rho} G(s+4k)$ converges to zero. It follows that
    \begin{equation*}
        G(s)
        \le 
        \frac{1}{1-\rho} \int_{D^\e_{0,1}  } |\nabla u|^2.
    \end{equation*}
    Finally, taking $s\to \infty$ in this inequality, we obtain the desired estimate \eqref{est.Finite Energy}.
\end{proof}

\begin{proofx}{Theorem \ref{thm.De.SVP}}
    This is a simple corollary of Lemma \ref{lem.exp2bdd} and Remark \ref{rem.thm.De.SVP}. Let
    \begin{equation*}
        H(s) = \int_{D^\e_s} |\nabla u|^2,
        \quad s\ge0.
    \end{equation*}
    By Lemma \ref{lem.exp2bdd} and Remark \ref{rem.thm.De.SVP}, we have
    \begin{equation*}
        H(s) \le C(H(s) - H(s+1)),
    \end{equation*}
    which gives
    \begin{equation*}
        H(s+1) \le \frac{C-1}{C} H(s).
    \end{equation*}
    Since $(C-1)/C < 1$, a simple iteration leads to \eqref{est.SVP} with $c = -\log((C-1)/C)$.
\end{proofx}

Theorem \ref{thm.De.SVP} implies the uniqueness of the ``bounded'' solution of the linearized Navier-Stokes system. Particularly, it implies the uniqueness of the Stokes flow (i.e., the solution when $u_0 = 0$) in $D^\e$ in suitable function spaces, as well as the exponential decay estimate of the Green's function in Section \ref{Green}. 

\begin{corollary}\label{coro.Stokes.uniqueness}
    Let $u_0 \in H^1_{\rm uloc}(D^\e_0)^3$ with $\nabla\cdot u_0 = 0$. Then there exists $\tau_0>0$ such that if 
    \begin{equation}\label{cond.Stokes.Uniqueness}
        \| u_0 \|_{H^1_{\rm uloc}(D^\e_0)} \le \tau_0,
    \end{equation}
    then the linearized Navier-Stokes system around $u_0$
    \begin{equation}\label{eq.Stokes.Uniqueness}
        \left\{
        \begin{array}{ll}
    -\Delta u + u_0\cdot \nabla u + u\cdot \nabla u_0 + \nabla p 
    =0 &\mbox{in}\ D^\e,\\
    \nabla\cdot u = 0 &\mbox{in}\ D^\e,\\
    \Phi^\e[u] = 0,\\
    u = 0 &\mbox{on}\ \Gamma^\e
    \end{array}\right.
    \end{equation}
    only has a trivial solution $u = 0$ in $H^1_{\rm uloc}(D^\e)^3$.
\end{corollary}

\begin{remark}\label{rem.coro.Stokes.uniqueness}
    Consequently, under the assumption \eqref{cond.Stokes.Uniqueness}, the family $\{u\}$ of solutions of
    \begin{equation*}
        \left\{
        \begin{array}{ll}
    -\Delta u + u_0\cdot \nabla u + u\cdot \nabla u_0 + \nabla p 
    =0 &\mbox{in}\ D^\e,\\
    \nabla\cdot u = 0 &\mbox{in}\ D^\e,\\
    u = 0 &\mbox{on}\ \Gamma^\e
    \end{array}\right.
    \end{equation*}
    in $H^1_{\rm uloc}(D^\e)^3$ forms the vector space where each of the elements is characterized by its flux. This fact will play a crucial role when ``normalizing" boundary layers in Section \ref{Normal}.
\end{remark}

\begin{proof}
Choose $\tau_0$ to be the same constant as in Theorem \ref{thm.De.SVP} and suppose that $u$ in $H^1_{\rm uloc}(D^\e)^3$ solves \eqref{eq.Stokes.Uniqueness}. To prove the statement, it suffices to show that, for all $k\in\R$, 
\begin{equation*}
    \int_{D^\e_{k,k+1}} |\nabla u|^2=0.
\end{equation*}
We may only consider the case $k=0$ for simplicity. The general case is no harder to prove.

Since $u_0$ satisfies \eqref{cond.Stokes.Uniqueness} and $u$ belongs to $H^1_{\rm uloc}(D^\e)^3$, Theorem \ref{thm.De.SVP} and Remark \ref{rem.thm.De.SVP} imply that, for any $t>0$,
\begin{align*}
\begin{split}
    \int_{D^\e_{0,1}} |\nabla u|^2
    \le
    \int_{D^\e\cap\{-\infty<x_1<1\}} |\nabla u|^2
    \le
    C e^{-c t} \int_{D^\e_{t,t+1}} |\nabla u|^2 \le Ce^{-c t} \| u \|_{H^1_{\rm uloc}(D^\e)}^2.
\end{split}
\end{align*}
Note that the left-hand side is independent of $t$. Letting $t\to \infty$, we have $\| \nabla u \|_{L^2(D^\e_{0,1})} = 0$.
\end{proof}

The next corollary is crucial for us. Particularly the localization estimate \eqref{est.MStokes.Localization} derived from the Saint-Venant's principle is the key to convert the local energy/fluctuation of the source term to that of the solution.

\begin{corollary}\label{coro.LP.ExistenceUniqueness}
    Let $u_0 \in H^1_{\rm uloc}(D^\e)^3$ with $\nabla\cdot u_0 = 0$ and $F \in L^2_{\rm uloc}(D^\e)^{3\times3}$. Then there exist $\tau_0,c>0$ such that if
    \begin{equation}
        \| u_0 \|_{H^1_{\rm uloc}(D^\e)} \le \tau_0,
    \end{equation}
    then the linearized Navier-Stokes system around $u_0$
    \begin{equation}\label{eq.Stokes+u0}
        \left\{
        \begin{array}{ll}
        -\Delta w + u_0\cdot \nabla w + w\cdot \nabla u_0 + \nabla \pi = \nabla\cdot F &\mbox{in}\ D^\e,\\
        \nabla\cdot w = 0&\mbox{in}\ D^\e,\\
        \Phi^\e[w] = 0,\\
        w = 0&\mbox{on}\ \Gamma^\e
    \end{array}\right.
    \end{equation}
    has a unique solution $w\in H^1_{\rm uloc}(D^\e)^3$ such that, for all $k\in \R$,
    \begin{equation}\label{est.MStokes.Localization}
        \| \nabla w \|_{L^2(D^\e_{k,k+1})} \le C\int_{\R} \| F \|_{L^2(D^\e_{s,s+1})} e^{-c |k-s|} \dd s.
    \end{equation}
    In particular, we have
    \begin{equation}
        \| \nabla w \|_{L^2_{\rm uloc}(D^\e)} \le C\| F \|_{L^2_{\rm uloc}(D^\e)}.
    \end{equation}
\end{corollary}
\begin{proof}
    Choose $\tau_0$ to be the same constant as in Theorem \ref{thm.De.SVP}. We first write $F = \sum_{j\in \Z} F_j$ with $F_j = F \mathbf{1}_{ \{ j\le x_1 < j+1\} } $ where $\mathbf{1}_E$ refers to the characteristic function on the set $E$. Fix $j\in\Z$. Clearly, $F_j \in L^2(D^\e)^{3\times3}$. By the Lax-Milgram theorem, the linear system \eqref{eq.Stokes+u0} with $F$ replaced by $F_j$ has a unique solution $w_j$ in the Hilbert space $H^1_{0}(D^\e)^3$ for any $j$, under additional smallness on $\tau_0$ if necessary. Moreover, the energy estimate yields
    \begin{equation*}
        \| \nabla w_j \|_{L^2(D^\e)} 
        \lesssim 
        \| F_j \|_{L^2(D^\e_{j,j+1})}
        =
        \| F \|_{L^2(D^\e_{j,j+1})}.
    \end{equation*}
    Then Theorem \ref{thm.De.SVP} and Remark \ref{rem.thm.De.SVP} imply that, for some $c>0$,
    \begin{equation*}
        \| \nabla w_j \|_{L^2(D^\e_{k,k+1})} \lesssim
        e^{-c |k-j|} \| \nabla w_j \|_{L^2(D^\e_{j,j+1})} 
        \lesssim
        e^{-c |k-j|} \| F \|_{L^2(D^\e_{j,j+1})},
        \quad k\in\R.
    \end{equation*}
    Set $w = \sum_{j\in \Z} w_j$. Then we have
    \begin{equation*}
    \begin{aligned}
        \| \nabla w \|_{L^2(D^\e_{k,k+1})} 
        &\lesssim
        \sum_{j\in \Z} e^{-c |k-j|} \| F \|_{L^2(D^\e_{j,j+1})}  \\
        &\lesssim
        \int_{\R} \| F \|_{L^2(D^\e_{s,s+1})} 
        e^{-c |k-s|}
        \dd s,
        \quad k\in\R.
    \end{aligned}
    \end{equation*}
    Note that the right-hand side is finite since $F \in L^2_{\rm uloc}(D^\e)^{3\times3}$. Thus $w$ is well-defined in the entire $D^\e$ and satisfies \eqref{est.MStokes.Localization} for all $k\in \R$.
    Moreover, $w$ is a solution of \eqref{eq.Stokes+u0} belonging to $H^1_{\rm uloc}(D^\e)^3$ by construction, which is unique due to Corollary \ref{coro.Stokes.uniqueness}. 
\end{proof}

\begin{proposition}\label{prop.NP.existence}
    Let $u_0 \in H^1_{\rm uloc}(D^\e)^3$ with $\nabla\cdot u_0 = 0$ and $F \in L^2_{\rm uloc}(D^\e)^{3\times3}$. Then there exist $\tau_0,\delta_0>0$ such that if
    \begin{equation}\label{cond.prop.NP.existence}
        \| u_0 \|_{H^1_{\rm uloc}(D^\e)} \le \tau_0, 
        \qquad 
        \| F \|_{L^2_{\rm uloc}(D^\e_0)} \le \delta_0,
    \end{equation}
    then the perturbed Navier-Stokes system around $u_0$
    \begin{equation}\label{eq.NS+u0}
        \left\{
        \begin{array}{ll}
        -\Delta w + w\cdot \nabla w + u_0\cdot \nabla w + w\cdot \nabla u_0 + \nabla \pi = \nabla\cdot F &\mbox{in}\ D^\e,\\
        \nabla\cdot w = 0&\mbox{in}\ D^\e,\\
        \Phi^\e[w] = 0,\\
        w = 0&\mbox{on}\ \Gamma^\e
        \end{array}\right.
    \end{equation}
    has a solution $w\in H^1_{\rm uloc}(D^\e)^3$ such that
    \begin{equation}\label{est.CMT}
        \| \nabla w \|_{L^2_{\rm uloc}(D^\e)} \le C\| F \|_{L^2_{\rm uloc}(D^\e)}.
    \end{equation}
\end{proposition}
\begin{proof}
    Choose $\tau_0$ to be the same constant as in Corollary \ref{coro.LP.ExistenceUniqueness}. The statement will follow from Corollary \ref{coro.LP.ExistenceUniqueness} and the contraction mapping theorem. Consider the nonlinear mapping
    \[
        T = T_F: H_{\rm uloc}^1(D^\e)^3 \to H_{\rm uloc}^1(D^\e)^3
    \]
    such that $w := T(u)$ is the solution of \eqref{eq.Stokes+u0} with $F$ replaced by $-u\otimes u + F$. The Gagliardo-Nirenberg inequality and Corollary \ref{coro.LP.ExistenceUniqueness} verify that $T$ is well-defined. Moreover, 
    \begin{equation}\label{est1.prf.prop.NP.existence}
        \| w \|_{H^1_{\rm uloc}(D^\e)} 
        \le C( \| u \|_{H^1_{\rm uloc}(D^\e)}^2  + \| F\|_{L^2_{\rm uloc}(D^\e)} ).
    \end{equation}
    Defining
    \begin{equation*}
        \mathcal{X}_{\delta_1} = \{ u \in H_{\rm uloc}^1(D^\e)^3~|~ u|_{\Gamma^\e}=0,\mkern9mu\| u\|_{H_{\rm uloc}^1(D^\e)} \le \delta_1 \},
    \end{equation*}
    we will show that there exist $\delta_0,\delta_1>0$ such that $T$ is a contraction in $\mathcal{X}_{\delta_1}$ for $F$ satisfying \eqref{cond.prop.NP.existence}. To make sure that $T$ maps $\mathcal{X}_{\delta_1}$ to itself, in view of \eqref{est1.prf.prop.NP.existence}, we require that
    \begin{equation*}
        C(\delta_1^2 + \delta_0) \le \delta_1.
    \end{equation*}
    This can be achieved if $\delta_1 \le 1/(2C)$ and $\delta_0 = \delta_1^2$. Next, let $u^1, u^2 \in \mathcal{X}_{\delta_1}$ and $w^1 = T(u^1)$ and $w^2 = T(u^2)$. Then, by Corollary \ref{coro.LP.ExistenceUniqueness}, we see that
    \begin{equation*}
    \begin{aligned}
        \| w^1 - w^2 \|_{H^1_{\rm uloc}(D^\e)} & \le C\| u^1\otimes u^1 - u^2\otimes u^2 \|_{L^2_{\rm uloc}(D^\e)} \\
        & \le C(\| u^1\|_{H^1_{\rm uloc}(D^\e)} + \| u^2\|_{H^1_{\rm uloc}(D^\e)} )\| u^1 - u^2 \|_{H^1_{\rm uloc}(D^\e)} \\
        & \le 2C\delta_1 \| u^1 - u^2 \|_{H^1_{\rm uloc}(D^\e)}.
    \end{aligned}
    \end{equation*}
    Hence $T$ is a contraction in $\mathcal{X}_{\delta_1}$ if $\delta_1 < 1/(2C)$. Consequently, under the assumption $\delta_1 < 1/(2C)$ and $\delta_0 = \delta_1^2$, we have a fixed point $w\in \mathcal{X}_{\delta_1}$ such that $w = T(w)$. It is routine to check that, letting $w^0 = 0$ and $w^k = T(w^{k-1})$, we have $w = \lim_{k\to \infty} w^{k}$ and \eqref{est.CMT}.
    \end{proof}

Next, given $u_0 \in H^1_{\rm uloc}(D^\e_0)^3$ satisfying $\nabla\cdot u_0 = 0$, consider the perturbed Navier-Stokes system around $u_0$
\begin{equation}\label{eq.NS.w}
\left\{
\begin{array}{ll}
-\Delta w + w\cdot \nabla w + u_0\cdot \nabla w + w\cdot \nabla u_0 + \nabla \pi = 0 &\mbox{in}\ D^\e_0,\\
\nabla\cdot w = 0&\mbox{in}\ D^\e_0,\\
\Phi^\e[w] = 0,\\
w = 0&\mbox{on}\ \Gamma^\e_0.
\end{array}\right.
\end{equation}
Notice that this system arises when one takes the difference of two solutions $(u_0,p_0), (u_1,p_1)$ of \eqref{eq.NS} with the same flux and puts $w=u_0-u_1$ and $\pi=p_0-p_1$.

\begin{theorem}\label{thm.SVP.w}
    Let $u_0 \in H^1_{\rm uloc}(D_0^\e)^3$ with $\nabla\cdot u_0 = 0$ and $w\in H^1_{\rm loc}(\overline{D_0^\e})^3$ be a weak solution of \eqref{eq.NS.w}. Then there exists $\tau_0 >0$ such that if
    \begin{equation}\label{cond.thm.SVP.w.u0<tau0}
        \| u_0 \|_{H^1_{\rm uloc}(D^\e_0)} \le \tau_0
    \end{equation}
    and
    \begin{equation}\label{cond.thm.SVP.w.t3growth}
        \liminf_{t\to \infty} t^{-3} \int_{D^\e_{0,t} } |\nabla w|^2 = 0,
    \end{equation}
    then there exist $C,c>0$ such that, for all $t\ge0$,
    \begin{equation}\label{est.SVP.w}
    \int_{D^\e_t} |\nabla w|^2 \le C e^{-c t} \int_{D^\e_{0,1}} |\nabla w|^2.
    \end{equation}
    The constants $C,c$ depend on $\| \nabla w \|_{L^2(D^\e_{0,1})}$.
\end{theorem}

\begin{remark}\label{rem.thm.SVP.w}
    As with Remark \ref{rem.thm.De.SVP}, the similar result with obvious modifications to Theorem \ref{thm.SVP.w} as well as to Lemma \ref{lem.Qubic2bdd} below holds if $D^\e_0$ is replaced by $D^\e_{k}$ or $D^\e_{-\infty, k}$ for any $k\in \R$.
\end{remark}

\begin{lemma}\label{lem.Qubic2bdd}
    Under the assumptions of Theorem \ref{thm.SVP.w}, we have
    \begin{equation}\label{est.Qubic2bdd}
        \int_{D_0^\e} |\nabla w|^2 \le C \int_{D_{0,1}^\e} |\nabla w|^2 + C \bigg( \int_{D_{0,1}^\e} |\nabla w|^2 \bigg)^{3/2}.
    \end{equation}
\end{lemma}

\begin{proof}
    Let $\eta, \tilde{\eta}$ be the same as in the proof of Lemma \ref{lem.exp2bdd} and define 
    \begin{equation}
        G(s) = \int_{D^\e_0} |\nabla w|^2 \eta(x_1 -s) \dd x, \quad s>0.
    \end{equation}
    Using the system \eqref{eq.NS.w} and the integration by parts, we obtain
    \begin{equation}\label{eq.Gs4w}
    \begin{aligned}
        G(s) 
        & = \int_{D^\e_0} |\nabla w|^2 (\eta(x_1-s)-\tilde{\eta}(x_1)) \dd x + \int_{D^\e_0} |\nabla w|^2 \tilde{\eta}(x_1) \dd x \\
        & = 
        - \int_{D^\e_0} (w\cdot \nabla w 
        + u_0\cdot \nabla w + w\cdot \nabla u_0)\cdot w (\eta(x_1-s)-\tilde{\eta}(x_1)) \dd x  \\
        &\quad 
        + \int_{D^\e_0} \pi w_1 (\eta'(x_1 - s) - \tilde{\eta}'(x_1)) \dd x \\
        &\quad 
        - \int_{D^\e_0} w\cdot \partial_1 w (\eta'(x_1 - s) - \tilde{\eta}'(x_1)) \dd x + \int_{D^\e_0} |\nabla w|^2 \tilde{\eta}(x_1) \dd x.
    \end{aligned}
    \end{equation}

    The estimate for $G(s)$ is the same as in the proof of Lemma \ref{lem.exp2bdd} (see \eqref{est.Gs+}), except for the terms involving the nonlinearity $(w\cdot \nabla w)\cdot w$ or the pressure. In fact, it follows from $(w\cdot \nabla w)\cdot w = (1/2) w\cdot \nabla(|w|^2)$ and $\nabla\cdot w = 0$, we have
    \begin{equation}\label{est.Gs-4}
    \begin{aligned}
        & \bigg| \int_{D^\e_0} (w\cdot \nabla w)\cdot w (\eta(x_1-s)-\tilde{\eta}(x_1)) \dd x \bigg| \\
        & = \bigg| \int_{D^\e_0}  \frac12 w_1 |w|^2 (\eta'(x_1 - s) - \tilde{\eta}'(x_1)) \dd x \bigg| \\
        & \le C\bigg\{ \bigg( \int_{D^\e_{s,s+1}} |\nabla w|^2 \bigg)^{3/2} + \bigg( \int_{D^\e_{0,1}} |\nabla w|^2 \bigg)^{3/2} \bigg\}.
    \end{aligned}
    \end{equation}
    For the term involving the pressure, we apply the similar argument as $I_3$ in Lemma \ref{lem.exp2bdd}. Let $D_{s,s+1}^*$ be the John domain constructed in Lemma \ref{lem.exp2bdd}.
    Recall that $(w,\pi)$ satisfies $-\Delta w + \nabla \pi = \nabla\cdot F$ with $F = -(w \otimes w + u_0\otimes w + w\otimes u_0)$ in $D^*_{s,s+1}$. Let $\bar{\pi} = \dashint_{D^*_{s,s+1}} \pi$. Then by Lemma \ref{lem.p2uF}, we have
    \begin{equation*}
    \begin{aligned}
        & \| \pi - \bar{\pi}\|_{L^2(D_{s,s+1}^*)} \\
        & \le C(\| \nabla w \|_{L^2(D_{s,s+1}^*)} + \| w\otimes w + u_0\otimes w + w \otimes u_0 \|_{L^2(D_{s,s+1}^*)} ) \\
        & \le C(\| \nabla w \|_{L^2(D_{s,s+1}^*)} + \| \nabla w \|_{L^2(D_{s-1,s+2}^\e)}^2 + \| u_0 \|_{H^1(D^\e_{s-1,s+2})} \| \nabla w \|_{L^2(D_{s-1,s+2}^\e)} ) \\
        & \le C\| \nabla w \|_{L^2(D_{s-1,s+2}^\e)} + C\| \nabla w \|_{L^2(D_{s-1,s+2}^\e)}^2.
    \end{aligned}
    \end{equation*}
    This yields
    \begin{equation}
    \begin{aligned}
        \bigg| \int_{D^\e_0} \pi w_1 \eta'(x_1 - s) \dd x  \bigg| & = \bigg| \int_{D^\e_0} (\pi-\bar{\pi}) w_1 \eta'(x_1 - s) \dd x  \bigg|
        \\ & \le C \bigg\{ \int_{D_{s-1,s+2}^\e} |\nabla w|^2 + \bigg( \int_{D_{s-1,s+2}^\e} |\nabla w|^2 \bigg)^{3/2} \bigg\}.
    \end{aligned}
    \end{equation}
    Similarly, one has
    \begin{equation}
        \bigg| \int_{D^\e_0} \pi w_1 \tilde{\eta}'(x_1) \dd x \bigg| \le C\bigg\{ \int_{D_{0,1}^\e} |\nabla w|^2 + \bigg( \int_{D_{0,1}^\e} |\nabla w|^2 \bigg)^{3/2} \bigg\}.
    \end{equation}
    As a result, we have
    \begin{equation}\label{est.Gs-5}
    \begin{aligned}
        G(s) & \le C\tau_0 G(s) + C\int_{D^\e_{s-1,s+2}} |\nabla w|^2 + C \int_{D^\e_{0,1}} |\nabla w|^2 \\
        & \quad + C\bigg\{ \bigg( \int_{D^\e_{s,s+1}} |\nabla w|^2 \bigg)^{3/2} + \bigg( \int_{D^\e_{0,1}} |\nabla w|^2 \bigg)^{3/2} \bigg\}.
    \end{aligned}
    \end{equation}

    Set
    \begin{equation}\label{def.M}
        M = C \int_{D_{0,1}^\e} |\nabla w|^2 + C \bigg( \int_{D_{0,1}^\e} |\nabla w|^2 \bigg)^{3/2}.
    \end{equation}
    From the same manner as in the proof of Lemma \ref{lem.exp2bdd}, we see that, for $s>2$,
    \begin{equation*}
        G(s) \le C\tau_0 G(s) + C(G(s+2) - G(s-2)) + C(G(s+2) - G(s-2))^{3/2} + M.
    \end{equation*}
    Assume that $\tau_0>0$ is small so that $C\tau_0\le 1/2$. Then we have
    \begin{equation*}
        G(s-2) \le G(s) \le C(G(s+2) - G(s-2)) + C(G(s+2) - G(s-2))^{3/2} + M.
    \end{equation*}
    Since $s>2$ is arbitrary, we replace $s-2$ by $s$ to see that
    \begin{equation}\label{est.Gs.iteration}
        G(s) \le C_1 (G(s+4) - G(s)) + C_2 (G(s+4) - G(s))^{3/2} + M.
    \end{equation}
    The assumption \eqref{cond.thm.SVP.w.t3growth} implies $\liminf_{k\to \infty} k^{-3} G(4k) = 0$. Thus, by Lemma \ref{lem.NonlinearIteration} below (with $b_k = G(4k)$), we have $G(4k) \le M$ for all $k\ge 1$ and the desired estimate \eqref{est.Qubic2bdd}.
\end{proof}

\begin{lemma}[ {\cite[Lemma 4.1]{P02}} ] \label{lem.NonlinearIteration}
    Let $\{ b_k \}_{k\ge  1}$ be a nondecreasing nonnegative sequence. Suppose that there are constants $A,B$ and $M$, such that, for all $k\ge 1$,
    \begin{equation}
        b_k \le A(b_{k+1} - b_k) + B(b_{k+1} - b_k)^{3/2} + M.
    \end{equation}
    and
    \begin{equation}
        \liminf_{k\to \infty} k^{-3} b_k = 0.
    \end{equation}
    Then $b_k \le M$ for all $k\ge 1$.
\end{lemma}

Now we prove Theorem \ref{thm.SVP.w}.

\begin{proofx}{Theorem \ref{thm.SVP.w}}
    By Lemma \ref{lem.Qubic2bdd}, we see that, for any $s\ge0$,
    \begin{equation}\label{est.Dw.Qs}
        \| \nabla w \|_{L^2(D^\e_{s,s+1})}
        \le
        \| \nabla w \|_{L^2(D^\e_{0})}
        \le
        M,
    \end{equation}
    where $M$ is defined in \eqref{def.M}. Let
    \begin{equation*}
        H(s) = \int_{D^\e_s} |\nabla w|^2,
        \quad s\ge0.
    \end{equation*}
    By Lemma \ref{lem.Qubic2bdd} and Remark \ref{rem.thm.SVP.w} with \eqref{est.Dw.Qs}, putting $C_M = C(1+M^{1/2})$, we have
    \begin{equation*}
        H(s) \le C_M ( H(s) - H(s+1)),
    \end{equation*}
    which gives
    \begin{equation*}
        H(s+1) \le \frac{C_M - 1}{C_M} H(s).
    \end{equation*}
    Since $(C_M-1)/C_M < 1$, a simple iteration leads to \eqref{est.SVP.w} with $c = -\log((C_M-1)/C_M)$.
\end{proofx}

Theorem \ref{thm.SVP.w} implies the uniqueness of small solutions of the Navier-Stokes system \eqref{eq.NS}. The following existence and uniqueness results will be proved later in Section \ref{prf.thma} after the existence of solutions to \eqref{eq.NS} with small flux is established.

\begin{theorem}\label{thm.NS.uniqueness}
    There exists sufficiently small $\phi_0>0$ such that for all $|\phi|<\phi_0$, the system \eqref{eq.NS} has a unique solution $u^\e \in H^1_{\rm uloc}(D^\e)^3$.
\end{theorem}

    \subsection{Large-scale boundary regularity}\label{LSLipschitz}
The regularity estimates play important roles in constructing and estimating the Green's function and boundary layers. In \cite{Ger09}, to obtain the quantitative sensitivity estimate of the boundary layers with respect to the boundary perturbation, G\'{e}rard-Varet established the uniform H\"{o}lder estimate for the Green's function over bumpy $C^{1}$ boundaries essentially by the compactness method originating from Avellaneda-Lin \cite{AL87}. The uniform Lipschitz estimate via compactness method over bumpy boundaries were developed in \cite{KP15,KP18} for elliptic equations in homogenization setting and \cite{HP20} for steady Navier-Stokes system. In a seires of papers \cite{Z21,GZ22,HZ23,HPZ-21}, a quantitative method originating from \cite{AS16} in elliptic homogenization was generalized to obtain the large-scale Lipschitz estimate and even higher order regularity over oscillating boundaries with minimal regularity assumption on the boundaries (which allows for fractals and cusps). Particularly in \cite{HPZ-21}, the large-scale boundary Lipschitz estimate was established for the Navier-Stokes system (including Stokes system) over a rough hyperplane of John type. The large-scale Lipschitz estimate in \cite{HPZ-21} can be generalized to the case of bumpy cylinders considered in this paper in the same spirit of \cite{Z21,GZ22} for curved boundaries. Recall that we do not need the boundary layers in the proof of large-scale Lipschitz estimate.

In principle, the regularity of solutions exhibits different behaviors distinguished between interior points and points close to the rough boundary. Recall that $r_x = (x_2^2 + x_3^2)^{1/2}$ is the distance from $x$ to the $x_1$-axis. Recall that $D_1\subset D^\e \subset D_{1+\e}$ and assume $B_{r,+}(x) = D^\e \cap B_r(x)$. Suppose that $(u,p)$ satisfies
\begin{equation}\label{eq.localStokes}
\left\{
\begin{array}{ll}
-\Delta u+\nabla p= 0 &\mbox{in}\ B_{1,+}(x),\\
\nabla\cdot u=0&\mbox{in}\ B_{1,+}(x),\\
u = 0&\mbox{on}\ \Gamma^\e\cap B_1(x).
\end{array}\right.
\end{equation}
Then we describe two cases:
\begin{itemize}
    \item For $x\in D^\e$ with $r_x>1-\e$, since the point $x$ is very close to the boundary, we can only show a large-scale regularity for the solutions of the Stokes system. That is, for any $\e<r<1/2$, we have
    \begin{equation}\label{est.bdry.Lip}
        \bigg( \dashint_{B_{r,+}(x)} |\nabla u|^2 \bigg)^{1/2} + \bigg( \dashint_{B_{r,+}(x)} |p - \dashint_{B_{1/2,+}(x)} p |^2 \bigg)^{1/2} \le C \bigg( \dashint_{B_{1,+}(x)} |\nabla u|^2 \bigg)^{1/2},
    \end{equation}
    where $C$ is independent of $\e$ and $r$.
    
    \item For $x\in D^\e $ with $r_x<1-\e$, we can apply a pointwise interior estimate followed by a large-scale estimate in the first case if the region intersects with the boundary to get an estimate like
    \begin{equation}\label{est.int.Lip}
        |\nabla u(x)| + |p(x) - \dashint_{B_{1/2,+}(x)} p| \le C \bigg( \dashint_{B_{1,+}(x)} |\nabla u|^2 \bigg)^{1/2}.
    \end{equation}
    
\end{itemize}
In the above estimates, it does not matter how $B_{r,+}(x)$ intersects the boundary. We refer to \cite[Theorem A]{HPZ-21} for a proof of \eqref{est.bdry.Lip}. Some straightforward modifications are needed as the system \eqref{eq.localStokes} is linear and the boundary is a perturbation of a smooth curved surface, while the system considered in \cite{HPZ-21} is nonlinear and the boundary is a perturbation of a plane.

    \subsection{Green's function of Stokes system}\label{Green}

One of the main ingredients in this paper is the Green's function.
Formally, the Green's function (matrix-valued) is the solution of, denoting by ,
\begin{equation}\label{eq.GreenFunc}
\left\{
\begin{array}{ll}
-\Delta G(\cdot,y)+\nabla \Pi(\cdot,y)= \delta_y I&\mbox{in}\ D^\e,\\
\nabla\cdot G(\cdot,y)=0&\mbox{in}\ D^\e,\\
\Phi^\e[G(\cdot,y)] = 0, & \\
G(\cdot,y) = 0&\mbox{on}\ \Gamma^\e
\end{array}\right.
\end{equation}
for each $y\in D^\e$, where $\delta_y$ denotes the Dirac delta supported at $y$. The flux condition
$$ \Phi^\e[G(\cdot,y)] = \int_{\Sigma^\e_t} e_1\cdot G(\cdot, y) = 0 $$
ensures the uniqueness. Otherwise, even in the case of $D^0$, one can add arbitrary Hagen-Poiseuille flow \eqref{eq.HPflow1} to the Green's function. Due to the flux condition, the Saint-Venant's principle applies to the Green's function. 

In terms of continuity with respect to $y$ variable, we need the following condition: suppose that $f\in C_0^\infty(D^\e)^3$ and let $(u,p)$ be the weak solution of
\begin{equation}\label{eq.force=f}
    \left\{
    \begin{array}{ll}
    -\Delta u+\nabla p= f&\mbox{in}\ D^\e,\\
    \nabla\cdot u=0&\mbox{in}\ D^\e,\\
    \Phi^\e[u] = 0, & \\
    u = 0&\mbox{on}\ \Gamma^\e.
    \end{array}\right.
\end{equation}
Then we have
\begin{equation}\label{eq.GreenFormula}
    u(x) = \int_{D^\e} G(x,y)f(y) \dd y, \qquad p(x) = \int_{D^\e} \Pi(x,y)\cdot f(y) \dd y.
\end{equation}
Moreover, it can be shown that $G(x,y)$ satisfies the symmetry property $G(x,y) = G(y,x)^T$ for any $x,y\in D^\e$.

Applying the large-scale estimates to the Green's function, we have the following theorem. Recall that $\delta(x) = 1-r_x$ for $x\in D^0 = \{ r<1 \}$. We use $\nabla_x G(\cdot,\cdot)$ and $\nabla_y G(\cdot,\cdot)$ to denote the derivatives with respect to the first and second variables of $G(\cdot,\cdot)$, respectively.
    \begin{theorem}\label{thm.Green}
    For $x\in D^\e$ with $r_x\le 1-\e$, the following hold.
    \begin{enumerate}[(1)]
    \item\label{item1.thm.Green}
    For $y\in D^\e$ with $r_y\le 1-\e$, we have
    \begin{equation}\label{est.G.interior}
        |G(x,y)| \le Ce^{-c |x_1-y_1|} \min \bigg\{ \frac{1}{|x-y|}, \frac{\delta(y)}{|x-y|^2}, \frac{\delta(x)}{|x-y|^2}, \frac{\delta(x)\delta(y)}{|x-y|^3} \bigg\}
    \end{equation}
    and 
    \begin{equation}\label{est.Green-2}
        |\nabla_x G(x,y)| \le Ce^{-c |x_1-y_1|} \min \bigg\{ \frac{1}{|x-y|^2}, \frac{\delta(y)}{|x-y|^3} \bigg\}.
    \end{equation}
    \item\label{item2.thm.Green}
    For $y\in D^\e$ with $r_y>1-\e$ and $|x-y|>10\e$, we have
    \begin{equation}\label{est.G.bdry}
        \bigg( \dashint_{B_{\e,+}(y)} |G(x,\cdot)|^2 \bigg)^{1/2} 
        \le Ce^{-c |x_1-y_1|} \min \bigg\{ \frac{\e}{|x-y|^2},  \frac{\delta(x) \e}{|x-y|^3} \bigg\}
    \end{equation}
    and
    \begin{equation}\label{est.Grenn-4}
    \begin{aligned}
        \bigg( \dashint_{B_{\e,+}(y)} |\nabla_x G(x,\cdot)|^2 \bigg)^{1/2} 
        \le Ce^{-c |x_1-y_1|}  \frac{\e}{|x-y|^3}.
    \end{aligned}
    \end{equation}
    Moreover, we have
    \begin{equation}\label{est.DyG.bdry}
        \bigg( \dashint_{B_{\e,+}(y)} |\nabla_y G(x,\cdot)|^2 \bigg)^{1/2} \le C e^{-c |x_1-y_1|} \frac{\delta(x)}{|x-y|^3}.
    \end{equation}
    \end{enumerate}
    \end{theorem}

\begin{proof}
    First consider the case $|x_1 - y_1| \le 2$. For this case, the proof of the desired estimates are almost the same as \cite[Appendix B]{HPZ-21} and therefore we skip the detailed proof. For the case $|x_1 - y_1| > 2$, we need to extract the exponential decay of the Green's function from the Saint-Venant's principle. Observe that in the latter case $|x-y|$ in the denominator can be ignored since $ e^{-c|x_1 - y_1|} \lesssim |x-y|^{-3}$.

    Let us start with an estimate of the mixed derivatives of the Green's function. Actually, by the duality argument as in the proof of \cite[Proposition B.3]{HPZ-21}, for fixed $x\in D^\e$ with $x_1 = 0$, we have $\| \nabla_x \nabla_y G(x, \cdot) \|_{L^2(D^\e_{1,2})} \le C$ where $C$ is a universal constant depending only on the John domain condition. Due to the symmetry $G(x,y) = G(y,x)^T$, we have 
    \begin{equation}\label{est.DxDyG-0}
        \| \nabla_x \nabla_y G(\cdot, y) \|_{L^2(D^\e_{1,2})} \le C
    \end{equation}
    for a fixed $y\in D^\e$ with $y_1 = 0$.

    In view of the $x_1$-translation invariance of the estimates to be proved, without loss of generality, we may assume $x_1>2$ and $y_1 = 0$. Note that $(\nabla_y G(x,y), \nabla_y \Pi(x,y))$ still satisfies the system \eqref{eq.Stokes.SemiDe} in $x$ variable. It follows from Theorem \ref{thm.De.SVP}, Remark \ref{rem.thm.De.SVP} and \eqref{est.DxDyG-0} that there exists a universal constant $c>0$ such that, for any $t\ge 1$,
    \begin{equation}
        \| \nabla_x \nabla_y G(\cdot,y) \|_{L^2(D_t^\e)} \le C e^{-c t}.
    \end{equation}
    Hence, the large-scale Lipscthiz estimates in Section \ref{LSLipschitz} gives, if $x_1>2$ and $r_x\le 1-\e$,
    \begin{equation}
    \begin{aligned}
        &|\nabla_x \nabla_y G(x,y)| + |\nabla_y \Pi(x,y) - \dashint_{B_{1/2,+}(y)} \nabla_y \Pi(z,y) \dd z| \\
        & \le C \bigg( \dashint_{B_{1,+}(x)} |\nabla_x \nabla_y G(z,y)|^2 \dd z\bigg)^{1/2} \le C e^{-c x_1}
    \end{aligned}
    \end{equation}
    and, if $x_1>2$ and $r_x > 1-\e$, \begin{equation}\label{est.DxDyG}
    \begin{aligned}
        & \bigg( \dashint_{B_{r,+}(x)} |\nabla_x\nabla_y G(z,y)|^2 dz \bigg)^{1/2} \\
        &
        + \bigg( \dashint_{B_{r,+}(x)} |\nabla_y \Pi(z,y) - \dashint_{B_{1/2,+}(x)} \nabla_y \Pi(\cdot,y)|^2 \dd z \bigg)^{1/2} \\
        & \le C \bigg( \dashint_{B_{1,+}(x)} |\nabla_x \nabla_y G(z,y)|^2 \dd z \bigg)^{1/2} \le C e^{-c x_1}.
    \end{aligned}
    \end{equation}
    Then, thanks to the fact that $\nabla_y G(x,y)$ regarded as a function of $x$ vanishes on the boundary $\Gamma^\e$, the above estimates and integrations in $x$ imply that, for $x\in D^\e$ with $r_x \le 1-\e$,
    \begin{equation}
        |\nabla_y G(x,y)| \le C\delta(x) e^{-c x_1}
    \end{equation}
    and that, for $x\in D^\e$ with $r_x > 1-\e$,
    \begin{equation}
        \bigg( \dashint_{B_{\e,+}(x)} |\nabla_y G(z,y)|^2 \dd z \bigg)^{1/2} \le C\e e^{-c x_1}.
    \end{equation}
    These prove the estimates of $\nabla_x G(x,y)$ in \eqref{est.Green-2} and \eqref{est.Grenn-4} by virtue of the symmetry $G(x,y) = G(y,x)^T$. Moreover, the pointwise estimate of $G(x,y)$ in \eqref{est.G.interior} as well as \eqref{est.G.bdry}  follows from \eqref{est.Green-2} and \eqref{est.Grenn-4} and the integrations in $x$. 

    It remains to prove \eqref{est.DyG.bdry}. Note that $\nabla_x G(x,y)$ regarded as a function of $y$ vanishes on the boundary $\Gamma^\e$. Given $y\in D^\e$ with $y_1 = 0$, we have $(1+\e') y/|y| \in \Gamma^\e$ for some $0<\e'\le \e$. Then, by the fundamental theorem of calculus and \eqref{est.DxDyG}, we see that
    \begin{equation}
    \begin{aligned}
        & \bigg( \dashint_{B_{r,+}(x)} |\nabla_x G(z,y)|^2 \dd z \bigg)^{1/2} \\
        & = \bigg( \dashint_{B_{r,+}(x)} \bigg|\int_{1-\delta(y)}^{1+\e'} \nabla_x \nabla_y G(z,\frac{yt}{|y|}) \cdot \frac{y}{|y|} \dd t \bigg|^2 \dd z \bigg)^{1/2}  \\
        & \le C(\delta(y)+\e) e^{-c x_1}.
    \end{aligned}
    \end{equation}
    This gives \eqref{est.DyG.bdry} by symmetry. This completes the proof of Theorem \ref{thm.Green}.
\end{proof}

    \begin{remark}\label{rem.thm.Green}
    The following estimates for $\Pi$ is also possible, though they will not be used in this paper:
    \begin{itemize}
    \item
    For $y\in D^\e$ with $r_y\le 1-\e$, we have 
    \begin{equation}
    |\Pi(x,y) - \Pi_{\pm}(y)| \le Ce^{-c |x_1-y_1|} \min \bigg\{ \frac{1}{|x-y|^2}, \frac{\delta(y)}{|x-y|^3} \bigg\}.
    \end{equation}
    \item 
    For $y\in D^\e$ with $r_y>1-\e$ and $|x-y|>10\e$, we have
    \begin{equation}
    \begin{aligned}
        \bigg( \dashint_{B_{\e,+}(y)} |\Pi(x,\cdot) - \Pi_{\pm}(y)|^2 \bigg)^{1/2}
        \le Ce^{-c |x_1-y_1|}  \frac{\e}{|x-y|^3}.
    \end{aligned}
    \end{equation}
    \end{itemize}
    Here $\Pi_{\pm}(y)$ satisfying $\nabla\cdot \Pi_{\pm}(y) = 0$ are two bounded functions that represent the limit of $\Pi(x,y)$ as $x_1 \to \pm \infty$. Moreover, we take $\Pi_+(x)$ if $x_1 > y_1$ and $\Pi_-(x)$ if $x_1 < y_1$. 
    
    To obtain such estimates, we need an argument with the oscillation estimates of the pressure combined with the estimates of $G(x,y)$, similar to \cite[Proposition B.5]{HPZ-21}. Roughly speaking, for each fixed $y\in D^\e$ with $r_y\le 1-\e$, we can show the oscillation of $\Pi(x,y)$ decays exponentially as $x_1 - y_1 \to \pm \infty$. This means $\Pi(x,y) \to \Pi_+(y)$ as $x_1 \to \infty$ and $\Pi(x,y) \to \Pi_-(y)$ exponentially as $x_1 \to -\infty$. Obviously, by considering the total oscillation of $\Pi(x,y)$ in $x\in D^\e\cap \{ r\le 1-\e \}$, we obtain $|\Pi_+(y) - \Pi_-(y)|$ is bounded uniformly in $y\in D^\e \cap \{ r\le 1-\e\}$. If we assume $\Pi_-(y) = 0$, then the solution given by \eqref{eq.GreenFormula} satisfies $p(x) \to 0$ as $x_1 \to -\infty$. Under this assumption, we see that $\Pi_{\pm}(y)$ are both bounded. If $y\in D^\e \cap \{ r>1-\e \}$, $\Pi_{\pm}(y)$ is bounded in the $L^2$ average sense as before. 

    To see $\nabla\cdot \Pi_{\pm}(y) = 0$, it suffices to notice that $\nabla_y\cdot \Pi(x,y) = \delta_x(y)$. This can be seen by considering the system \eqref{eq.force=f} with $f = \nabla g$ for some scalar function $g\in C_0^\infty(D^\e)$. Clearly, $(0,g)$ is the solution of \eqref{eq.force=f} with $f = \nabla g$. Thus by \eqref{eq.GreenFormula}, we have
\begin{equation}
    g(x) = \int_{D^\e} \Pi(x,y) \cdot \nabla_y g(y) \dd y = - \int_{D^\e} \nabla_y \Pi(x,y) g(y) \dd y.
\end{equation}
Since the above identity holds for any $g\in C_0^\infty(D^\e)$, we have $\nabla_y \Pi(x,y) = - \delta_x(y)$.
This implies $\nabla\cdot \Pi_{\pm}(y) = 0$ by letting $x\to \pm \infty$.
    \end{remark}

    \section{Boundary layers}\label{sec.bl}

Let us set
\begin{equation}\label{def.U0}
    U^0 = U^0(r) = (1-r^2)e_1.
\end{equation}
Then the Hagen-Poiseuille flow defined in \eqref{eq.HPflow1} is represented as $u^0=(2\phi/\pi)U^0$. Our aim in this section is to construct boundary layers which correct the trace of $U^0$ on $\Gamma^\e=\partial D^\e$. A boundary layer $u$ will be constructed as a solution in $H_{\rm uloc}^1(D^\e)^3$ of the Stokes system
\begin{equation}\label{eq.1stBL}
    \left\{
    \begin{array}{ll}
    -\Delta u+\nabla p=0&\mbox{in}\ D^\e,\\
    \nabla\cdot u=0&\mbox{in}\ D^\e,\\
    u = U^0 &\mbox{on}\ \Gamma^\e.
    \end{array}\right.
\end{equation}
Notice the smallness $|U^0|\le 2\e$ on $\Gamma^\e$,  which will be important when estimating $u$. Remark that there are no conditions on the flux in \eqref{eq.1stBL}. This leads to the nonuniqueness of solutions. To avoid such indeterminacy, in Section \ref{Normal}, we ``normalize" a boundary layer $u$ as
\begin{equation}\label{def.normalize}
    \Phi^\e[U^0 - u] = \frac{\pi}{2}.
\end{equation}
The process of this normalization is somewhat indirect as boundary layers do not satisfy the no-slip boundary condition and thus one cannot define the flux in $D^\e$ directly by zero extension method. The boundary layer satisfying \eqref{def.normalize} will be called the {\it normalized boundary layer} and denoted by $v^\e_{\rm bl}$. The normalized boundary layer $v^\e_{\rm bl}$ satisfies the flux condition $\Phi^\e[u^0 - (2\phi/\pi)v^\e_{\rm bl}] = \phi$, which will be used when we prove the main theorems in Section \ref{sec.prf}. Once the normalized boundary layer $v^\e_{\rm bl}$ is achieved, we investigate its relationship with the Navier's wall law in Section \ref{NWL} and estimate the decay of correlation in Section \ref{Corr}. The results in Section \ref{Corr} will be adopted in Section \ref{Conc} to prove the concentration inequalities for $v^\e_{\rm bl}$ under the functional inequalities in Definition \ref{def.funct.ineq}; see Proposition \ref{prop.flu.vbl}.

    \subsection{Existence and size estimates}\label{BL}

In this section, we prove constructively the existence and estimates of the solution $(u,p)$ of \eqref{eq.1stBL}. Let $\eta$ be a smooth cut-off function such that $\eta(r) = 1$ for $r>1-\e$, $\eta(r) = 0$ for $r<1-3\e$, $\eta'(r) \lesssim \e^{-1}$ and $\eta''(r) \lesssim \e^{-2}$. Letting $r = (x_2^2+x_3^2)^{1/2}$, we consider
\begin{equation}
    w(x) = u(x) - \eta(r)U(r),
    \qquad
    \pi(x) = p(x) + 4\eta(r)x_1.
\end{equation}
Then $(w,\pi)$ satisfies
\begin{equation}\label{eq.wq}
    \left\{
    \begin{array}{ll}
    -\Delta w+\nabla \pi= F(x)&\mbox{in}\ D^\e,\\
    \nabla\cdot w=0&\mbox{in}\ D^\e,\\
    w = 0 &\mbox{on}\ \Gamma^\e,
\end{array}\right.
\end{equation}
where
\begin{equation}\label{eq.Fr}
    F(x) 
    = 
    (\eta''(r) + \eta'(r) r^{-1}) U^0(r) - 4r\eta'(r)e_1 + 4x_1 \nabla_x (\eta(r)).
\end{equation}
In the following, we consider $(w,\pi)$ instead of $(u,p)$. Since $F$ is supported in $\{ 1-3\e \le r \le 1-\e \}$ and $|F(x)| \lesssim \e^{-1}(1+|x_1|)$, the Green's representation in Section \ref{Green} gives
\begin{equation}\label{eq.w.Green}
    w(x) 
    = \int_{D^\e \cap \{1-3\e < r < 1-\e\}} G(x,y) F(y) \dd y
    = \int_{\{1-3\e < r < 1-\e\}} G(x,y) F(y) \dd y.
\end{equation}
We estimate $w$ using \eqref{eq.w.Green}. Note that the thin thickness of $\{1-3\e < r < 1-\e\}$ will provide an $\e$ and the smallness of $G(x,y)$ for $y$ close to the boundary will provide another $\e$. The exponential decay of $G$ guarantees that the integral is finite even though $F$ has linear growth at infinity. Moreover, the term $4x_1 \nabla_x (\eta(r))$ in \eqref{eq.Fr} can be modified to $4(x_1 + \pi_0)\nabla_x(\eta(r))$ with any constant $\pi_0\in \R$ since this change will not contribute to $w$ as $\nabla_y\cdot G^T(x,y) = 0$, which follows from $\nabla_x\cdot G(x,y) = 0$ and the symmetry $G(x,y) = G^T(y,x)$.

Setting
\begin{equation}\label{def.uebl}
    u^\e_{\rm bl}(x) = w(x) + U^0(r) \eta(r),
    \qquad
    p^\e_{\rm bl}(x) = q(x) - 4\eta(r) x_1,
\end{equation}
we call $u^\e_{\rm bl}$ a {\it boundary layer}. By definition, a boundary layer is a solution of \eqref{eq.1stBL}. We point out that a boundary layer essentially depends on the choice of the cut-off function $\eta$. In particular, it follows that the solutions of \eqref{eq.1stBL} are not unique. Motivated by this observation, we will later ``normalize" a boundary layer to exclude such ambiguity.

The following proposition gives the size estimates of the boundary layer $(u^\e_{\rm bl},p^\e_{\rm bl})$.

    \begin{proposition}[Size estimates of boundary layers]\label{prop.ubl}
    The following hold.
    \begin{enumerate}[(1)]
    \item
    For $x\in D^\e$ with $r_x\le 1-4\e$, we have
    \begin{equation}\label{est.int-size}
        |u^\e_{\rm bl}(x)| \le C\e, \qquad |\nabla u^\e_{\rm bl}(x)| \le \frac{C \e}{1-r_x}.
    \end{equation}
    \item
    For $x\in D^\e$ with $r_x>1-4\e$, we have
    \begin{equation}\label{est.bdry-size}
        \bigg( \dashint_{B_{\e,+}(x)} |u^\e_{\rm bl}|^2 \bigg)^{1/2} \le C\e, \qquad \bigg( \dashint_{B_{\e,+}(x)} |\nabla u^\e_{\rm bl}|^2 \bigg)^{1/2} \le C.
    \end{equation}
    \end{enumerate}
    \end{proposition}
\begin{proof}
    In view of the construction of $u^\e_{\rm bl}$ in \eqref{def.uebl}, it suffices to estimate $w$ given by \eqref{eq.w.Green}. Also recall that we can modify the last term on the right of \eqref{eq.Fr} so that
    \begin{equation}\label{eq.w.Green.Fxr}
        w(x) = \int_{D^\e} G(x,y) F_x(y) \dd y,
    \end{equation}
    where
    \begin{equation}\label{eq.Fxr}
        F_x(y) 
        = 
        (\eta''(r_y) + \eta'(r_y) r_y^{-1}) U^0(r_y) - 4r_y\eta'(r_y)e_1 + 4(y_1 -x_1)\nabla_y (\eta(r_y)).
    \end{equation}
    Then we have
    \begin{equation}\label{est.Fxr}
        |F_x(y)| \lesssim \e^{-1} (1+ |x_1 - y_1|).
    \end{equation}

    \textbf{Step 1: Interior estimate of $w$.}
    Fix $x\in D^\e$ with $r_x\le 1-4\e$. Since $F_x$ is supported in the thin layer $\{1-3\e\le r\le 1-\e \}$, we use \eqref{eq.w.Green.Fxr} to obtain
    \begin{equation*}
    \begin{aligned}
        |w(x)| & \le \int_{\{ 1-3\e<r<1-\e \}} |G(x,y)| |F_x(y)| \dd y \\
        & \lesssim \int_{\{ 1-3\e<r<1-\e \}} \frac{\delta(x) \delta(y) \e^{-1} }{|x-y|^3} e^{-c|x_1-y_1|} (1+|x_1 - y_1|) \dd y \\
        & \lesssim (1-r_x) \int_{\{ 1-3\e<r<1-\e \}} \frac{1}{|x-y|^3} e^{-c|x_1-y_1|} (1+|x_1 - y_1|) \dd y,
    \end{aligned}
    \end{equation*}
    where we have used \eqref{est.Fxr} and the interior estimate of the Green's function \eqref{est.G.interior} in the second inequality and $\delta(y) = 1-r_y \approx \e$ in the last. Divide the region of integration into
    \begin{equation*}
        D^\e \cap \{ 1-3\e<r_y<1-\e \} = K_1 \cup K_2
    \end{equation*}
    with
    \begin{equation*}
    \begin{aligned}
        K_1 & = D^\e \cap \{ 1-3\e<r_y<1-\e \} \cap \{ |y_1-x_1|<1 \}, \\
        K_2 & = D^\e \cap \{ 1-3\e<r_y<1-\e \} \cap \{ |y_1-x_1|>1 \}.
    \end{aligned}
    \end{equation*}
    On $K_1$, we have
    \begin{equation*}
        \int_{K_1} \frac{1}{|x-y|^3} e^{-c|x_1-y_1|} (1+|x_1 - y_1|) \dd y \le \int_{K_1} \frac{2}{|x-y|^3}\dd y \lesssim \frac{\e}{1-r_x}. 
    \end{equation*}
    On $K_2$, we have
    \begin{equation*}
        \int_{K_2} \frac{1}{|x-y|^3} e^{-c|x_1-y_1|} (1+|x_1 - y_1|) \dd y \lesssim \e.
    \end{equation*}
    As a result, we obtain, for $x\in D^\e$ with $r_x\le 1-4\e$,
    \begin{equation}\label{est1.prf.prop.ubl}
        |w(x)| \lesssim \e.
    \end{equation}

    \textbf{Step 2: Boundary estimate of $w$.} 
    Fix $x\in D^\e$. Let $\mathbf{1}_{10\e}$ be the characteristic function of $B_{10\e,+}(x) := D^\e \cap B_{10\e}(x)$. To handle the boundary case, we split $w$ as $w = w^{(1)} + w^{(2)}$, where $w^{(1)}$ is the weak solution of
    \begin{equation}\label{eq.w1F}
        \left\{
        \begin{array}{ll}
        -\Delta w^{(1)}+\nabla q^{(1)}= F_x \mathbf{1}_{10\e}&\mbox{in}\ D^\e, \\ 
        \nabla\cdot w^{(1)}=0&\mbox{in}\ D^\e, \\ 
        w^{(1)} = 0 &\mbox{on}\ \Gamma^\e, 
    \end{array}\right.
    \end{equation}
    and $w^{(2)}$ is expressed by
    \begin{equation}\label{eq.w2.Green}
        w^{(2)}(x) = \int_{\{ 1-3\e<r<1-\e \}} G(x,y) F_x(y)(1-\mathbf{1}_{10\e}(y)) \dd y.
    \end{equation}
    Then we can use the energy estimate to $w^{(1)}$ and avoid the singularity of the Green's function in \eqref{eq.w2.Green}. Indeed, integrating the first equation of \eqref{eq.w1F} against $w^{(1)}$ yields that
    \begin{equation*}
    \begin{aligned}
        \int_{D^\e} |\nabla w^{(1)}|^2 & = \int_{D^\e\cap B_{10\e,+}(x)} F_x \cdot w^{(1)} \\
        & \lesssim \e^{-1} \int_{D^\e\cap B_{10\e,+}(x)} |w^{(1)}|  \\
        & \lesssim \e^{-1} |D^\e\cap B_{10\e,+}(x)|^{5/6} \bigg( \int_{D^\e\cap B_{10\e,+}(x)} |w^{(1)}|^6 \bigg)^{1/6} \\
        & \lesssim \e^{3/2} \bigg(\int_{D^\e} |\nabla w^{(1)}|^2 \bigg)^{1/2},
    \end{aligned}
    \end{equation*}
    where we have used \eqref{est.Fxr} in the first inequality and the Sobolev-Poincar\'{e} inequality in the last inequality. This implies that, for any $x\in D^\e$,
    \begin{equation}\label{est.Dw1}
        \bigg( \dashint_{B_{\e,+}(x)} |\nabla w^{(1)}|^2 \bigg)^{1/2} \lesssim 1.
    \end{equation}
    Hence, by the fact $w^{(1)} =0$ on $\Gamma^\e$ and the Poincar\'{e} inequality, we obtain, for $x\in D^\e$ with $r_x>1-4\e$, 
    \begin{equation*}
        \bigg( \dashint_{B_{\e,+}(x)} |w^{(1)}|^2 \bigg)^{1/2} \lesssim \e.
    \end{equation*}
    Next we estimate $w^{(2)}$. By the boundary estimate of the Green's function \eqref{est.G.bdry}, we have
    \begin{equation*}
    \begin{aligned}
        &\bigg( \dashint_{B_{\e,+}(x)} |w^{(2)}|^2 \bigg)^{1/2}\\
        & \lesssim \int_{\{ 1-3\e<r<1-\e \} \setminus B_{10\e,+}(x) } \bigg( \dashint_{B_{\e,+}(x)} |G(z,y)|^2 |F_z(y)|^2 \dd z \bigg)^{1/2} \dd y \\
        & \lesssim \e \int_{\{ 1-3\e<r<1-\e \} \setminus B_{10\e,+}(x) } \frac{ e^{-c|x_1-y_1|}(1+|x_1-y_1|)}{|x-y|^3} \dd y \\
        & \lesssim \e.
    \end{aligned}
    \end{equation*}
    Combining the estimates of $w^{(1)}$ and $w^{(2)}$, we obtain, for $x\in D^\e$ with $r_x>1-4\e$,
    \begin{equation}\label{est2.prf.prop.ubl}
        \bigg( \dashint_{B_{\e,+}(x)} |w|^2 \bigg)^{1/2} \lesssim \e.
    \end{equation}

    \textbf{Step 3: Gradient estimates of $w$ and end of the proof.} 
    We first estimate $\nabla w(x)$ for $x\in D^\e$ with $r_x\le 1-4\e$. Note that $w$ satisfies $-\Delta w + \nabla q = 0$ in $\{ r<1-3\e \}$ and that, for any $x\in D^\e$ with $r_x\le 1-4\e$, we have $B_{1-r_x-3\e,+}(x) \subset \{ r < 1-3\e \}$. Thus, by the interior Caccioppoli inequality and Lipschitz estimate, we see that
    \begin{equation}\label{est3.prf.prop.ubl}
        |\nabla w(x)| \lesssim \frac{1}{1-r_x-3\e} \bigg( \dashint_{B_{1-r_x-3\e,+}(x)} |w|^2 \bigg)^{1/2} \lesssim \frac{\e}{1-r_x-3\e} \lesssim \frac{\e}{1-r_x}.
    \end{equation}
    Here we have used the estimate of $w$ in Step 1 and Step 2, and $1-r_x-3\e \approx 1-r_x$.

    Next, for $x\in D^\e$ with $r_x > 1-4\e$, we consider the split $w = w^{(1)} + w^{(2)}$ as in Step 2. The estimate of $\nabla w^{(1)}$ is contained in \eqref{est.Dw1}. The estimate for $\nabla w^{(2)}$ is derived from 
    \begin{equation*}
        \nabla w^{(2)}(x) 
        = \int_{\{ 1-3\e<r<1-\e \}} \nabla_x G(x,y) F_x(y)(1-\chi_{10\e}(y)) \dd y.
    \end{equation*}
    Indeed, using \eqref{est.DyG.bdry} and the symmetry of $G(x,y)$, we have
    \begin{equation*}
    \begin{aligned}
        & \bigg( \dashint_{B_{\e,+}(x)} |\nabla w^{(2)}|^2 \bigg)^{1/2} \\
        & \lesssim
        \e^{-1} \int_{\{ 1-3\e<r<1-\e \} \setminus B_{10\e,+}(x) } \bigg( \dashint_{B_{\e,+}(x)} |\nabla_x G(z,y)|^2 |F_z(y)|^2 \dd z \bigg)^{1/2} \dd y\\
        & \lesssim 
        \e^{-1} \int_{\{ 1-3\e<r<1-\e \} \setminus B_{10\e,+}(x) } \frac{(1-r_y) e^{-c|x_1-y_1|} (1+|x_1-y_1|)}{|x-y|^3} \dd y\\
    & \lesssim 1.
    \end{aligned}
    \end{equation*}
    By the estimates of $\nabla w^{(1)}$ and $\nabla w^{(2)}$, we obtain for $x\in D^\e$ with $r_x>1-4\e$,
    \begin{equation}\label{est4.prf.prop.ubl}
        \bigg( \dashint_{B_{\e,+}(x)} |\nabla w|^2 \bigg)^{1/2} \lesssim 1.
    \end{equation}
    Summing up \eqref{est1.prf.prop.ubl} and \eqref{est2.prf.prop.ubl}--\eqref{est4.prf.prop.ubl}, we obtain the desired estimates \eqref{est.int-size}--\eqref{est.bdry-size}.
\end{proof}

\begin{remark}\label{rmk.L2size}    
We point out that the large-scale Lipschitz estimate is the key to derive the size estimates in Proposition \ref{prop.ubl} which particularly including the large-scale size estimate \eqref{est.bdry-size} near the rough boundary. The crucial consequence of these estimates are the global $L^2$ size estimates, i.e., 
    \begin{equation}
        \| u^\e_{\rm bl} \|_{L^2_{\rm uloc}(D^\e)} + \| \delta \nabla u^\e_{\rm bl} \|_{L^2_{\rm uloc}(D^\e)} \lesssim \e
    \end{equation}
    and
    \begin{equation}
        \| \nabla u^\e_{\rm bl} \|_{L^2_{\rm uloc}(D^\e)} \lesssim \e^{1/2}.
    \end{equation}
    To prove these estimates, we decompose $D^\e$ as $D_{1-4\e} = \{ r<1-4\e \}$ and $D^\e\setminus D_{1-4\e}$. Then \eqref{est.int-size} directly implies $\| u^\e_{\rm bl} \|_{L^2(D_{1-4\e}\cap \{ k<x_1<k+1 \})} \lesssim \e$. Moreover, by \eqref{est.bdry-size} and the Fubini's theorem, we have
    \begin{equation*}
    \begin{aligned}
        &\int_{(D^\e\setminus D_{1-4\e})\cap \{ k<x_1<k+1 \}} |u^\e_{\rm bl}(x)|^2 \dd x \\
        & \lesssim \int_{(D^\e\setminus D_{1-4\e})\cap \{ k<x_1<k+1 \}} \dashint_{B_{\e,+}(x)} |u^\e_{\rm bl}(y)|^2\dd y \dd x \\
        & \lesssim \e^2 |(D^\e\setminus D_{1-4\e})\cap \{ k<x_1<k+1 \}| \\
        & \lesssim \e^3.
        \end{aligned}
    \end{equation*}
    As a result, we have $\| u^\e_{\rm bl} \|_{L^2_{\rm uloc}(D^\e)} \lesssim \e$. The proof for the estimates of $\nabla u^\e_{\rm bl}$ is similar.
\end{remark}

    \subsection{Normalization}\label{Normal}

As we have pointed out, a boundary layer $(u^\e_{\rm bl}, p^\e_{\rm bl})$ constructed above is not a unique solution of \eqref{eq.1stBL} and depends on the choice of the cut-off function $\eta$. However, one can ``normalize" it in the following somewhat indirect manner. Fixing $\eta$ and $(u^\e_{\rm bl}, p^\e_{\rm bl})$, we set 
\begin{equation}\label{eq.HPflow}
    U^\e = U^0 - u^\e_{\rm bl},
    \qquad
    P^\e = -4x_1 - p^\e_{\rm bl}.
\end{equation}
It is straightforward to verify that $(U^\e, P^\e)$ is a nontrival solution in $H_{\rm uloc}^1(D^\e)^3$ of
\begin{equation}\label{eq.HP}
    \left\{
    \begin{array}{ll}
    -\Delta u + \nabla p = 0 &\mbox{in}\ D^\e,\\
    \nabla\cdot u = 0 &\mbox{in}\ D^\e,\\
    u = 0 &\mbox{on}\ \Gamma^\e
\end{array}\right.
\end{equation}
with nonzero flux $\Phi^\e[U^\e] \neq 0$. Therefore, by Remark \ref{rem.coro.Stokes.uniqueness}, the vector space $\mathcal{V}$ spanned by the solutions of \eqref{eq.HP} is one dimensional and $U^\e$ can be taken as a basis. Hence any boundary layer takes the form of $u^\e_{\rm bl} = U^0 - V^\e$ for some $V^\e\in \mathcal{V}$, which excludes the dependence of $u^\e_{\rm bl}$ on the cut-off functions indirectly. Then, by choosing $V^\e\in \mathcal{V}$ such that
\begin{equation}\label{eq.determineBL}
    \Phi^\e[V^\e] = \Phi^0[U^0] = \frac{\pi}{2}, 
\end{equation}
we define
\begin{equation}\label{def.NBL}
    v^\e_{\rm bl} = U^0 - V^\e
\end{equation}
and call $v^\e_{\rm bl}$ the {\it normalized boundary layer}. By Remark \ref{rem.coro.Stokes.uniqueness} again, $v^\e_{\rm bl}$ is uniquely determined. Moreover, $v^\e_{\rm bl}$ satisfies the normalized condition \eqref{def.normalize} with $u=v^\e_{\rm bl}$. 

Note that a boundary layer does not satisfy the no-slip boundary condition and thus its flux in $D^\e$ cannot be defined directly. It is emphasized that fixing a special boundary layer enables us to analyze the fluctuation of the boundary layers by preventing arbitrarily change as the domain $D^\e$ varies. In addition, we have the following.

\begin{proposition}\label{prop.nbl}
    The unique normalized boundary layer $v^\e_{\rm bl}$ determined by \eqref{eq.determineBL}--\eqref{def.NBL} satisfies the size estimates in Proposition \ref{prop.ubl}.
\end{proposition}
\begin{proof}
    Let $u^\e_{\rm bl}$ be the boundary layer constructed as in Proposition \ref{prop.ubl} and let $U^\e$ be given by \eqref{eq.HPflow}. Since $U^\e$ solves \eqref{eq.HP}, by the large-scale regularity of Stokes system with no-slip boundary condition, we know that, for $x\in D^\e$ with $r_x \le 1-4\e$, 
    \begin{equation}\label{est.Ue.int}
        |U^\e(x)| + |\nabla U^\e(x)| \lesssim 1,
    \end{equation}
    and, for $x\in D^\e$ with $r_x>1-4\e$,
    \begin{equation}\label{est.Ue.bdry}
        \bigg( \dashint_{B_{\e,+}(x)} |U^\e|^2 \bigg)^{1/2} + \bigg( \dashint_{B_{\e,+}(x)} |\nabla U^\e|^2 \bigg)^{1/2} \lesssim  1.
    \end{equation}
    Moreover, by the size estimates in Proposition \ref{prop.ubl}, it is not difficult to see that
    \begin{equation}
        \Phi^\e[U^\e]
        = \Phi^\e[U^0] + \Phi^\e[u^\e_{\rm bl}]
        = \Phi^0[U^0] + O(\e).
    \end{equation}
    Let $V^\e$ defined in \eqref{def.NBL}. Then, by \eqref{eq.determineBL}, we have
    \begin{equation}\label{est.Ve-Ue.flux}
        \Phi^\e[V^\e - U^\e] = O(\e).
    \end{equation}
    Since $V^\e - U^\e$ solves \eqref{eq.HP}, we have $V^\e - U^\e = \e_1 U^\e$ by Remark \ref{rem.coro.Stokes.uniqueness}, where $\e_1 := (2/\pi)\Phi^\e[V^\e - U^\e]$. On the other hand, observe that $V^\e - U^\e = -v^\e_{\rm bl} + u^\e_{\rm bl}$. It follows that
    \begin{equation*}
        v^\e_{\rm bl} = u^\e_{\rm bl} - \e_1 U^\e.
    \end{equation*}
    Hence, the size estimates of $v^\e_{\rm bl}$ follows from those of $u^\e_{\rm bl}$ contained in Proposition \ref{prop.ubl} and the estimates of $\e_1 U^\e$ obtained by \eqref{est.Ue.int}--\eqref{est.Ue.bdry} and \eqref{est.Ve-Ue.flux}.
\end{proof}

    \subsection{Navier's wall law}\label{NWL}

Let $v^\e_{\rm bl}$ be the normalized boundary layer uniquely determined by \eqref{eq.determineBL}--\eqref{def.NBL}. We will investigate the connection between the boundary layer $v^\e_{\rm bl}$ and the Navier's wall law. Specifically, we will determine the slip length $\lambda^\e$ in the effective problem \eqref{eq.uN}. The approach here is based on the asymptotic analysis when $\e\to0$. 

Let $\E[v^\e_{\rm bl}]$ denote the expectation of $v^\e_{\rm bl}$ restricted to $D^0$. Then $\E[v^\e_{\rm bl}]$ is a solution of the Stokes system in $D^0$ with some boundary condition on $\Gamma^0=\partial D^0$. Recall that the probability measure $\mathbb{P}$ is assumed to be stationary, i.e., $x_1$-translation and $\theta$-rotation invariant. Hence $\E[v^\e_{\rm bl}]$ is invariant under $x_1$-translations and $\theta$-rotations. Then one can check that, solving the Stokes system in $D^0$ explicitly, $\E[v^\e_{\rm bl}]$ takes the form of
\begin{equation}\label{eq.Evbl}
    \E[v^\e_{\rm bl}]
    = \e(\alpha + \beta r^2) e_1
    \quad \text{in} \mkern9mu D^0
\end{equation}
for some constants $\alpha,\beta$ with $|\alpha|+|\beta| \le C$ where $C$ is independent of the element of $\Omega^\e$. Thus $\E[v^\e_{\rm bl}]$ has a natural extension to the whole space $\R^3$, which will be denoted by $\E[v^\e_{\rm bl}]$ again. Hereafter, we will adopt this convention and use the following estimate: 
\begin{align}\label{est.Evbl}
\begin{split}
    \|\E[v^\e_{\rm bl}]\|_{L^\infty(D^\e)} \lesssim \e.
\end{split}
\end{align}  

To further reduce the free parameters and determine the structure of $\E[v^\e_{\rm bl}]$, we take advantage of the flux condition on $v^\e_{\rm bl}$ to show the following property.
\begin{proposition}\label{prop.blFlux}
    For sufficiently small $\e$, we have
    \begin{equation*}
        \int_{D^0 \cap \{ 0<x_1<1 \} } e_1\cdot\E[v^\e_{\rm bl}]
        = O(\e^2).
    \end{equation*}
\end{proposition}
\begin{proof}
    Recall that $v^\e_{\rm bl}$ is defined in \eqref{def.NBL} and $V^\e$ is chosen as the solution of the Stokes system \eqref{eq.HP} satisfying \eqref{eq.determineBL}. We consider the decomposition
    \begin{equation}\label{eq.PhieUe}
    \begin{aligned}
        \int_{D^\e \cap \{ 0<x_1<1 \} } e_1 \cdot V^\e
        = \int_{(D^\e \setminus D^0) \cap \{ 0<x_1<1 \} } e_1 \cdot V^\e + \int_{D^0 \cap \{ 0<x_1<1 \} } e_1 \cdot V^\e.
    \end{aligned}
    \end{equation}
    Using the fact that $1-r^2 = O(\e)$ for $ 1-\e < r<1+\e$ and the boundary size estimate of $v^\e_{\rm bl}$ (the first estimate in \eqref{est.bdry-size}), we see that $V^\e=U^0-v^\e_{\rm bl}$ satisfies, for all $x\in D^\e \setminus D^0$, 
    \begin{equation*}
        \bigg( \dashint_{B_{\e,+}(x)} |V^\e|^2 \bigg)^{1/2} \lesssim \e.
    \end{equation*}
    Thus, by the H\"{o}lder inequality and a simple covering argument, we have
    \begin{equation}\label{est.UeFlux.bdry}
        \bigg| \int_{(D^\e \setminus D^0) \cap \{ 0<x_1<1 \} } e_1 \cdot V^\e \bigg| \lesssim \e^2.
    \end{equation}
    On the other hand, by $V^\e=U^0-v^\e_{\rm bl}$ again, we have
    \begin{equation}
        \int_{D^0 \cap \{ 0<x_1<1 \} } e_1 \cdot V^\e = \int_{D^0 \cap \{ 0<x_1<1 \} } e_1 \cdot U^0 - \int_{D^0 \cap \{ 0<x_1<1 \} } e_1 \cdot v^\e_{\rm bl}.
    \end{equation}
    Then obviously \eqref{eq.determineBL} yields
    \begin{equation}\label{eq.Flux.u0=Ue}
        \int_{D^0 \cap \{ 0<x_1<1 \} } e_1 \cdot U^0 = \Phi^0[V^0] = \Phi^\e[V^\e].
    \end{equation}
    Combining \eqref{eq.PhieUe}--\eqref{eq.Flux.u0=Ue}, we obtain
    \begin{equation}
        \bigg| \int_{D^0 \cap \{ 0<x_1<1 \} } e_1 \cdot v^\e_{\rm bl} \bigg| \lesssim \e^2,
    \end{equation}
    which implies the desired result.
\end{proof}

Now by \eqref{eq.Evbl} and Proposition \ref{prop.blFlux}, we must have
\begin{equation}
    \E[v^\e_{\rm bl}] = \e (\alpha - 2\alpha r^2) e_1 + O(\e^2),
\end{equation}
as the leading term of order $O(\e)$ has flux zero. 
If we denote
\begin{equation}\label{def.U1}
    U^1 = U^1(r) = (1 - 2r^2) e_1,
\end{equation}
then precisely we have
\begin{equation}\label{est.Evbl-eavbar}
    \|\E[v^\e_{\rm bl}] - \e \alpha U^1 \|_{L^2_{\rm uloc}(D^0)} 
    \lesssim \e^2
\end{equation}
and
\begin{equation}\label{est.DEvbl-Deavbar}
    \|\nabla \E[v^\e_{\rm bl}] - \nabla (\e \alpha U^1) \|_{L^2_{\rm uloc}(D^0)} 
    \lesssim \e^2.
\end{equation}
We point out that the implicit parameter $\alpha$ depends on the probability space $(\Omega^\e, \mathcal{F}^\e, \PP)$.

The above asymptotic structure of $\E[v^\e_{\rm bl}]$ allows us to obtain the Navier's wall law at the formal level; see also Remark \ref{rem.thmb}. Indeed, as is explained in the introduction, $u^\e_{\rm app}$ defined by \eqref{Intro.uapp} will be our higher-order approximation in $D^\e$ with no-slip boundary condition on $\Gamma^\e$. To get the effective Navier-Stokes system and Navier's slip condition, we examine
\begin{equation}
    \E[u^\e_{\rm app}] = u^0 - \frac{2\phi}{\pi} \E[v^\e_{\rm bl}] = u^0 - \frac{2\phi}{\pi} \e \alpha U^1 + O(\e^2)
    \quad \text{in} \mkern9mu L^2_{\rm uloc}(D^0)^3.
\end{equation}
Note that $U^1$ satisfies the non-penetration condition on $\Gamma^0$ and has zero flux in $D^0$. Define
\begin{equation}\label{def.u1}
    u^1(r)
    = -\frac{2\phi}{\pi} U^1(r)
    = -\frac{2\phi}{\pi} (1 - 2r^2) e_1, 
\end{equation}
where the notation is aligned with the Hagen-Poiseuille flow in \eqref{eq.HPflow1}, and
\begin{equation}\label{def.uN}
    u^{\rm N}(r) 
    = u^0(r) + \e \alpha u^1(r)
    = \frac{2\phi}{\pi}
    \big(
    (1 - r^2) - \e \alpha (1 - 2r^2)
    \big)
    e_1.
\end{equation}
Then it is easy to check
\begin{equation}
    \E[u^\e_{\rm app}] = u^{\rm N} + O(\e^2) \quad \text{in} \mkern9mu L^2_{\rm uloc}(D^0)^3,
\end{equation}
and that $u^{\rm N}$ solves the effective problem \eqref{eq.uN} with some pressure $p^{{\rm N}}$. Using the explicit form \eqref{def.uN}, one can compute the slip lengths $\lambda^\e$ as
\begin{equation}
    \lambda^\e = \frac{\e \alpha}{1 - 2\e \alpha}.
\end{equation}
Consequently, $\lambda^\e$ is a constant determined by the probability space $(\Omega^\e,\mathcal{F}^\e,\PP)$ and independent of the flux $\phi$. Moreover, $u^{\rm N}$ is a unique bounded solution when $\phi,\e$ are small.

    \subsection{Decay of correlation}\label{Corr}

In this section, we study the correlation between the normalized boundary layers, each of which corresponds to a rough domain $D^\e$. The mapping from a rough domain or its identified boundary to the corresponding boundary layer is highly nonlinear, determined implicitly by the system \eqref{eq.1stBL} and the uniqueness condition \eqref{eq.determineBL}. Thus it is of great importance to understand how sensitive a boundary layer is to the perturbation of a boundary. In the following, we will show a sensitivity estimate that reveals the decay of correlation between the boundary layers. As a relevant result, we refer to \cite[Proposition 8]{Ger09} for the decay of correlation between boundary layers in bumpy half-spaces with random roughness. 

Let $D^\e, \widetilde{D}^\e \in \Omega^\e$ and $\Gamma^\e = \partial D^\e, \widetilde{\Gamma}^\e = \partial \widetilde{D}^\e$. Fix $z\in \Gamma^0=\partial D^0$. Assume $\Gamma^\e = \widetilde{\Gamma}^\e$ outside of the cylindrical cube $S_{\e}(z)$ defined in \eqref{def.cyl.cube}. 
Let $v^\e_{\rm bl}$ and $\tilde{v}^\e_{\rm bl}$ be the boundary layers constructed by \eqref{eq.determineBL}--\eqref{def.NBL} corresponding to $D^\e$ and $\widetilde{D}^\e$, respectively. From the construction, we see that, for $\Sigma^\e_t= D^\e \cap \{ x_1 = t\}$ not intersecting with $S_{\e}(z)$,
\begin{equation}\label{eq.v-tv.0flux}
    \int_{\Sigma^\e_{t}} e_1\cdot ({v}^\e_{\rm bl} - \tilde{v}^\e_{\rm bl}) \dd \sigma = 0.
\end{equation}

\begin{proposition}\label{prop.correlation} 
    For $x\in D^\e$ with $r_x<1-8\e$, we have
    \begin{equation}\label{est.vbl-tvbl}
        | v^\e_{\rm bl}(x) - \tilde{v}^\e_{\rm bl}(x)| 
        \le \frac{C\e^3(1-r_x) e^{-c|x_1-z_1|}}{|x-z|^3}
    \end{equation}
    and
    \begin{equation}\label{est.Dvbl-Dtvbl}
        |\nabla  v^\e_{\rm bl}(x) - \nabla \tilde{v}^\e_{\rm bl}(x)| 
        \le \frac{C\e^3 e^{-c|x_1-z_1|}}{|x-z|^3}.
    \end{equation}
    The constant $C$ is independent of $D^\e, \widetilde{D}^\e$.
\end{proposition}

\begin{proof}
    First we note that it suffices to prove the estimates for $|x-z|> 20\e$ since the size estimates in Proposition \ref{prop.ubl} directly gives the desired estimates for $|x-z|\le 20\e$.

    To simplify notation, we let $v = v^\e_{\rm bl}$ and $\tilde{v} = \tilde{v}^\e_{\rm bl}$ with the associated pressures $q$ and $\tilde{q}$, respectively.
    Let $\zeta$ be a smooth cut-off function such that $\zeta = 0$ in $S_\e(z)$, $\zeta = 1$ in $D^\e \setminus S_{2\e}(z)$ and $|\nabla \zeta| \le C\e^{-1}$. Thus the support of $\nabla \zeta$ is contained in $D^\e \cap S_{2\e}(z)$. Define
    \begin{equation}
        V = ( v - \tilde{v}) \zeta,
        \qquad
        Q = (q - \tilde{q} - q_0 )\zeta
    \end{equation}
    with arbitrary $q_0\in \R$. Then it is easy to verify that $(V,Q)$ satisfies
    \begin{equation}\label{eq.VQ}
    \left\{
        \begin{array}{ll}
        -\Delta V+\nabla Q = -\nabla (v - \tilde{v})\cdot \nabla \zeta + (q - \tilde{q} - q_0) \nabla \zeta &\mbox{in}\ D^\e,\\
        \nabla\cdot V = (v - \tilde{v})\cdot \nabla \zeta &\mbox{in}\ D^\e,\\
        V = 0 &\mbox{on}\ \Gamma^\e.
    \end{array}\right.
    \end{equation}
    Note that
    \begin{equation}\label{eq.V1}
    V = {v} - \tilde{v}
    \quad \text{in} \mkern9mu \{r<1-2\e\}.
    \end{equation} 
    To handle the nonvanishing divergence, we introduce a correction term $\hat{v}$ which solves
    \begin{equation}\label{eq.div-v}
    \left\{
        \begin{array}{ll}
        \nabla\cdot \hat{v} = ({v} - \tilde{v})\cdot \nabla \zeta &\mbox{in}\ S_{2\e}^*(z),\\
        \hat{v} = 0 &\mbox{on}\ \partial S_{2\e}^*(z).
    \end{array}\right.
    \end{equation}
    Here $S^*_{2\e}(z)$ refers to a John domain such that, according to Definition \ref{def.John2}, 
    \begin{equation}\label{def.S2east}
        D^\e \cap S_{2\e}(z) \subset S^*_{2\e}(z) \subset D^\e \cap S_{4\e}(z).
    \end{equation}
    It follows from Lemma \ref{lem.Bogovski} that \eqref{eq.div-v} has a solution $\hat{v} \in H_0^1(S^*_{2\e}(z))^3$ satisfying
    \begin{equation}\label{est.hatv}
        \| \nabla \hat{v}\|_{L^2(S^*_{2\e}(z))} \lesssim \| ({v} - \tilde{v})\cdot \nabla \zeta\|_{L^2(S^*_{2\e}(z))} \lesssim \e^{3/2}.
    \end{equation}
    We extend $\hat{v}$ by zero from $S^*_{2\e}(z)$ to the entire domain $D^\e$. Clearly, we have $\hat{v} \in H^1_0(D^\e)^3$.

    Define $(\widehat{V},\widehat{Q}) = (V - \hat{v}, Q)$. Observe that
    \begin{equation}\label{eq.V2}
    \widehat{V} = V
    \quad \text{in} \mkern9mu \{r<1-4\e\}
    \end{equation}
    and that
    \begin{equation}\label{eq.hat-VQ}
    \left\{
        \begin{array}{ll}
        -\Delta \widehat{V} + \nabla \widehat{Q} = \Delta \hat{v} -\nabla (v - \tilde{v})\cdot \nabla \zeta + (q - \tilde{q} - q_0) \nabla \zeta &\mbox{in}\ D^\e,\\
        \nabla\cdot \widehat{V} = 0 &\mbox{in}\ D^\e,\\
        \Phi^\e[\widehat{V}] = 0, & \\
        \widehat{V} = 0 &\mbox{on}\ \Gamma^\e.
    \end{array}\right.
    \end{equation}
    We point out that the flux condition $\Phi^\e[\widehat{V}] = 0$ follows from $\nabla\cdot \widehat{V} = 0$ and \eqref{eq.v-tv.0flux}.
    Then we estimate $\widehat{V}$ to prove \eqref{est.vbl-tvbl}. The Green's representation gives
    \begin{equation*}
    \begin{aligned}
        \widehat{V}(x) & = \int_{D^\e} G(x,\cdot) \big(\Delta \hat{v} -\nabla ({v} - \tilde{v})\cdot \nabla \zeta + (q - \tilde{q} - q_0) \nabla \zeta \big) \\
        & =: V^{(1)} + V^{(2)} + V^{(3)},
    \end{aligned}
    \end{equation*}
    where $V^{(i)}$ is the integral defined according to each of the three terms in the parenthesis. The above formula holds for any $q_0 \in \R$ since $\nabla_y \cdot G^T(x,y) = 0$.

    \textbf{Estimate of $V^{(1)}$.} Recall that $\hat{v}$ is supported in the John domain $S^*_{2\e}(z)$ in \eqref{def.S2east}. It follows from integration by parts and the Cauchy-Schwarz inequality that
    \begin{equation*}
    \begin{aligned}
        |V^{(1)}(x)| & = \bigg| \int_{D^\e} G(x,\cdot) \Delta \hat{v} \bigg| = \bigg| \int_{D^\e} \nabla G(x,\cdot) \cdot \nabla \hat{v} \bigg| \\
        & \le \| \nabla G(x,\cdot) \|_{L^2(S^*_{2\e}(z))} \| \nabla \hat{v} \|_{L^2(S^*_{2\e}(z))} \\
        & \lesssim \frac{\e^3 (1-r_x) e^{-c|x_1 -z_1|}}{|x-z|^3},
    \end{aligned}
    \end{equation*}
    where we have used \eqref{est.DyG.bdry} and \eqref{est.hatv} in the last inequality. Notice that \eqref{est.DyG.bdry} is actually applicable here by a simple covering argument.

    \textbf{Estimate of $V^{(2)}$.} Recall that $\nabla \zeta$ is supported in $S_{2\e}(z)$ and $|\nabla \zeta| \le C\e^{-1}$. It follows from the Cauchy-Schwarz inequality that
    \begin{equation*}
    \begin{aligned}
        |V^{(2)}(x)| 
        & = \bigg| \int_{D^\e} G(x,\cdot) \big(-\nabla ({v} - \tilde{v})\cdot \nabla \zeta \big) \bigg| \\
        & \lesssim \e^2 \bigg( \dashint_{S_{2\e}(z)} |G(x,\cdot)|^2 \bigg)^{1/2}  \bigg( \dashint_{S_{2\e}(z)} |\nabla({v} - \tilde{v})|^2 \bigg)^{1/2} \\
        & \lesssim \frac{\e^3(1-r_x) e^{-c|x_1-z_1|}}{|x-z|^3},
    \end{aligned}
    \end{equation*}
    where we have used \eqref{est.G.bdry} and \eqref{est.bdry-size} in the last inequality.

    \textbf{Estimate of $V^{(3)}$.} This estimate involves the pressure and thus we first need the connection between the pressure and velocity. In view of Lemma \ref{lem.pressure}, one has
    \begin{equation}\label{est.pi-c}
        \inf_{c_1\in \R} \bigg( \dashint_{S^*_{2\e}(z)} |q - c_1|^2 \bigg)^{1/2} \lesssim \bigg( \dashint_{S^*_{2\e}(z)} |\nabla v|^2 \bigg)^{1/2}
    \end{equation}
    and
    \begin{equation}\label{est.tpi-c}
        \inf_{c_2\in \R} \bigg( \dashint_{\tilde{S}^*_{2\e}(z)} |\tilde{q} - c_2|^2 \bigg)^{1/2} \lesssim  \bigg( \dashint_{\tilde{S}^*_{2\e}(z)} |\nabla \tilde{v}|^2 \bigg)^{1/2},
    \end{equation}
    where $\tilde{S}^*_{2\e}(z)$ is a John domain defined in a similar manner for $S^*_{2\e}(z)$ in which $D^\e$ is replaced by $\widetilde{D}^\e$. It follows from the Cauchy-Schwarz inequality that
    \begin{equation*}
    \begin{aligned}
        &|V^{(3)}(x)| 
        = \inf_{q_0\in \R} \bigg| \int_{ D^\e } G(x, \cdot) (q - \tilde{q} - q_0) \nabla \zeta \bigg| \\
        & \lesssim \e^2 \bigg( \dashint_{S_{2\e}(z)} |G(x,\cdot)|^2 \bigg)^{1/2} \inf_{q_0\in \R } \bigg( \dashint_{S_{2\e}(z)} |q - \tilde{q} - q_0|^2 \bigg)^{1/2} \\
        & \lesssim \e^2 \bigg( \dashint_{S_{2\e}(z)} |G(x,\cdot)|^2 \bigg)^{1/2} \bigg\{ \bigg(\dashint_{S^*_{2\e}(z)} |\nabla v|^2 \bigg)^{1/2} + \bigg(\dashint_{S^*_{2\e}(z)} |\nabla \tilde{v}|^2 \bigg)^{1/2} \bigg\} \\
        & \lesssim \frac{\e^3(1-r_x) e^{-c|x_1-z_1|}}{|x-z|^3},
    \end{aligned}
    \end{equation*}
    where we have used \eqref{est.pi-c}--\eqref{est.tpi-c} in the third inequality, and \eqref{est.G.bdry} and \eqref{est.bdry-size} in the last.

    Combining the above estimates of $V^1,V^2$ and $V^3$, we obtain the desired estimate \eqref{est.vbl-tvbl}. Indeed, from \eqref{eq.V1} and \eqref{eq.V2}, we see that, for $x\in D^\e$ with $r_x<1-4\e$ and $|x-z|>20\e$, 
    \begin{equation}\label{est.V}
        |{v}(x) - \tilde{v}(x)|
        = |V(x)|
        = |\widehat{V}(x)| \lesssim \frac{\e^3(1-r_x) e^{-c|x_1-z_1|}}{|x-z|^3}.
    \end{equation}

    Next we prove \eqref{est.Dvbl-Dtvbl}. Observe that, by \eqref{eq.V1}, we have
    $\nabla {v} - \nabla \tilde{v} = \nabla V$ in $\{r<1-4\e\}$.
    We use the system \eqref{eq.VQ} and the estimate of $V$ obtained in \eqref{est.V}. Since $-\Delta V + \nabla Q = 0$ and $\nabla\cdot V = 0$ in $\{ r<1-4\e \}$, by fixing $x\in \{ r<1-8\e \}$, we consider the system in $B_{(1-r_x)/2}(x) \subset \{ r<1-4\e \}$. Then the interior Lipschitz estimate shows that
    \begin{equation*}
        |\nabla {v}(x) - \nabla \tilde{v}(x)|
        =
        |\nabla V(x)| \lesssim \frac{1}{1-r_x} \bigg( \dashint_{B_{(1-r_x)/2}(x)} |V|^2 \bigg)^{1/2} \lesssim \frac{\e^3 e^{-c|x_1-z_1|}}{|x-z|^3},
    \end{equation*}
    which implies the desired derivative estimate \eqref{est.Dvbl-Dtvbl}. This ends the proof.
\end{proof}

    \subsection{Concentration of boundary layers}\label{Conc}

Combining the functional inequalities in Definition \ref{def.funct.ineq} with the deterministic estimate for the decay of correlation in the previous section, we can show the concentration inequalities of the normalized boundary layers. The resulting estimates are summarized in Proposition \ref{prop.flu.vbl} below.

We first collect the estimates that hold independently of functional inequalities. Set 
\[
    D_{0,1}' = \{ r<1-8\e \} \cap \{ 0<x_1<1 \}.
\] 
Let $v^\e_{\rm bl}$ be the boundary layer constructed in Section \ref{Normal} and consider the random variable $Y_x = v^\e_{\rm bl}(x)$. For $x\in D^\e$ with $r_x<1-8\e$, Proposition \ref{prop.correlation} implies that, for any $z\in \Gamma^0=\partial D^0$,
\begin{equation}\label{est.osc.vbl}
    \underset{D^\e|_{S_{\e}(z)}}{\osc} [Y_x] \le \frac{C\e^3(1-r_x) e^{-c|x-z|}}{|x-z|^3}.
\end{equation}
For $x\in D^\e$ with $r_x > 1-8\e$, on the other hand, we have the deterministic large-scale estimates in \eqref{est.bdry-size} for which the randomness does not play a role. Indeed, the estimates in \eqref{est.bdry-size} in Proposition \ref{prop.ubl}, combined with \eqref{est.Evbl} and a covering argument (also see Remark \ref{rmk.L2size}), give
\begin{equation}\label{est.bl.smallsize}
    \int_{D^0_{0,1}\setminus D_{0,1}'} |v^\e_{\rm bl}(x)|^2 \dd x + \int_{D^0_{0,1}\setminus D_{0,1}'} |v^\e_{\rm bl}(x) - \E[v^\e_{\rm bl}]|^2 \dd x \le C\e^3.
\end{equation}
We also consider $Y_x' = \delta(x) \nabla v^\e_{\rm bl}(x)$. By the parallel argument as above, we have
\begin{equation}\label{est.osc.Dvbl}
    \underset{D^\e|_{S_{\e}(z)}}{\osc} [Y'_x] \le \frac{C\e^3(1-r_x) e^{-c|x-z|}}{|x-z|^3}
\end{equation}
and
\begin{equation}\label{est.Dbl.smallsize}
    \int_{D_{0,1}\setminus D'_{0,1}} \delta^2| \nabla v^\e_{\rm bl} |^2 + \int_{D_{0,1}\setminus D'_{0,1}} \delta^2| \nabla v^\e_{\rm bl} - \nabla \E[v^\e_{\rm bl}]|^2 \le C\e^3.
\end{equation}
Moreover, we consider the scalar random variable
\begin{equation}\label{def.X}
    X=
    \bigg(
    \int_{D'_{0,1}} |v^\e_{\rm bl} - \E[v^\e_{\rm bl}]|^2
    \bigg)^{1/2}.
\end{equation}
The triangle inequality and \eqref{est.osc.vbl} give
\begin{equation}\label{est.X-Q1}
\begin{aligned}
\underset{D^\e|_{S_{\e}(z)}}{\osc}[X]
& \le
\bigg(
\int_{Q'_{0,1}}
\Big|
\underset{D^\e|_{S_{\e}(z)}}{\osc} [v^\e_{\rm bl}(x)]
\Big|^2 \dd x
\bigg)^{1/2} \\
& \le C \bigg( \int_{Q'_{0,1}} \frac{\e^6(1-r_x)^2 e^{-2c|x-z|}}{|x-z|^6} \dd x \bigg)^{1/2} \\
& \le
C\e^{5/2} e^{-c |z|}.
\end{aligned}
\end{equation}
The estimates \eqref{est.osc.vbl}--\eqref{est.X-Q1} are independent of the functional inequalities in Definition \ref{def.funct.ineq}.

\begin{proposition}\label{prop.flu.vbl}
    The following hold.
    \begin{enumerate}[(1)]
    \item\label{item1.prop.flu.vbl}
    Assume \eqref{ineq.SG}. Then there exists a constant $\beta_1>0$ such that we have
    \begin{equation}\label{est.vbl.SG}
        \sup_{k\in \R} \E
        \Big[
        \exp
        \Big(
        \beta_1 \e^{-3/2}
        \|v^\e_{\rm bl} - \E[v^\e_{\rm bl}]\|_{L^2(D^0_{k,k+1})}
        \Big)
        \Big]
        \le
        C
    \end{equation}
    and
    \begin{equation}\label{est.Dvbl.SG}
        \sup_{k\in \R} \E
        \Big[
        \exp
        \Big(
        \beta_1 \e^{-3/2}
        \|\delta( \nabla v^\e_{\rm bl} - \nabla \E[v^\e_{\rm bl}] )\|_{L^2(D^0_{k,k+1})}
        \Big)
        \Big]
        \le
        C.
    \end{equation}
    \item\label{item2.prop.flu.vbl}
    Assume \eqref{ineq.LSI}. Then there exists a constant $\beta_2>0$ such that we have
    \begin{equation}\label{est.vbl.LSI}
        \sup_{k\in \R} \E
        \Big[
        \exp
        \Big(
        \beta_2 \e^{-3}
        \|v^\e_{\rm bl} - \E[v^\e_{\rm bl}]\|^2_{L^2(D^0_{k,k+1})}
        \Big)
        \Big]
        \le
        C
    \end{equation}
    and
    \begin{equation}\label{est2.vbl.LSI}
        \sup_{k\in \R} \E
        \Big[
        \exp
        \Big(
        \beta_2 \e^{-3}
        \| \delta (\nabla v^\e_{\rm bl} - \nabla \E[v^\e_{\rm bl}]) \|^2_{L^2(D^0_{k,k+1})}
        \Big)
        \Big]
        \le
        C.
    \end{equation}
    \end{enumerate}
\end{proposition}
\begin{proof}
(\ref{item1.prop.flu.vbl})
Assume \eqref{ineq.SG}. We first prove \eqref{est.vbl.SG}. In view of the stationarity and the boundary estimate \eqref{est.bl.smallsize}, it suffices to restrict ourselves in $D_{0,1}'= \{ r<1-8\e \} \cap \{ 0<x_1<1 \}$. Using the assumption \eqref{ineq.SG} and the oscillation estimate \eqref{est.osc.vbl}, we have
\begin{equation}\label{est.pointwiseVAR}
\begin{aligned}
    \var[Y_x] 
    & \le \E\bigg[ C\e^{-2} \int_{\Gamma} \underset{D^\e|_{S_{\e}(z)}}{\osc} [Y_x]^2 \dd \Gamma_z \bigg] \\
     & \le C\int_{\Gamma} \frac{\e^4(1-r_x)^2 e^{-2c|x_1-z_1|}}{|x-z|^6} \dd z \\
     & \le \frac{C\e^4}{(1-r_x)^2}.
\end{aligned}
\end{equation}
Integrating it in $x$ over $D_{0,1}'$, we obtain
\begin{equation}\label{est.2moment}
    \int_{D_{0,1}'} \var[Y_x] \dd x = \E \bigg[ \int_{D_{0,1}'} | v^\e_{\rm bl} - \E[v^\e_{\rm bl}]|^2 \bigg] = \E[X^2] \le C\e^3.
\end{equation}
Hence, by \eqref{est.2moment} and the Jensen's inequality, we see that 
\begin{equation}\label{est.EX}
\E[X] \le \E[X^2]^{1/2} \le C\e^{3/2}.
\end{equation}

With the bound \eqref{est.EX} at our disposal, we apply \eqref{ineq.SG} directly to $X$ defined in \eqref{def.X}. By the estimate of $p$-moments in Proposition \ref{prop.SG.p-moment} and \eqref{est.X-Q1}, we have, for any $p \ge 1$,
\begin{align*}
\begin{split}
& \E[|X - \E[X]|^{2p}]\\ 
&\le (Cp^{2})^p \E\bigg[ \int_{\Gamma} \e^{-2} \Big(
\underset{D^\e|_{S_{\e}(z)}}{\osc} [X]
\Big)^{2p} \dd \Gamma_z \bigg] \\
&\quad + (Cp^{2})^p \E \bigg[ \bigg( \int_\Gamma \Big( \e^{-1}
\underset{D^\e|_{S_{\e}(z)}}{\osc} [X] 
\Big)^2 \dd \Gamma_z\bigg)^p \bigg]
\\ 
&\le
(C p^2 \e^3)^p, 
\end{split}
\end{align*}
resulting in
\begin{align*}
\begin{split}
\E[(\e^{-3/2} |X-\E[X]|)^{2p}]
\le
C^{p}
p^{2p}
=
(2^{-1}C^{1/2})^{2p}
(2p)^{2p}.
\end{split}
\end{align*}
After replacing $2p$ by $p$ (thus $p \ge 2$), we have, for any $s>0$,
\begin{align}\label{est.X.geomseries}
\E
\bigg[
\frac{(s\e^{-3/2} |X-\E[X]|)^{p}}{p!}
\bigg]
\le
\frac{p^p}{p!}
(2^{-1}sC^{1/2})^{p},
\end{align}
which also holds for $1\le p\le 2$ by the Jensen's inequality. Hence, by the Stirling's formula
\[
p! \ge \sqrt{2\pi p} \big(\frac{p}{e} \big)^p
\]
and by choosing $s = \beta_1>0$ small enough, we see that the right side of \eqref{est.X.geomseries} defines a convergent geometric series with integers $p \in \N$. Therefore, summing \eqref{est.X.geomseries} over $p\in \N$, we have
\begin{equation*}
    \E \big[ \exp( \beta_1 \e^{-3/2} |X - \E[X]|) \big] \le C.
\end{equation*}
This combined with \eqref{est.EX} and the triangle inequality leads to the desired estimate \eqref{est.vbl.SG}. The proof of \eqref{est.Dvbl.SG} is exactly the same as above using \eqref{est.osc.Dvbl}--\eqref{est.Dbl.smallsize} and
\begin{equation*}
    \int_{ D'_{0,1}} \var[Y'_x] \dd x = \E \bigg[ \int_{ D_{0,1}'} \delta^2| \nabla v^\e_{\rm bl} - \nabla \E[v^\e_{\rm bl}]|^2 \bigg] \le C\e^3.
\end{equation*}

(\ref{item2.prop.flu.vbl})
Assume \eqref{ineq.LSI}. Notice that $\e$ can be assumed to be small in the following, since the statement is trivial if $\e$ is not small. We first prove \eqref{est.vbl.LSI}. It follows from \eqref{est.X-Q1} that
\begin{equation}
    \sup_{D^\e} \int_{\Gamma} \Big(
\underset{D^\e|_{S_{\e}(z)}}{\osc} [X] 
\Big)^2 \dd \Gamma_z \le C_1\e^5
\end{equation}
and
\begin{equation}
    \sup_{D^\e} \sup_{z\in \Gamma}
\underset{D^\e|_{S_{\e}(z)}}{\osc} [X] \le C_1\e^{5/2}.
\end{equation}
Thus the assumptions of Proposition \ref{prop.LSI} hold with $(L')^2=C_1\e^5$ and $L=C_1\e^{5/2}$. Hence, by the concentration bound \eqref{eq:gaussdecay} with $t$ replaced by $\e^{3/2} \eta t$ for $\eta>0$, one has that 
$$
    \PP[|X - \E[X]| \geq \e^{3/2} \eta t] \leq \inf_{\lambda>0} \exp(-\e^{3/2} \eta t \lambda + C_2\e^{1/2} \lambda (e^{C_2\e^{5/2}\lambda}-1)),
$$
where $C_2:=\max\{1,C_0,C_1\}$. Thus, setting $\lambda  = C_2^{-1}\e^{-3/2} t$ leads to
$$
  \PP[\e^{-3/2}|X - \E[X]| \geq \eta  t] \leq \exp(- C_2^{-1} \eta t^2 + \e^{-1} t(e^{\e t}-1)).
$$
Since  $e^{a}-1\leq  (e-1) a$ for $0<a\leq 1$, we see that, when $t \leq \e^{-1}$, 
$$
\exp(- C_2^{-1} \eta t^2 + \e^{-1} t(e^{\e t}-1))
\le
\exp(- C_2^{-1} \eta t^2 +(e-1) t^2).
$$
Therefore, choosing $\eta>0$ so that $\eta > 2C_2(e-1)$, we have, for $t \leq \e^{-1}$, 
$$
  \PP[\e^{-3/2}|X - \E[X]| \geq \eta  t] \leq \exp\big(-\frac{\eta}{2C_2} t^2\big).
$$
On the other hand, by Proposition \ref{prop.nbl}, we have that $|X- \E[X]| \leq C\e$ almost surely for some $C>0$ independent of $\e$. As a result, we have, for $t>\e^{-1}$ with $\e \ll 1$, 
$$
  \PP[\e^{-3/2}|X - \E[X]| \geq \eta  t] = 0.
$$
Consequently, choosing $\beta_2 = \eta^{-2}$, we see that \eqref{est.vbl.LSI} follows from
$$
\begin{aligned}
    &\E\big[\exp( \beta_2 \e^{-3} |X - \E[X]|^2 )\big] \\
    &=
    \E\big[\exp( \eta^{-2}( \e^{-3/2}|X - \E[X]|)^2 )\big] \\
    & = \int_0^\infty 
    \Big(\frac{\dd}{\dd t}\exp(\eta^{-2}t^2)\Big)
    \PP\big[(\e^{-3/2}|X - \E[X]|)^2 \geq   \eta^2 t^2\big]\dd t \\
    &  \leq \int_0^{\e^{-1}} 2\eta^{-2} t   \exp(\eta^{-2}t^2)  \exp\big(-\frac{\eta}{2C_2} t^2\big) \dd t \le C. 
\end{aligned}
$$
The proof of \eqref{est2.vbl.LSI} is exactly the same as above using \eqref{est.osc.Dvbl}--\eqref{est.Dbl.smallsize}.
\end{proof}

\begin{remark}\label{rmk.optimality}
    We point out that the optimality of the order $O(\e^{3/2})$ of the convergence rates is due to the deterministic $L^2$ estimate near the boundary. Particularly, the estimates \eqref{est.bl.smallsize} and \eqref{est.Dbl.smallsize} cannot be improved in general. However, in view of the scaling of the 2D random boundary, the optimal convergence rates are expected to be $O(\e^2)$. This is actually true for interior estimates, which can be seen from \eqref{est.pointwiseVAR}. In fact, by adding the weight $\delta^\gamma$ with $\gamma>1/2$ and following the same argument, we can show
        \begin{equation*}
        \sup_{k\in \R} \E
        \Big[
        \exp
        \Big(
        \beta \e^{-4}
        \|\delta^\gamma( v^\e_{\rm bl} -\E[v^\e_{\rm bl}] )\|_{L^2(D^0_{k,k+1})}^2
        \Big)
        \Big]
        \le
        C,
    \end{equation*}
    provided \eqref{ineq.LSI} assumption (similar for \eqref{ineq.SG} assumption). The last estimate is optimal in terms of both the $O(\e^2)$ order of convergence rate and the Gaussian stochastic integrability. If $\gamma = 1/2$, the convergence rate $O(\e^2)$ should be replaced by $O(\e^2|\log \e|^{1/2})$ due to a technical reason.
\end{remark}

    \section{Proof of main theorems}\label{sec.prf}

We prove the three main theorems in this section.

    \subsection{Proof of Theorem \ref{thma}}\label{prf.thma}

In this section, we prove Theorem \ref{thma} under deterministic settings. Using the flow $V^\e$ in \eqref{eq.determineBL} and the normalized boundary layer $v^\e_{\rm bl}$ constructed by \eqref{eq.determineBL}--\eqref{def.NBL}, we define
\begin{equation}\label{def.ueapp}
    u^\e_{\rm app}
    = \frac{2\phi}{\pi} V^\e
    = u^0 - \frac{2\phi}{\pi} v^\e_{\rm bl}.
\end{equation}
Here the multiplier $2\phi/\pi$ is introduced so that $\Phi^\e[u^\e_{\rm app}] = \phi$. Also, recall that $u^0$ denotes the Hagen-Poiseuille flow in \eqref{eq.HPflow1}. Hence $u^\e_{\rm app}$ satisfies the Stokes system
\begin{equation}\label{eq.ueapp}
    \left\{
    \begin{array}{ll}
    -\Delta u^\e_{\rm app} + \nabla p^\e_{\rm app} = 0
    &\mbox{in}\ D^\e,\\
    \nabla\cdot u^\e_{\rm app} = 0
    &\mbox{in}\ D^\e,\\
    \Phi^\e[u^\e] = \phi,\\
    u^\e_{\rm app} = 0
    &\mbox{on}\ \Gamma^\e.
    \end{array}\right.
\end{equation}
Moreover, using the size estimates of $v^\e_{\rm bl}$ in Proposition \ref{prop.nbl}, we have (see Remark \ref{rmk.L2size})
\begin{equation}\label{est.vbl.L2}
    \|v^\e_{\rm bl}\|_{L^2_{{\rm uloc}}(D^\e)}
    \lesssim
    \e,
    \qquad
    \|\nabla v^\e_{\rm bl}\|_{L^2_{{\rm uloc}}(D^\e)}
    \lesssim \e^{1/2},
\end{equation}
and thus
\begin{equation}\label{est.ueapp}
    \|u^\e_{\rm app}\|_{L^2_{{\rm uloc}}(D^\e)}
    \lesssim
    |\phi|,
    \qquad
    \|\nabla u^\e_{\rm app}\|_{L^2_{{\rm uloc}}(D^\e)}
    \lesssim
    |\phi|.
\end{equation}

The following proposition shows that $u^\e_{\rm app}$ is rigorously an approximation of $u^\e$.

\begin{proposition}\label{prop.ueapp.err1}
    There exists sufficiently small $\phi_0>0$ such that for all $|\phi|<\phi_0$, the system \eqref{eq.NS} has a solution $u^\e \in H^1_{\rm uloc}(D^\e)^3$ such that
    \begin{align}\label{est.prop.ueapp.err1}
    \begin{split}
        \|\nabla u^\e - \nabla u^\e_{\rm app}\|_{L^2_{{\rm uloc}}(D^\e)}
        &\le
        C |\phi|^2 \e.
    \end{split}
    \end{align}
    The constant $C$ is independent of $\phi,\e$ and $D^\e$.
\end{proposition}

\begin{proof}
We construct a solution as a perturbation of $u^\e_{\rm app}$. Namely, we construct a vector field $w^\e$ in $H^1_{{\rm uloc}}(D^\e)^3$ such that a solution $u^\e$ of \eqref{eq.NS} is obtained from the ansatz
\begin{align}\label{ansatz.prf.prop.ueapp.err1}
    u^\e =  w^\e + u^\e_{\rm app},
    \qquad
    \|\nabla w^\e\|_{L^2_{{\rm uloc}}(D^\e)}
    \lesssim
    |\phi|^2 \e.
\end{align}

The system satisfied by $w^\e$ is as follows. By cancellation of laminar flows $U=U_1(r) e_1$:
\begin{align}\label{canc.flows}
    U \cdot\nabla U = 0
\end{align}
with $U=u^0$, the nonlinear term in \eqref{eq.NS} under the the ansatz \eqref{ansatz.prf.prop.ueapp.err1} can be written as
\begin{align*}
\begin{split}
    u^\e \cdot\nabla u^\e
    &=
    (w^\e + u^\e_{\rm app})
    \cdot\nabla
    (w^\e + u^\e_{\rm app})\\
    &=
    w^\e\cdot\nabla w^\e 
    + w^\e\cdot\nabla u^\e_{\rm app}
    + u^\e_{\rm app}\cdot\nabla w^\e\\
    &\quad
    + (u^\e_{\rm app} - u^0)\cdot\nabla u^0
    + u^0\cdot\nabla (u^\e_{\rm app} - u^0)\\
    &\quad    
    + (u^\e_{\rm app} - u^0)\cdot\nabla (u^\e_{\rm app} - u^0).
\end{split}
\end{align*}
Thus, recalling \eqref{eq.ueapp}, we find that $w^\e$ needs to solve the following nonlinear problem
\begin{equation}\label{eq1.prf.prop.ueapp.err1}
    \left\{
    \begin{array}{ll}
    -\Delta w^\e 
    + w^\e\cdot\nabla w^\e 
    + u^\e_{\rm app}\cdot\nabla w^\e
    + w^\e\cdot\nabla u^\e_{\rm app}
    + \nabla s^\e 
    = \nabla\cdot F^\e &\mbox{in}\ D^\e,\\
    \nabla\cdot w^\e = 0 &\mbox{in}\ D^\e,\\
    \Phi^\e[w^\e] = 0,\\
    w^\e = 0 &\mbox{on}\ \partial D^\e
    \end{array}\right.
\end{equation}
with $F^\e$ defined by
\begin{align*}
\begin{split}
    F^\e
    =
    -(u^\e_{\rm app} - u^0)\otimes u^0
    - u^0\otimes (u^\e_{\rm app} - u^0)
    - (u^\e_{\rm app} - u^0)\otimes (u^\e_{\rm app} - u^0).
\end{split}
\end{align*}
Using the equality
\begin{align}\label{eq.ueapp-u0}
    u^\e_{\rm app} - u^0
    = - \frac{2\phi}{\pi} v^\e_{\rm bl},
\end{align}
we estimate $F^\e$ as
\begin{align}\label{est.Fe.L2}
\begin{split}
    \|F^\e\|_{L^2_{\rm uloc}(D^\e)}
    \lesssim
    |\phi|^2 \|v^\e_{\rm bl}\|_{L^2_{\rm uloc}(D^\e)}
    + |\phi|^2 \|v^\e_{\rm bl}\otimes v^\e_{\rm bl}\|_{L^2_{\rm uloc}(D^\e)}.
\end{split}
\end{align}
One needs to be careful in estimating $v^\e_{\rm bl}\otimes v^\e_{\rm bl}$. Indeed, the classical Gagliardo-Nirenberg inequality cannot apply directly due to roughness of the boundary $\Gamma^\e$ and the nonzero value of $v^\e_{\rm bl}$ on $\Gamma^\e$. Here we take advantage of the special value of $v^\e_{\rm bl}$ on $\Gamma^\e$ to resolve this issue.

Let $\eta$ be given as in Section \ref{BL}, i.e., it is a cut-off function such that $\eta(r) = 1$ for $r>1-\e$, $\eta(r) = 0$ for $r<1-3\e$, and $\eta'(r) \lesssim \e^{-1}$. Let $h(r) = \eta(r) (1-r^2) e_1$. Then it is obvious that $v^\e_{\rm bl} - h$ vanishes on the boundary $\Gamma^\e$ and therefore the Sobolev-Poincar\'{e} inequality applies to $v^\e_{\rm bl} - h$. Precisely, for any $k\in \R$, we have
\begin{equation}\label{est.vov}
    \begin{aligned}
        &\|v^\e_{\rm bl}\otimes v^\e_{\rm bl}\|_{L^2(D^\e_{k,k+1})} \\ 
        & \lesssim \| v^\e_{\rm bl} \|_{L^4(D^\e_{k,k+1})}^2 \\
        & \lesssim \| v^\e_{\rm bl} \|_{L^2(D^\e_{k,k+1})}^{1/2} \| v^\e_{\rm bl} \|_{L^6(D^\e_{k,k+1})}^{3/2} \\
        & \lesssim \| v^\e_{\rm bl} \|_{L^2(D^\e_{k,k+1})}^{1/2} ( \| v^\e_{\rm bl} - h \|_{L^6(D^\e_{k,k+1})} + \| h \|_{L^6(D^\e_{k,k+1}))} )^{3/2} \\
        & \lesssim \| v^\e_{\rm bl} \|_{L^2(D^\e_{k,k+1})}^{1/2} ( \| \nabla(v^\e_{\rm bl} - h) \|_{L^2(D^\e_{k,k+1})} + \| h \|_{L^6(D^\e_{k,k+1})} )^{3/2} \\
        & \lesssim \e^{1/2} ( \e^{1/2} + \e )^{3/2} \lesssim \e^{5/4},
    \end{aligned}
\end{equation}
where we have used \eqref{est.vbl.L2} and the facts that $|h| \lesssim \e$ and $\| \nabla h\|_{L^2(D^\e_{k,k+1})} \lesssim \e^{1/2}$.

As a result of \eqref{est.Fe.L2}--\eqref{est.vov}, we arrive at
\begin{equation}
    \|F^\e\|_{L^2_{\rm uloc}(D^\e)}
    \lesssim
    |\phi|^2 \e.
\end{equation}
Let $\phi_0>0$ be sufficiently small. Then Theorem \ref{prop.NP.existence} implies that for all $|\phi|<\phi_0$ there is a solution of \eqref{eq1.prf.prop.ueapp.err1} with the estimate in \eqref{ansatz.prf.prop.ueapp.err1}. The assertion follows from the ansatz \eqref{ansatz.prf.prop.ueapp.err1}.
\end{proof}

The solutions of \eqref{eq.NS} in Proposition \ref{prop.ueapp.err1} are unique if an additional smallness is assumed on the flux $\phi$. This fact is stated in Theorem \ref{thm.NS.uniqueness}, which is established now as follows.

\begin{proofx}{Theorem \ref{thm.NS.uniqueness}}
The existence of solutions in $H^1_{\rm uloc}(D^\e)^3$ is obtained in Proposition \ref{prop.ueapp.err1}. We thus focus on the uniqueness of solutions. Let $u^\e$ be the solution in Proposition \ref{prop.ueapp.err1}. Suppose that $\tilde{u}^\e$ in $H^1_{\rm uloc}(D^\e)^3$ solves \eqref{eq.NS}. Then the difference $w^\e:=u^\e-\tilde{u}^\e$ satisfies 
\begin{equation}\label{eq1.prf.thm.NS.uniqueness}
    \left\{
    \begin{array}{ll}
    -\Delta w^\e + w^\e\cdot \nabla w^\e + u^\e\cdot \nabla w^\e + w^\e\cdot \nabla u^\e + \nabla \pi^\e = 0 &\mbox{in}\ D^\e,\\
    \nabla\cdot w^\e = 0&\mbox{in}\ D^\e,\\
    \Phi^\e[w^\e] = 0,\\
    w^\e = 0&\mbox{on}\ \Gamma^\e,
\end{array}\right.
\end{equation}
for some pressure $\pi^\e$. Moreover, by the Poincar\'{e} inequality and \eqref{est.ueapp}--\eqref{est.prop.ueapp.err1}, we have
\begin{align}\label{est1.prf.thm.NS.uniqueness}
\begin{split}
    \| u^\e \|_{H^1_{\rm uloc}(D^\e)}
    \lesssim
    \| \nabla u^\e_{\rm app} \|_{L^2_{\rm uloc}(D^\e)} 
    + \| \nabla u^\e - \nabla u^\e_{\rm app} \|_{L^2_{\rm uloc}(D^\e)}
    \lesssim
    |\phi|.
\end{split}
\end{align}

To prove the uniqueness, it suffices to show that, for any $k\in\R$, 
\begin{equation*}
    \int_{D^\e_{k,k+1}} |\nabla w^\e|^2=0.
\end{equation*}
Choose $\phi_0$ smaller if necessary so that $|\phi|\le\phi_0\le\tau_0$ holds with the constant $\tau_0$ in \eqref{cond.thm.SVP.w.u0<tau0} in Theorem \ref{thm.SVP.w}. Then, by \eqref{est1.prf.thm.NS.uniqueness} and a similar argument as in the proof of Corollary \ref{coro.Stokes.uniqueness}, we see that Theorem \ref{thm.SVP.w} and Remark \ref{rem.thm.SVP.w} imply $\int_{D^\e_{k,k+1}} |\nabla w^\e|^2=0$ for any $k\in\R$.
\end{proofx}

Now we prove Theorem \ref{thma}.

\begin{proofx}{Theorem \ref{thma}}
Choose $\phi_0>0$ so that all $\phi$ with $|\phi|<\phi_0$ satisfy the conditions of Theorem \ref{thm.NS.uniqueness} and Proposition \ref{prop.ueapp.err1}. Then, for such $\phi$, there is a unique solution $u^\e$ of \eqref{eq.NS} satisfying \eqref{est.prop.ueapp.err1}. We thus only need to prove the estimates \eqref{0th.err}.

The triangle inequality and Poincar\'{e} inequality combined with \eqref{est.prop.ueapp.err1} give
\begin{align*}
\begin{split}
    \|u^\e - u^0\|_{L^2_{{\rm uloc}}(D^\e)}
    &\lesssim
    \|\nabla u^\e - \nabla u^\e_{\rm app}\|_{L^2_{{\rm uloc}}(D^\e)}
    + \|u^\e_{\rm app} - u^0\|_{L^2_{{\rm uloc}}(D^\e)}\\
    &\lesssim
    |\phi|^2 \e
    + \|u^\e_{\rm app} - u^0\|_{L^2_{{\rm uloc}}(D^\e)}.
\end{split}
\end{align*}
Moreover, we have
\begin{align*}
\begin{split}
    \|\nabla u^\e - \nabla u^0\|_{L^2_{{\rm uloc}}(D^\e)}
    &\lesssim
    |\phi|^2 \e
    + \|\nabla u^\e_{\rm app} - \nabla u^0\|_{L^2_{{\rm uloc}}(D^\e)}.
\end{split}
\end{align*}
Hence the desired estimates \eqref{0th.err} follow from \eqref{eq.ueapp-u0} and \eqref{est.vbl.L2}.
\end{proofx}

    \subsection{Proof of Theorem \ref{thmb}}\label{prf.thmb}

In this section, we prove Theorem \ref{thmb} under probabilistic settings. Recall that $u^{\rm N}$ in Theorem \ref{thmb} is defined in \eqref{def.uN} or \eqref{eq.uN.sol}--\eqref{eq.u1p1}. We first prove Lemma \ref{lema}, as a consequence of

\begin{proposition}\label{prop.ueapp.err2}
    There exists sufficiently small $\phi_0>0$ such that for all $|\phi|<\phi_0$, the unique solution of \eqref{eq.NS} in Theorem \ref{thma} satisfies, for all $k\in\R$,
    \begin{align}\label{est.prop.ueapp.err2}
    \begin{split}
        &\|\nabla u^\e - \nabla u^\e_{\rm app}\|_{L^2(D^\e_{k,k+1})}\\
        &\le
        C |\phi|^2 
        \bigg(
        \e^{3/2}
        + \int_{\R}
        \|v^\e_{\rm bl} - \E[v^\e_{\rm bl}]\|_{L^2(D_{s,s+1})}  e^{-c |s-k|} \dd s
        \bigg).
    \end{split}
    \end{align}
    The constants $c,C$ are independent of $\phi,\e,k$.
\end{proposition}

\begin{proof}
Take the unique solution $u^\e$ in Theorem \ref{thma}. Then $u^\e$ is given as $u^\e=w^\e + u^\e_{\rm app}$ by the proof of Theorem \ref{thma}, where $w^\e$ is a solution of \eqref{eq1.prf.prop.ueapp.err1} in the proof of Proposition \ref{prop.ueapp.err1}. Hence we estimate $\nabla w^\e$ by using the system satisfied by $w^\e$ which is different from \eqref{eq1.prf.prop.ueapp.err1}.

Set
\begin{align*}
\begin{split}
    \E[u^\e_{\rm app}]
    =
    u^0 - \frac{2\phi}{\pi} \E[v^\e_{\rm bl}].
\end{split}
\end{align*}
Here $\E[u^\e_{\rm app}]$ is understood as the expectation of $u^\e_{\rm app}$ restricted to $D^0$ and is then extended to $\R^3$. Recall that $\E[v^\e_{\rm bl}]$ is defined on $\R^3$ according to the convention in Section \ref{NWL}. Using the cancellation \eqref{canc.flows} for $\E[u^\e_{\rm app}]$, we rewrite the nonlinear term in \eqref{eq.NS} as
\begin{align*}
\begin{split}
    u^\e \cdot\nabla u^\e
    &=
    w^\e\cdot\nabla w^\e 
    + w^\e\cdot\nabla u^\e_{\rm app}
    + u^\e_{\rm app}\cdot\nabla w^\e\\
    &\quad
    + (u^\e_{\rm app} - \E[u^\e_{\rm app}])\cdot\nabla \E[u^\e_{\rm app}]
    + \E[u^\e_{\rm app}]\cdot\nabla (u^\e_{\rm app} - \E[u^\e_{\rm app}])\\
    &\quad
    + (u^\e_{\rm app} - \E[u^\e_{\rm app}])\cdot\nabla (u^\e_{\rm app} - \E[u^\e_{\rm app}]).
\end{split}
\end{align*}
Thus $w^\e$ can be regarded as a solution of the linearized problem
\begin{equation}\label{eq1.prf.prop.ueapp.err2}
    \left\{
    \begin{array}{ll}
    -\Delta w^\e 
    + u^\e_{\rm app}\cdot\nabla w^\e
    + w^\e\cdot\nabla u^\e_{\rm app}
    + \nabla s^\e 
    = \nabla\cdot G^\e &\mbox{in}\ D^\e,\\
    \nabla\cdot w^\e = 0 &\mbox{in}\ D^\e,\\
    \Phi^\e[w^\e] = 0,\\
    w^\e = 0 &\mbox{on}\ \partial D^\e
    \end{array}\right.
\end{equation}
with $G^\e$ defined by
\begin{align*}
\begin{split}
    G^\e
    &=
    -w^\e\otimes w^\e
    - (u^\e_{\rm app} - \E[u^\e_{\rm app}])\otimes \E[u^\e_{\rm app}]
    - \E[u^\e_{\rm app}]\otimes (u^\e_{\rm app} - \E[u^\e_{\rm app}])\\
    &\quad
    - (u^\e_{\rm app} - \E[u^\e_{\rm app}])\otimes (u^\e_{\rm app} - \E[u^\e_{\rm app}]).
\end{split}
\end{align*}

Let us estimate $G^\e$. Fix $k\in\R$ and recall that the estimate of $\E[v^\e_{\rm bl}]$ in $D^\e$ is given in \eqref{est.Evbl}. Since $w^\e$ vanishes on $\Gamma^\e$, by the Sobolev-Poincar\'{e} inequality, we have
\begin{align*}
\begin{split}
    \|w^\e\otimes w^\e\|_{L^2(D^\e_{k,k+1})}
    \lesssim \| \nabla w^\e \|_{L^2(D^\e_{k,k+1})}^2 \lesssim
    |\phi|^4 \e^2,
\end{split}
\end{align*}
where the estimate in \eqref{ansatz.prf.prop.ueapp.err1} is applied. Also, a similar argument as in \eqref{est.vov} and \eqref{est.Evbl} give
\begin{align*}
\begin{split}
    &\|(v^\e_{\rm bl}-\E[v^\e_{\rm bl}])\otimes (v^\e_{\rm bl}-\E[v^\e_{\rm bl}])\|_{L^2(D^\e_{k,k+1})}\\
    &\lesssim
    \|v^\e_{\rm bl}-\E[v^\e_{\rm bl}\|_{L^2(D^\e_{k,k+1})}^{1/2} \|v^\e_{\rm bl}-\E[v^\e_{\rm bl}\|_{L^6(D^\e_{k,k+1})}^{3/2} \\
    & \lesssim \|v^\e_{\rm bl}-\E[v^\e_{\rm bl}\|_{L^2(D^\e_{k,k+1})}^{1/2} \\
    &\quad\times ( \|v^\e_{\rm bl}-h\|_{L^6(D^\e_{k,k+1})} + \| h \|_{L^6(D^\e_{k,k+1})} + \|\E[v^\e_{\rm bl}]\|_{L^6(D^\e_{k,k+1})} )^{3/2} \\
    & \lesssim \|v^\e_{\rm bl}-\E[v^\e_{\rm bl}\|_{L^2(D^\e_{k,k+1})}^{1/2} \\
    &\quad\times ( \|\nabla(v^\e_{\rm bl}-h)\|_{L^2(D^\e_{k,k+1})} + \| h \|_{L^6(D^\e_{k,k+1})} + \|\E[v^\e_{\rm bl}]\|_{L^6(D^\e_{k,k+1})} )^{3/2} \\
    & \lesssim 
    \e^{3/4} \|v^\e_{\rm bl}-\E[v^\e_{\rm bl}\|_{L^2(D^\e_{k,k+1})}^{1/2}\\
    & \lesssim \e^{3/2} + \|v^\e_{\rm bl}-\E[v^\e_{\rm bl}\|_{L^2(D^\e_{k,k+1})},
\end{split}
\end{align*}
where the Young's inequality is applied in the last line. Thus, using the equality
\begin{align}\label{eq.ueapp-ueN}
    u^\e_{\rm app} - \E[u^\e_{\rm app}]
    = - \frac{2\phi}{\pi} (v^\e_{\rm bl} - \E[v^\e_{\rm bl}]),
\end{align}
we estimate $G^\e$ as, for any $k\in\R$, 
\begin{align}\label{est1.prf.prop.ueapp.err2}
\begin{split}
    \|G^\e\|_{L^2(D^\e_{k,k+1})}
    &\lesssim
    |\phi|^2
    (
    \e^{3/2}
    + \|v^\e_{\rm bl}-\E[v^\e_{\rm bl}]\|_{L^2(D^\e_{k,k+1})}
    )\\
    &\lesssim
    |\phi|^2
    (
    \e^{3/2}
    + \|v^\e_{\rm bl}-\E[v^\e_{\rm bl}]\|_{L^2(D_{k,k+1})}
    ).
\end{split}
\end{align}

Now choose $\phi_0>0$ in Theorem \ref{thma} even smaller if necessary so that $u^\e_{\rm app}$, which is estimated in \eqref{est.ueapp}, satisfies the condition in Corollary \ref{coro.LP.ExistenceUniqueness} for all $\phi$ with $|\phi|<\phi_0$. Then, by the corollary, the solution $w^\e$ of \eqref{eq1.prf.prop.ueapp.err2} can be estimated as, for any $k\in\R$, 
\begin{align}\label{est2.prf.prop.ueapp.err2}
\begin{split}
    \|\nabla w^\e\|_{L^2(D^\e_{k,k+1})} 
    \lesssim 
    \int_\R \|G^\e\|_{L^2(D^\e_{s,s+1})} e^{-c |s-k|} \dd s,
\end{split}
\end{align}
which combined with \eqref{est1.prf.prop.ueapp.err2} leads to the desired estimate \eqref{est.prop.ueapp.err2}.
\end{proof}

\begin{proofx}{Lemma \ref{lema}}
Assume if necessary that $\phi_0$ in Theorem \ref{thma} is even smaller so that all $\phi$ with $|\phi|<\phi_0$ satisfy the condition of Proposition \ref{prop.ueapp.err2}. Notice that, in the following proof, smallness on $\e$ can be assumed since the statement of Lemma \ref{lema} is trivial if $\e$ is not small.

Fix $k\in\R$. The triangle inequality and the Poincar\'{e} inequality give
\begin{align}\label{est1.prf.lema}
\begin{split}
    &\|u^\e - u^{\rm N}\|_{L^2(D^0_{k,k+1})}\\
    &\lesssim
    \|\nabla u^\e - \nabla u^\e_{\rm app}\|_{L^2(D^\e_{k,k+1})}
    + \|u^\e_{\rm app} - u^{\rm N}\|_{L^2(D^0_{k,k+1})}.
\end{split}
\end{align}
By the definition of $u^{\rm N}$ in \eqref{def.uN} and the estimate \eqref{est.Evbl-eavbar}, we have
\begin{align}\label{est2.prf.lema}
\begin{split}
    & \|u^\e_{\rm app} - u^{\rm N}\|_{L^2(D^0_{k,k+1})} \\
    & \lesssim \|u^\e_{\rm app} - \E[u^\e_{\rm app}]\|_{L^2(D^0_{k,k+1})} + \|\E[u^\e_{\rm app}] - u^{\rm N}\|_{L^2(D^0_{k,k+1})}
    \\
    & \lesssim
    |\phi|(\|v^\e_{\rm bl} - \E[v^\e_{\rm bl}]\|_{L^2(D^0_{k,k+1})}
    + \e^2).
\end{split}
\end{align}
Hence the estimate \eqref{1st.err} follows from \eqref{est1.prf.lema}--\eqref{est2.prf.lema} and Proposition \ref{prop.ueapp.err2}. The proof of the derivative estimate \eqref{1st.err-2} proceeds in a similar manner. This completes the proof.
\end{proofx}

Now we prove Theorem \ref{thmb} by combining Lemma \ref{lema} and Proposition \ref{prop.flu.vbl}.

\begin{proofx}{Theorem \ref{thmb}}
Assume if necessary that $\phi_0$ in Theorem \ref{thma} is even smaller so that all $\phi$ with $|\phi|<\phi_0$ satisfy the condition of Lemma \ref{lema} and that $|\phi|>0$ otherwise the statement is trivial. We only prove \eqref{est.ue-uN.SG} as the proofs of the other estimates are analogous.

Fix $k\in\R$. Lemma \ref{lema} provides that
\begin{align}\label{est1.prf.thmb}
\begin{split}
    &C_0^{-1} |\phi|^{-1} \e^{-3/2} \|u^\e - u^{\rm N}\|_{L^2(D^0_{k,k+1})} \\
    &\le
    1 + \e^{-3/2} \|v^\e_{\rm bl} - \E[v^\e_{\rm bl}] \|_{L^2(D^0_{k,k+1})}\\
    &\quad
    + \int_{\R}
    \e^{-3/2} \|v^\e_{\rm bl} - \E[v^\e_{\rm bl}] \|_{L^2(D^0_{s,s+1})}  e^{-c |s-k|} \dd s.
\end{split}
\end{align}
The Cauchy-Schwarz and Jensen's inequalities applied to \eqref{est1.prf.thmb} give, for any $\beta>0$,
\begin{align}\label{est2.prf.thmb}
\begin{split}
    &\exp
    \Big(
    \frac12 \beta C_0^{-1} |\phi|^{-1} \e^{-3/2} \|u^\e - u^{\rm N}\|_{L^2(D^0_{k,k+1})}
    \Big)\\
    &\lesssim
    \exp\Big(
    \beta \e^{-3/2} \|v^\e_{\rm bl} - \E[v^\e_{\rm bl}] \|_{L^2(D^0_{k,k+1})}
    \Big)\\
    &\quad
    + \int_{\R} 
    \exp\Big(
    \beta \e^{-3/2} \|v^\e_{\rm bl} - \E[v^\e_{\rm bl}] \|_{L^2(D^0_{s,s+1})}
    \Big)
    e^{-c |s-k|} \dd s.
\end{split}
\end{align}

Now assume \eqref{ineq.SG} and let $\beta_1$ be defined in Proposition \ref{prop.flu.vbl} (\ref{item1.prop.flu.vbl}). Then we complete the proof of \eqref{est.ue-uN.SG} with $c = (1/2) \beta_1 C_0^{-1}$ by using \eqref{est2.prf.thmb} with $\beta=\beta_1$ and \eqref{est.vbl.SG}.
\end{proofx}

\begin{remark}
    Combined with the deterministic estimate of the boundary layers and Theorem \ref{thma}, we can actually state Theorem \ref{thmb} (and Theorem \ref{thmc}) in a more precise way. For example, under \eqref{ineq.SG}, we have for some $C>0$,
    \begin{equation*}
        \PP \bigg( \bigg\{ \frac{c \|u^\e - u^{{\rm N}} \|_{L^2(D^0_{k,k+1})}}{|\phi| \e^{3/2}} \ge t \bigg\} \bigg) \lesssim \left\{  
        \begin{aligned}
            & e^{- t}, & \quad \mbox{for } t<\e^{-1/2},\\
            & 0, & \quad \mbox{for } t\ge \e^{-1/2}.
        \end{aligned}
        \right.
    \end{equation*}
    Thus the exponential moment bound kicks in only for $1< t < \e^{-1/2}$.
    This type of piecewise estimate has been applied in the proof Proposition \ref{prop.flu.vbl} which plays a crucial role in deriving the Gaussian stochastic integrability under \eqref{ineq.LSI} (an oscillation version of log-Sobolev inequality).
\end{remark}

    \subsection{Proof of Theorem \ref{thmc}}\label{prf.thmc}

In this section, we prove Theorem \ref{thmc}.

\begin{proofx}{Theorem \ref{thmc}}
Let $w^\e = u^\e - u^{\rm N}$ and $\pi^\e = p^\e - p^{\rm N}$. Then $(w^\e, \pi^\e)$ satisfies
\begin{equation}
\left\{
\begin{array}{ll}
-\Delta w^\e + w^\e \cdot \nabla w^\e + u^{\rm N}\cdot \nabla w^\e + w^\e \cdot \nabla u^{\rm N} + \nabla \pi^\e = 0 &\mbox{in}\ D^0,\\
\nabla\cdot w^\e = 0&\mbox{in}\ D^0.
\end{array}\right.
\end{equation}

Let $te_1 = (t,0,0)$. Then by the oscillation estimate of pressure in Lemma \ref{lem.Pressure.osc}, we have
\begin{equation}
    \mathcal{O}_{B_{1/4}(te_1)}[\pi^\e] \le C(\| u^{\rm N} \|_{C^1(B_{1/2}(te_1))}, \| w^\e \|_{H^1(B_{1/2}(te_1))} ) \| w^\e \|_{H^1(B_{1/2}(te_1))}.
\end{equation}
By the explicit expression \eqref{eq.uN.sol}--\eqref{eq.u1p1} and the $O(1)$-bound of $\phi,\e,\alpha$ combined with Theorem \ref{thma}, 
we have $\| u^{\rm N} \|_{C^1(B_{1/2}(te_1))} \lesssim 1$ and $\| w^\e \|_{H^1(B_{1/2}(te_1))} \lesssim 1$. Thus we see that
\begin{equation}
    \mathcal{O}_{B_{1/4}(te_1)}[\pi^\e] \le C\| w^\e \|_{H^1(B_{1/2}(te_1))},
\end{equation}
for some constant $C>0$. It follows that for any $\ell \in (0,1]$,
\begin{equation}
\begin{aligned}
    |\Delta_\ell \pi^\e(te_1)| & \le \mathcal{O}_{B_{1/4}(te_1)}[\pi^\e] + \mathcal{O}_{B_{1/4}((t+1/2) e_1)}[\pi^\e] + \mathcal{O}_{B_{1/4}((t+1)e_1)}[\pi^\e] \\
    & \le C \| w^\e \|_{H^1(D^0_{1/2} \cap \{ t-1/2<x_1<t+3/2 \})} \\
    & \le C \| w^\e \|_{L^2(D^0_{t-1/2,t+3/2})} + C\| \delta \nabla w^\e \|_{L^2(D^0_{t-1/2,t+3/2})}.
\end{aligned}
\end{equation}

Now assume \eqref{ineq.SG}. Then we apply Theorem \ref{thmb} (\ref{item1.thmb}) to obtain, for $\beta>0$,
\begin{equation}
\begin{aligned}
    & \E\Big[ \exp \Big( \frac{\beta |\Delta_\ell \pi^\e(te_1)| }{ |\phi| \e^{3/2}}   \Big) \Big] \\
    & \le \E\Big[ \exp \Big( \frac{\beta C \| w^\e \|_{L^2(D^0_{t-1/2,t+3/2})} }{ |\phi| \e^{3/2}}   \Big) \exp \Big( \frac{\beta C \| \delta \nabla w^\e \|_{L^2(D^0_{t-1/2,t+3/2})} }{ |\phi| \e^{3/2}}   \Big) \Big] \\
    & \le \frac12 \E\Big[ \exp \Big( \frac{2\beta C \| w^\e \|_{L^2(D^0_{t-1/2,t+3/2})} }{ |\phi| \e^{3/2}}   \Big) + \exp \Big( \frac{2\beta C \| \delta \nabla w^\e \|_{L^2(D^0_{t-1/2,t+3/2})} }{ |\phi| \e^{3/2}}   \Big) \Big] \\
    & \lesssim 1,
\end{aligned}
\end{equation}
provided that $2\beta C<c$, which is possible if $\beta$ is chosen to be small. Finally, observe that
\begin{equation}
    \Delta_\ell \pi^\e(te_1) = \Delta_\ell p^\e(te_1) - \Delta_\ell p^{\rm N}(te_1) = \Delta_t p^\e - \frac{8\phi \ell}{\pi}(1 - 2\e \alpha).
\end{equation}
This ends the proof for the case of \eqref{ineq.SG}. The proof for the case of \eqref{ineq.LSI} is similar.
\end{proofx}

\section*{Acknowledgments} Part of this work was done when MH was visiting Morningside Center of Mathematics, Chinese Academy of Sciences. MH is grateful to the center for the kind hospitality during his stay. YL acknowledges the support from NSF award DMS-2343135. JZ is partially supported by NSFC No. 12288201 and a grant for Excellent Youth from NSFC.

\appendix

    \section{Pressure estimates}\label{sec.pressure}

    This appendix collects the estimates of pressure for the Stokes and Navier-Stokes systems.
    We begin with the Bogovski's lemma crucial for the study of the Navier-Stokes system.
    \begin{lemma}[Bogovski] \label{lem.Bogovski}
        Let $\omega \subset \R^d$ be a bounded John domain and $f \in L^2(\omega)$. Then there exists a vector-valued function $u \in H^1_0(\omega)^d$ such that $\nabla\cdot u = f$ in $\omega$ and
        \begin{equation}
            \| \nabla u \|_{L^2(\omega)} \le C\| f\|_{L^2(\omega)},
        \end{equation}
        where $C$ depends only on $d$ and the constant of John domain.
    \end{lemma}
    The Bogovski's lemma in John domains was proved in \cite{ADM06} and it was shown that the John domain is a  sufficient and almost necessary condition for the $L^p$ solvability of the divergence equation. As a duality statement, we have the following lemma.

    \begin{lemma}\label{lem.pressure}
        Let $\omega \subset \R^d$ be a bounded John domain. Then, for $p\in L^2(\omega)$, we have
        \begin{equation}
            c\| \nabla p \|_{H^{-1}(\omega)} \le \| p - \dashint_{\omega} p \|_{L^2(\omega)} \le C\| \nabla p \|_{H^{-1}(\omega)},
        \end{equation}
        where $H^{-1}(\omega)$ is the dual space of $H^1_0(\omega)$.
    \end{lemma}

    The above lemma can be used to estimate the pressure for Stokes or Navier-Stokes systems in bounded John domains.

    \begin{lemma}\label{lem.p2uF}
        Let $(u,p)$ be a weak solution of $-\Delta u + \nabla p = \nabla\cdot F$ in a bounded John domain $\omega$. Then we have
    \end{lemma}
    \begin{equation}
        \| p - \dashint_{\omega} p \|_{L^2(\omega)} \le C( \| \nabla u \|_{L^2(\omega)} + \| F \|_{L^2(\omega)}).
    \end{equation}

Next we discuss the interior regularity of the Navier-Stokes systems. Let $\omega$ be a bounded domain and $\omega' \subset\subset \omega$. If $(u,p)$ is a weak solution of the nonlinear system
\begin{equation}\label{eq.weakSol.omega}
\left\{
\begin{array}{ll}
-\Delta u + u\cdot \nabla u + u_0\cdot \nabla u + u\cdot \nabla u_0 + \nabla p = 0 &\mbox{in}\ \omega,\\
\nabla\cdot u = 0&\mbox{in}\ \omega,
\end{array}\right.
\end{equation}
with $u_0 \in C^\infty(\omega)^3$, then the classical regularity theory of the stationary Navier-Stokes system implies that $u\in C^\infty(\omega')^3$. Moreover, for any $k \ge 1$,
\begin{equation}
    \| u \|_{C^k(\omega')} + \| p - \dashint_{\omega} p \|_{C^{k-1}(\omega')} \le C(\| u_0\|_{C^k(\omega)}, \| u \|_{H^1(\omega)}) \| u \|_{H^1(\omega)},
\end{equation}
where the constnat $C(\| u_0\|_{C^k(\omega)}, \| u \|_{H^1(\omega)})$, depending also on $\omega'$ and $\omega$, is an increasing function of $\| u_0\|_{C^k(\omega)}$ and $\| u \|_{H^1(\omega)}$. In particular, we have
\begin{lemma}\label{lem.Pressure.osc}
    Let $(u,p)$ be a weak solution of \eqref{eq.weakSol.omega}. Suppose that there exists some $\tau_0 > 0$ such that $\| u_0 \|_{C^1(\omega)} \le \tau_0$ and $\| u \|_{L^2(\omega)} \le \tau_0$. Then, for any $\omega' \subset\subset \omega$, we have
    \begin{equation}
        \mathcal{O}_{\omega'} [p] 
        : = \sup_{\omega'} p - \inf_{\omega'} p
        \le C(\tau_0,\omega',\omega) \| u \|_{H^1(\omega)}.
    \end{equation}
\end{lemma}
The last lemma is useful for the pointwise estimate of the pressure, particularly in unbounded domains.

    \section{Random cylinders satisfying functional inequalities}\label{sec.ex}

This appendix presents non-trivial examples of random cylinders that satisfy the function inequalities \eqref{ineq.SG} and  \eqref{ineq.LSI}. We remark that the aim is to demonstrate the relevance of assuming these function inequalities, rather than to offer techniques for the construction of general examples that satisfy \eqref{ineq.SG} or \eqref{ineq.LSI}.

    \subsection{An infinite cylinder with Poisson boundary}

In this section, we define a random infinite cylinder with the boundary specified by a Poisson point process. We will show that the law of the cylinder satisfies both \eqref{ineq.SG} and \eqref{ineq.LSI}. To this end,   recall for $0<\e<1$ that $D_{1/\e}$ denotes an infinite cylinder with radius $1/\e$ and the boundary $\Gamma_{1/\e} := \{(x_1,x_2,x_3)\in \R^3\ |\ x_2^2 + x_3^2 = 1/\e^2 \}$. Let $\dd \Gamma_{1/\e}$ denote the surface measure of $\Gamma_{1/\e}= \partial D_{1/\e}$. Notice that $\Gamma_{1/\e}\simeq\R\times\mathbb{S}^1$. 

\paragraph{Scaled cylinder.} Our construction of a random cylinder starts with a Poisson point process defined on the rescaled circular boundary $\Gamma_{1/\e}$ with radius $1/\e$ and  is inspired by the  Poisson point process on a torus \cite[Definition 2]{GloOtt15} and on the Euclidean space \cite[Definition 2.3]{GloOtt17}. We remark that one can actually construct such a process, starting from a suitable stationary Poisson point process on $\R^2$, using \cite[Theorem 5.1]{LasPen2018book} and the quotient mapping from $\R$ to $\mathbb{S}^1$. To be more specific, given a realization $\tilde{Z}=\{\tilde{z}^{(n)}\}_{n\in\N}\subset \Gamma_{1/\e}$ of the Poisson point process with a unit intensity, we define a transformation $A : \Gamma_{1/\e} \to \R^3$ by setting
$$
A(\tilde{z}) := \begin{cases}
    (\tilde{z}_1,0,0)  + (1+\e/2)(0,\tilde{z}_2,\tilde{z}_3) & \text{ if } \tilde{z}\in S_{1}(\tilde{z}^{(n)}) \cap \Gamma_{1/\e},\\
    \tilde{z} & \text{ if } \Gamma_{1/\e} \setminus \cup_{n} S_{1}(\tilde{z}^{(n)}),
\end{cases}
$$
where  $S_{1}(\tilde{z}^{(n)})$ denotes the 3D cylindrical cube with the  center $\tilde{z}^{(n)} \in \Gamma_{1/\e}$; see \eqref{def.cyl.cube} for the definition.  Note that the transformation $A$ admits jumps  near the Poisson points $\tilde{z}^{(n)}$. Let $\tilde{D}^{1/\e}$  be the cylindrical domain with the boundary defined by the range of $A$. By construction, the cylindrical domain $\tilde{D} ^{1/\e}$ is a John domain and is  stationary with respect to $x_1$-translations and $\theta$-rotations. Let $\tilde{\Omega}_{1/\e}$ denote the set of all possible John cylindrical domains $\tilde{D} ^{1/\e}$ 
and $\tilde{\mathcal{F}}_{1/\e}$ the corresponding $\sigma$-algebra. We denote by $\tilde{\mathbb{P}}_{1/\e}$ the law of the random domain induced by the law of the  Poisson point process on $\Gamma_{1/\e}$. Given a random variable $X_{1/\e}$ on the probability space $(\tilde{\Omega}_{1/\e},\tilde{\mathcal{F}}_{1/\e},\tilde{\mathbb{P}}_{1/\e} )$, we write the expectation $\E_{1/\e}[X]$ of $X$  by $\E_{1/\e}[X]=\E_{0,1/\e}[X\circ D]$, where $\E_{0,1/\e} [\cdot]$ denotes the expectation taken against the law of random domain $D \in \tilde{\Omega}_{1/\e}$. For clarity, we omit the dependence of the random domain $D$ on $1/\e$ when it is clear from the context.

Thanks to the log-Sobolev inequality for Poisson measure \cite[Corollary 2.1]{wu2000new} (see also \cite{ane2000logarithmic}), the induced law of the random domain $D\in \tilde{\Omega}_{1/\e}$ satisfies a log-Sobolev inequality with an $\e$-independent constant $C>0$: 
\begin{align}
\begin{split}\label{eq:scaledLSI}
\Ent[X_{1/\varepsilon}]  &\le C
\int_{\Gamma_{1/\e}}
\E_{0,1/\e}
\Big[\big(
X_{1/\e}\circ D(\tilde{Z}\cup\{\tilde{z}\}) - X_{1/\e}\circ D(\tilde{Z})
\big)^2
\Big]
\dd (\Gamma_{1/\e})_{\tilde{z}}\\
&\le C
\int_{\Gamma_{1/\e}}
\E_{0,1/\e}
\Big[
\Big(
\underset{D|_{S_{1}(\tilde{z})}}{\osc} X_{1/\e}
\Big)^2
\Big]
\dd (\Gamma_{1/\e})_{\tilde{z}}.
\end{split}
\end{align}

\paragraph{The target cylinder.}  Now we turn to the construction of the target random domain and show that its law satisfies the desired log-Sobolev inequality.  Consider the scaling isomorphism  $\Phi:\Gamma_{1/\e}\mapsto\Gamma$ defined by $\Phi(x) = \e x$ for every $x\in \Gamma_{1/\e}$.    
This isomorphism transforms the Poisson point process $\tilde{Z}= \{\tilde{z}^{(n)}\}_{n\in\N}$ on $\Gamma_{1/\e}$ to the one $Z= \{z^{(n)}\}_{n\in \N} =\{\e\tilde{z}^{(n)}\}_{n\in\N}$ on $\Gamma$. This also allows us define the John domain 
$
D^{\e} := \e \tilde{D}^{1/\e} 
$
as in Definition \ref{def.John2} and the set $\Omega_{\e} := \e \tilde{\Omega}_{1/\e}$. In addition, with the isomorphism $\Phi$, given every random variable $X$ defined on $\Omega^\e$, we obtain a random variable $X_{1/\e}:= X\circ \Phi$ defined on $\tilde{\Omega}_{1/\e}$. Moreover, thanks to \eqref{eq:scaledLSI}, it holds that 
\begin{align}
\begin{split}
\Ent[X]
=
\Ent[X\circ\Phi]
&\le C
\int_{\Gamma_{1/\e}}
\E_{0,1/\e}
\Big[
\Big(
\underset{D|_{S_{1}(\tilde{z})}}{\osc} X\circ\Phi
\Big)^2
\Big]
\dd (\Gamma_{1/\e})_{\tilde{z}}.
\end{split}
\end{align}
An application of the change of variables $z=\e\tilde{z}$ and the argument of the proof of \cite[Lemma 2.4]{GloOtt17}, we have from above that 
\begin{align}
\begin{split}
\Ent[X]
&\le C
\int_{\Gamma_{1/\e}}
\E_{0,1/\e}
\Big[
\Big(
\underset{D|_{S_{1}(\tilde{z})}}{\osc} X\circ\Phi
\Big)^2
\Big]
\dd (\Gamma_{1/\e})_{\tilde{z}}\\
&= C
\e^{-2}
\int_{\Gamma}
\E
\Big[
\Big(
\underset{D^\e|_{S_{\e}(z)}}{\osc} X
\Big)^2
\Big]
\dd \Gamma_{z}.
\end{split}
\end{align}
Therefore, the law of  $D^{\e}$  satisfies \eqref{ineq.LSI}, which in turn implies the inequality \eqref{ineq.SG}.

    \subsection{An infinite cylinder with Bernoulli radii} 

In this section, we define the random infinite cylinders with boundaries specified by Bernoulli process. Assume for simplicity that $M=1/\e \in  \N$.  Let $b: \R \times [0,2\pi] \to \R$ be a random binary function defined by
$$
b_{\bxi}(x_1,\theta)=  \sum_{i\in \Z, 0\le j\le M-1} \bxi_{ij} \mathbf{1}_{[\e i, \e(i+1))}(x_1) \mathbf{1}_{[\e j 2\pi, \e(j+1)2\pi)}(\theta) ,
$$
where  $\bxi =\{\xi_{ij}\}\sim \PP_{\bxi} := \text{Ber}(1/2)^{\otimes\infty}$ is a sequence of independent Bernoulli random variables with equal probability, i.e.  $\PP(\xi_{ij} = 0)= \PP(\xi_{ij} = 1)=1/2$.

We define a random infinite cylinder with the binary radius function $1 + \e b_{\bxi}(x_1,\theta)$, namely we set
$$
D^\e = D^\e(\bxi) := \{(x_1,r,\theta) \in \R^3 \ ~|~ x_1\in \R, \mkern9mu \theta \in [0,2\pi], \mkern9mu 0\leq r\leq 1+\e b_{\bxi}(x_1,\theta)\}. 
$$
By the geometrical construction, it is not hard to see that $D^\e$ is a John domain in Definition \ref{def.John2} with constants independent of $\bxi$, whose boundary is given by
$$
 \Gamma^{\e}:= \{(x_1,r,\theta) \in \R^3 \ ~|~ x_1\in \R, \mkern9mu \theta \in [0,2\pi], \mkern9mu r= 1+\e b_{\bxi}(x_1,\theta)\}.
$$
Let $\Gamma = \{ r = 1 \}$ and $\Gamma_{ij} = \{(x_1, r, \theta ) ~|~r= 1, x_1\in [\e i, \e(i+1)), \theta \in [\e2\pi j, \e2\pi (j+1)   \}$ for $i\in \Z, 0\le j\le M-1$. Clearly, $\{ \Gamma_{ij} \}$ are mutually disjoint and $\Gamma = \cup_{ij} \Gamma_{ij}$.

Let $\Omega^{\e}$ be the sample space of all the random infinite cylindrical domains $D^\e$ and let $\PP$  be the induced  probability measure  on the sample space $\Omega^\e$  by the Bernoulli measure $\PP_{\bxi}$. It is obvious that for any $i\in \Z$ and $0\leq j\leq M-1$, the measure $\PP$ is stationary with respect to translation along $x_1$-direction by increments of $i\e$ and rotation in $\theta$ by   angles of $2\pi j\e$. Further, any realization of $D^\e$  is a John domain. By slightly abusing the terminology, we use $\E$ to denote the expectation against both the law $\PP$ of the random cylinder and the Bernoulli measure $ \PP_{\bxi}$.

Let $X$ be a random variable defined on $\Omega^\e$. Define $F: \{0,1\}^{\infty} \rightarrow \R$ by setting $F(\bxi) := X\circ D^\e(\bxi)$. Then it follows from the log-Sobolev inequality \cite[Theorem 1]{bobkov1998modified} of the Bernoulli measure $\PP_{\bxi}$ that 
\begin{equation}\label{eq:entber}
    \Ent[X] = \Ent[F]
    \leq \frac{1}{4} \E\sum_{ij} |F(\bxi+\mathbf{e}_{ij}) - F(\bxi)|^2,
\end{equation}
where $\{\mathbf{e}_{ij}\}$ is the canonical basis of $\R^{\infty}$ and the addition inside $F$  is modulo $2$. Moreover, by the definition of $F$, we see that
$$\begin{aligned}
|F(\bxi+\mathbf{e}_{ij}) - F(\bxi)| & = |X \circ D^\e(\bxi+\mathbf{e}_{ij}) -  X \circ D^\e(\bxi)|\\
& \leq \inf_{z\in \Gamma_{ij}} \underset{D^\e|_{S_{2\pi \e}(z)}}{\osc} X \\
& \leq \frac{1}{|\Gamma_{ij}|} \int_{\Gamma_{ij}} \underset{D^\e|_{S_{2\pi \e}(z)}}{\osc} X \dd \Gamma_z.
\end{aligned}
$$
Summing the inequality above over $i$ and $j$ and plugging the resulting estimate into  \eqref{eq:entber} leads to 
\begin{equation}
    \Ent[X] \le \frac{1}{8\pi \e^2} \E \bigg[ \int_{\Gamma} \Big( \underset{D^\e|_{S_{2\pi \e}(z)}}{\osc} X \Big)^2 \dd \Gamma_z \bigg],
\end{equation}
where we have used the fact $|\Gamma_{ij}| = 2\pi \e^2$. This is the desired log-Sobolev inequality, and the spectral gap inequality follows as a direct consequence.

\begin{remark}\label{rem.Ber}
    The probability space $(\Omega^\e, \mathcal{F}^\e, \PP)$ constructed above by Bernoulli process is not stationary with respect to all the $x_1$-translations and $\theta$-rotations. Indeed, it is stationary only with respect to those $x_1$-translations and $\theta$-rotations in particular steps, namely,
    \begin{equation}
        \PP  \circ R_{\theta_j} = \PP  \circ T_{t_i} =  \PP ,
    \end{equation}
    for $\theta_j = 2\pi \e j$ and $t_i = \e i$ with $i,j\in \Z$. Nevertheless, our main results still hold for this case. The only difference in the proof is that $\E[v^\e_{\rm bl}]$ is not a simple laminar flow, but a periodic flow with period $2\pi \e$ in $\theta$ and $\e$ in $x_1$. Then we only need to take another step to show the deterministic convergence rate of this periodic flow to a real laminar flow, which is much easier by using the Green's representation in $D^0$. The details will be omitted.
\end{remark}

    \section{Concentration bounds under functional inequalities}

This appendix collects the inequalities utilized in Section \ref{Conc} to establish the concentration of boundary layers based on functional inequalities.

    \subsection{Moment bound under the spectral gap inequality}

The following proposition is used in the proof of Proposition \ref{prop.flu.vbl} (\ref{item1.prop.flu.vbl}).

\begin{proposition}\label{prop.SG.p-moment}
    Assume that $(\Omega^\e, \mathcal{F}^\e, \PP)$ satisfies \eqref{ineq.SG}. Let $X$ be a scalar random variable. Then we have, for any $p \ge 1$, 
    \begin{equation}
    \begin{aligned}
        \E[|X - \E[X]|^{2p}] 
        &\le (Cp)^{2p} \E\bigg[ \int_{\Gamma} \e^{-2} \Big(\underset{D^\e|_{S_{\e}(z)}}{\osc} [X]\Big)^{2p} \dd \Gamma_z \bigg] \\
        &\quad
        + (Cp)^{2p} \E \bigg[ \bigg( \int_\Gamma \Big(\e^{-1} \underset{D^\e|_{S_{\e}(z)}}{\osc} [X]\Big)^2 \dd \Gamma_z \bigg)^p \bigg].
    \end{aligned}
    \end{equation}
    The constant $C$ is independent of $p$ and $X$.
\end{proposition}
\begin{proof}
     We follow an argument in \cite[Section 3.1]{DueGlo20-1}. For simplicity, we let $\osc[\,\cdot\,]$ denote $\osc_{D^\e|_{S_{\e}(z)}} [\,\cdot\,]$ in this proof. Without loss of generality, assume $\E[X] = 0$. Applying \eqref{ineq.SG} to the random variable $Y = |X|^p = |X - \E[X]|^p$, we have
    \begin{equation}\label{est.EX-2p-1}
        \E[|X|^{2p}] \le \E[|X|^p]^2 + C\E \bigg[ \int_\Gamma \e^{-2} \osc[|X|^p]^2 \dd \Gamma_z \bigg].
    \end{equation}
    For $1\le p\le 2$, the Jensen's inequality gives
    \begin{equation}
        \E[|X|^p]^2 \le \E[X^2]^p \le C^p \E \bigg[ \bigg( \int_\Gamma (\e^{-1} \osc[X])^2  \dd \Gamma_z\bigg)^p \bigg],
    \end{equation}
    where we have also used \eqref{ineq.SG} and the H\"{o}lder  inequality in the second inequality. For $p>2$, the Young's inequality gives, for any $\delta>0$,
    \begin{equation}
        \E[|X|^p]^2 \le \frac{p-2}{p-1} \delta \E[|X|^{2p}] + \frac{1}{p-1} \delta^{2-p} \E[|X|^2]^p.
    \end{equation}
    By choosing $\delta$ small (say, $\delta = 1/2$), the first term of the right-hand side of the above inequality can be absorbed to the left-hand side of \eqref{est.EX-2p-1}. It follows that, for any $p\ge 1$,
    \begin{equation}\label{est.Ex-2p-4}
        \E[|X|^{2p}] \le C^p \E \bigg[ \bigg( \int_\Gamma (\e^{-1} \osc[X])^2 \dd \Gamma_z\bigg)^p \bigg] + C\E \bigg[ \int_\Gamma \e^{-2} \osc[|X|^p]^2 \dd \Gamma_z \bigg].
    \end{equation}
    Thus it suffices to estimate the last term in the above inequality.

    Using the inequality $|a^p - b^p| \le p|a-b|(|a|^{p-1} + |b|^{p-1})$ for all $a,b\in\R$, we have
    \begin{equation}
        \osc[|X|^p] \le 2p \osc[X] \sup |X|^{p-1},
    \end{equation}
    where
    \begin{equation}
        \sup |X|: = \sup_{D^\e|_{S_{\e}(z)}} |X| := \sup \{|X(\tilde{D}^\e)~|~ \tilde{D}^\e \in\Omega^\e, \mkern9mu \tilde{D}^\e = D^\e \mbox{ in } \R^3\setminus S_{\e}(z)\} .
    \end{equation}
    It follows that
    \begin{equation}
        \osc[|X|^p] \le 2p \osc[X] (\sup |X|) ^{p-1} \le 2p \osc[X] ( |X| + \osc[X])^{p-1}.
    \end{equation}
    Then by using the inequality $(a+b)^{p-1} \le 2a^{p-1} + (Cp)^{p} b^{p-1}$ for all $a,b\ge0$, we have
    \begin{equation}
        \osc[|X|^p] \le 4p |X|^{p-1} \osc[X] + (Cp)^p \osc[X]^p.
    \end{equation}
    Thus, we have
    \begin{equation}\label{est.oscXp-1}
    \begin{aligned}
        \int_\Gamma \e^{-2} \osc[|X|^p]^2 \dd \Gamma_z &\le (8p)^2 |X|^{2(p-1)} \int_\Gamma (\e^{-1}\osc[X])^2 \dd \Gamma_z\\
        &\quad
        + 2(Cp)^{2p} \int_\Gamma \e^{-2} \osc[X]^{2p} \dd \Gamma_z.
    \end{aligned}
    \end{equation}
    Then the H\"{o}lder and Young's inequalities give
    \begin{equation}
    \begin{aligned}\label{est.oscXp-2}
        & \E \bigg[ (8p)^2 |X|^{2(p-1)} \int_\Gamma (\e^{-1}\osc[X])^2 \Gamma_z \bigg] \\
        & \le \frac{p-1}{p} \delta^{p/(p-1)} \E [|X|^{2p}] + \delta^{-p}(8p)^{2p} \E \bigg[ \bigg(  \int_\Gamma (\e^{-1} \osc[X])^2 \Gamma_z \bigg)^{p} \bigg].
    \end{aligned}
    \end{equation}
    Inserting \eqref{est.oscXp-1}--\eqref{est.oscXp-2} into \eqref{est.Ex-2p-4}, and by taking $\delta$ small, we see that $\E [|X|^{2p}]$ on the right-hand side can be absorbed to the left-hand side, which gives the desired estimate.
\end{proof}

    \subsection{Exponential moment bound under the  logarithmic Sobolev inequality}

The following proposition is used in the proof of Proposition \ref{prop.flu.vbl} (\ref{item2.prop.flu.vbl}).

\begin{proposition}\label{prop.LSI}
    Assume that $(\Omega^\e, \mathcal{F}^\e, \PP)$ satisfies \eqref{ineq.LSI}. Moreover, assume that a scalar random variable $X$ has $L^2$-bounded local oscillation, i.e. there exists $L^\prime >0$ such that 
\begin{equation}\label{eq:avrosc}
    \sup_{D^\e\in\Omega^\e}  \int_{\Gamma}
\Big(
\underset{D^\e|_{S_{\e}(z)}}{\osc} [X]
\Big)^2 
\dd \Gamma_z \leq (L^\prime)^2
\end{equation} 
and that $X$ has uniformly bounded local oscillation, i.e. there exists $L>0$ such that 
\begin{equation}\label{eq:bdosc}
    \sup_{D^\e\in\Omega^\e}\sup_{z\in \Gamma} \underset{D^\e|_{S_{\e}(z)}}{\osc} [X] \leq L.
\end{equation}
Then we have, for any $\lambda > 0$, 
    \begin{equation}\label{eq:expmoment}
        \E \big[\exp(\lambda (X-\E[X]))\big] \leq \exp\Big(\frac{C_0\e^{-2} (L^{\prime})^2}{L}\lambda (e^{L\lambda}-1)\Big) 
    \end{equation}
and, for any $t > 0$,
\begin{equation}\label{eq:gaussdecay}
    \PP [|X-\E[X]| \geq t]
  \leq 2 \inf_{\lambda>0}  \exp\Big(-\lambda t + \frac{C_0\e^{-2} (L^{\prime})^2}{L} \lambda (e^{L\lambda}-1)\Big).
\end{equation}
\end{proposition}
\begin{proof}
    We follow an argument in \cite[Section 3.2]{DueGlo20-1}, which is essentially based on the Herbst's argument. For simplicity, we let $\osc[\,\cdot\,]$ denote $\osc_{D^\e|_{S_{\e}(z)}} [\,\cdot\,]$ in this proof. First, using \eqref{ineq.LSI}, we obtain for the random variable $e^{\lambda X/2}$: 
\begin{equation}
\label{eq:bdlsi}
    \Ent[e^{\lambda X}] = \E [\lambda X e^{\lambda X}] - \E [e^{\lambda X}] \cdot  \log \E [e^{\lambda X}] \leq C \E \bigg[\int_{\Gamma} \e^{-2} (\osc[e^{\lambda X/2}])^2 \dd \Gamma_z\bigg]. 
\end{equation}
    By using the inequality $|e^a - e^b| \leq (e^a + e^b) |a-b|$ for all $a,b\in \R$, we have
$$\begin{aligned}(\osc [e^{\lambda X/2}])^2 
&\leq 4\lambda^2 e^{\lambda X} e^{\lambda{\osc}[X]} (\osc[X])^2.
\end{aligned}
$$
Thanks to the assumptions \eqref{eq:avrosc}--\eqref{eq:bdosc}, we then have
\[
\int_{\Gamma} \e^{-2} (\osc[e^{\lambda X/2}])^2 \dd \Gamma_z \leq C\e^{-2} (L')^2 \lambda^2 e^{\lambda L}.
\]
Define $\phi(\lambda) = \E[e^{\lambda X}]$ for $\lambda\ge0$. Then inserting the last display into \eqref{eq:bdlsi} leads to 
$$
\lambda \phi^\prime(\lambda) - \phi(\lambda) \log \phi(\lambda) \leq C\e^{-2} (L')^2 \lambda^2 e^{\lambda L}.
$$
Let $\psi(\lambda) =   \log \phi(\lambda) $. The estimate above implies that 
 $$
\Big(\frac{\psi(\lambda) }{\lambda}\Big)^{\prime} \leq C\e^{-2} (L')^2 e^{\lambda L}.
 $$
 Integrating both sides on $[a,\lambda]$ with $0<a\le \lambda$ shows that
 $$
 \frac{\psi(\lambda)}{\lambda} - \frac{\psi(a) - \psi(0)}{a}
 \le \frac{C\e^{-2} (L')^2}{L} (e^{\lambda L}- e^{a L}),
 $$
where $\psi(0)=0$ is used. Since $\psi'(0)=\phi'(0)/\phi(0)=\E[X]$, taking the limit $a\to0$ gives 
$$
 \log \E \big[\exp(\lambda (X-\E[X]))\big]
  =
 \psi(\lambda)
 - \E[\lambda X]
 \le \frac{C\e^{-2} (L')^2}{L} \lambda (e^{\lambda L}- 1),
 $$
which proves the first assertion \eqref{eq:expmoment}. Moreover, the Markov's inequality leads to
$$\begin{aligned}
    \PP [X-\E[X] \geq t]
 & \leq \inf_{\lambda>0} e^{-\lambda t} 
 \E \big[\exp(\lambda (X-\E[X]))\big]\\
 & \leq \inf_{\lambda>0}  \exp\Big(-\lambda t + \frac{C\e^{-2} (L')^2}{L} \lambda (e^{\lambda L}- 1)\Big).
\end{aligned} 
$$
Notice that the assumptions \eqref{eq:avrosc}--\eqref{eq:bdosc} are invariant under $X\mapsto-X$. Hence the same argument as above is valid for $-X$, which implies the second assertion \eqref{eq:gaussdecay}.
\end{proof}

\addcontentsline{toc}{section}{References}
\bibliography{Mybib}
\bibliographystyle{alpha}

\end{document}